\documentclass[reqno]{amsart}
\usepackage{amssymb, amsthm, amsmath, amsfonts,mathtools}
\usepackage{array}
\usepackage{bbm}
\usepackage{MnSymbol}
\usepackage{enumerate}
\usepackage[only,llbracket,rrbracket]{stmaryrd}

\usepackage[hidelinks]{hyperref}
\usepackage[numbers,square]{natbib}
\usepackage{todonotes} %for comments

\allowdisplaybreaks

\numberwithin{equation}{section}

\newcommand{\He}{\mathop{\mathrm{He}}\nolimits}

\newcommand{\ii}{{\rm{i}}}

\newcommand{\bD}{\mathbb{D}}

\newcommand{\cD}{\mathcal{D}}

\newcommand{\cM}{\mathcal{M}}

\def\dint{\textup{d}}

\newcommand{\E}{\mathbb E}

\newcommand{\R}{\mathbb{R}}
\newcommand{\N}{\mathbb{N}}

\newcommand{\C}{\mathbb{C}}
\newcommand{\Z}{\mathbb{Z}}

\renewcommand{\P}{\mathbb{P}}

\renewcommand{\Re}{\operatorname{Re}}
\renewcommand{\Im}{\operatorname{Im}}

\newcommand{\eps}{\varepsilon}

\newcommand{\toprobab}{\overset{\P}{\underset{n\to\infty}\longrightarrow}}

\newcommand{\toasj}{\overset{a.s.}{\underset{j\to\infty}\longrightarrow}}

\newcommand{\ton}{\overset{}{\underset{n\to\infty}\longrightarrow}}

\newcommand{\ind}{\mathbbm{1}}

\newcommand{\dd}{{\rm d}}
\newcommand{\eee}{{\rm e}}

\theoremstyle{plain}
\newtheorem{theorem}{Theorem}[section]
\newtheorem{lemma}[theorem]{Lemma}
\newtheorem{corollary}[theorem]{Corollary}
\newtheorem{proposition}[theorem]{Proposition}
\newtheorem{conjecture}[theorem]{Conjecture}

\theoremstyle{definition}
\newtheorem{definition}[theorem]{Definition}
\newtheorem{example}[theorem]{Example}

\newtheorem{openproblem}[theorem]{Open Problem}

\theoremstyle{remark}
\newtheorem{remark}[theorem]{Remark}

\begin{document}

\author{Brian C. Hall}
\address{Brian C. Hall: University of Notre Dame, Notre Dame, IN 46556, USA
}
\email{bhall@nd.edu}

\author{Ching-Wei Ho}
\address{Ching-Wei Ho: Institute of Mathematics, Academia Sinica, Taipei 10617, Taiwan
}
\email{chwho@gate.sinica.edu.tw}

\author{Jonas Jalowy}
\address{Jonas Jalowy: Institut f\"ur Mathematische Stochastik,
Westf\"alische Wilhelms-Universit\"at\linebreak M\"unster,
Orl\'eans-Ring 10,
48149 M\"unster, Germany}
\email{jjalowy@uni-muenster.de}

\author{Zakhar Kabluchko}
\address{Zakhar Kabluchko: Institut f\"ur Mathematische Stochastik,
Westf\"alische Wilhelms-Universit\"at M\"unster,
Orl\'eans-Ring 10,
48149 M\"unster, Germany}
\email{zakhar.kabluchko@uni-muenster.de}

\title
{
Zeros of random polynomials undergoing the heat flow
}

\date{\today}



\begin{abstract}
We investigate the evolution of the empirical distribution of the complex roots of high-degree random polynomials,  when the polynomial undergoes the heat flow. In one prominent example of Weyl polynomials, the limiting zero distribution evolves from the circular law into the elliptic law until it collapses to the Wigner semicircle law, as was recently conjectured for characteristic polynomials of random matrices by Hall and Ho, 2022. Moreover, for a general family of random polynomials with independent coefficients and isotropic limiting distribution of zeros, we determine the zero distribution of the heat-evolved polynomials in terms of its logarithmic potential. Furthermore, we explicitly identify two critical time thresholds, at which singularities develop and at which the limiting distribution collapses to the semicircle law. We completely characterize the limiting root distribution of the heat-evolved polynomials before singularities develop as the push-forward of the initial distribution under a transport map. Finally, we discuss the results from the perspectives of partial differential equations (in particular Hamilton--Jacobi equation and Burgers' equation), optimal transport, and free probability. The theory is accompanied by explicit examples, simulations, and conjectures.
\end{abstract}

\maketitle
\tableofcontents

\section{Introduction}\label{sec:intro}
This paper is concerned with the behavior of zeros of high-degree random polynomials undergoing the heat flow.
In order to motivate the study, we shall begin with the counterpart of our problem in the setting of random matrices, the heat-flow conjecture introduced by Hall and Ho in~\cite{hallho}.

Let us recall two fundamental results from the world of random matrices, the circular law and semicircle law. Consider a non-degenerate complex-valued random variable $\xi$ with zero mean and unit variance and let $X_n$ denote an $n\times n$-random matrix whose entries are independent realizations of $\xi$.
The \emph{circular law} was introduced by Ginibre \cite{ginibre} for Gaussian entries and by Girko \cite{girko_circular_law} for general entries, and then established in increasing generality by many authors, including Bai \cite{bai}, G\"otze--Tikhomirov \cite{GT10Circular},  and Tao--Vu \cite{tao_vu_circular}. The result states that the empirical eigenvalue distribution of the non-Hermitian random matrix $n^{-1/2} X_n$ converges weakly as $n\to\infty$ to the uniform distribution on the complex unit disk $\{|z|\leq 1\}$. On the other hand, the \textit{semicircle law} was
introduced by Wigner \cite{wigner55,wigner58} and states that the empirical eigenvalue distribution of the Hermitian random matrix $(2n)^{-1/2} (X_n + X^*_n)$ converges weakly to the Wigner semicircle distribution $\mathsf{sc}_1$ with density $\frac 1 {2\pi}\sqrt{4-x^2}$ on the interval $[-2,2]$. Note that if $\xi\sim\mathcal N_{\C}(0,1)$ has standard complex Gaussian distribution, then $X_n$ is a so-called Ginibre matrix, while $(X_n + X_n^*)/\sqrt 2$ is a GUE matrix. We refer to~\cite{byun2022progress} for a complete overview of Ginibre matrices, to~\cite{BC12} for a great survey on the circular law and to~\cite[Theorem~2.1.1]{AGZ} for a proof of the semicircle law and random matrix theory in general.

Meanwhile, the family of \textit{elliptic laws}, introduced by Girko \cite{girko_elliptic} and studied also in
\cite{naumov,nguyenorourke,
akemann2018universality,byun2023real,ameur2023almost,fyodorov1997almost}, provides a continuous interpolation between these fundamental results. %circular and the semicircle laws.
The elliptic law with Hermiticity parameter $t\in (-1,1)$ is defined as the uniform distribution on the ellipse
$$
\mathcal E_t \coloneqq \left\{z\in\C:\frac{(\Re z)^2}{(1+t)^2} + \frac{(\Im z)^2}{(1-t)^2} \leq 1\right\}.
$$
Note that for $t=0$ one recovers the circular law, while at $t\uparrow 1$, the ellipse collapses to the interval $[-2,2]$ and the elliptic law converges to the Wigner semicircle distribution on this interval. The elliptic laws appear as the limits for the empirical eigenvalue distributions of the matrices
$$
\sqrt{1+t} \,\frac{X_n+X_n^*}{2\sqrt{n}}+\sqrt{1-t}\,\frac{X_n-X_n^*}{2\sqrt{n}}.
$$

%{\color{red} Zakhar: $\sqrt n$ missing in both denominators?}

It is conjectured in \cite{hallho} that the transition from circular to elliptic can be seen as the evolution of the corresponding characteristic polynomials under the heat flow. Specifically, applying the heat flow for time $t$ to the characteristic polynomial of a Ginibre matrix conjecturally should give a new polynomial whose zeros are asymptotically uniform on the ellipse $\mathcal E_t$. This claim is an example of a more general \emph{heat-flow conjecture}, which has been studied in \cite{hallho}. In the above example, the uniform distribution on $\mathcal E_t$ is easily seen to be the 
push-forward of the circular law under the transport map $T_t(w)=w+t\bar w$. The results in \cite{hallho} suggest that the individual zeros of the heat-evolved characteristic polynomial should move approximately along curves of the form $t\mapsto T_t(w)$. 

In this work, we prove an analogue of the heat-flow conjecture in the setting of random polynomials with independent coefficients.
Let us summarize the results of this paper.
\begin{itemize}
\item We confirm the transition between the above-mentioned circular, elliptic and semicircle laws for random Weyl polynomials undergoing the heat flow, as their degree tends to infinity.
\item More generally, we study a large class of rotationally invariant initial distributions $\nu_0$ that can arise as limiting empirical root distributions of random polynomials with independent coefficients. After the heat flow at time $t$, we determine the zero distribution $\nu_t$ of the heat-evolved polynomials  in terms of its logarithmic potential.
\item We find two explicit critical times of the heat flow: Singularities develop in the limiting distribution $\nu_t$ at time $t=t_{\mathrm{sing}}$ and $\nu_t$ collapses to the semicircle distribution at time $t= t_{\mathrm{Wig}}$.
\item For times $0<t<t_{\mathrm{sing}}$, we give a complete description of the non-singular distribution $\nu_t$ as the push-forward of $\nu_0$ under an explicit transport map $T_t$, which turns out to be linear in $t$. In the case of Weyl polynomials, $T_t$ is optimal in Wasserstein-$2$-sense.
\item We exemplify the theory by the several examples including Weyl, Kac and Littlewood--Offord random polynomials.
\item These results will be accompanied by various perspectives, such as Hamilton--Jacobi equations for the logarithmic potential, Burgers' equation for the Stieltjes transform and a strong connection to free probability:
\item If the initial distribution $\nu_0$ has a freely additive circular component, then the effect of the heat flow can be viewed as increasing a semicircle component in the real direction and decreasing it in the imaginary direction.
\end{itemize}

Some of the results listed above could be predicted using the heuristic
arguments developed in \cite{hallho} to support the heat flow conjecture. See, for
example, Section \ref{subsec:heuristics} in the present paper, where one such heuristic argument
is explained. We emphasize, however, that it is currently not known how to
carry out a rigorous analysis of the $n\rightarrow\infty$ limit using the
methods in \cite{hallho}. The setting of random polynomials with independent
coefficients---including Weyl polynomials and more generally the polynomials
in \cite{KZ14}---allows for rigorous results using completely different techniques.
We thus obtain the first rigorous results for polynomials with complex roots
evolving under the heat flow.

The starting point for our approach is Theorem \ref{theo:main_general_g}. This theorem expresses the
logarithmic potential of the limiting root distribution of a heat-evolved
random polynomial as the supremum of a certain expression, involving the
\textquotedblleft exponential profile\textquotedblright\ of the initial
distribution $\nu_{0}$ and the logarithmic potential of the semicircle
distribution. The proof of Theorem \ref{theo:main_general_g}, meanwhile, generalizes the pattern
laid out in \cite{KZ14}. Once this theorem is proved, we then want to (1) compute the
supremum as explicitly as possible (which is a nontrivial task already in
simple special cases such as Kac polynomials), and (2) understand how the
limiting root distribution evolves in time. The time-evolution of the limiting
root distribution can be expressed either (a) by a transport theorem (Section
\ref{sec:transport_map}), (b) by the PDE\ satisfied by its log potential (Section \ref{subsec:PDE_perspective}), or (c) in
the language of free probability (Section \ref{sec:free_prob}). The tasks (1) and (2) occupy the
technical heart of the paper and give the results stated above in terms newly
available objects such as transport maps.

The results of this paper belong to an active research area that investigates how differential operators acting on random functions affect their zeros, see \cite{Byun,OSteiner,KT22,Steiner21,hallho,COR23,diff-paper, ZakharLeeYang,tao_blog1,rodgers_tao,Angst}.
We will review related literature during the course of our presentation, which is organized as follows. In Section \ref{sec:main_results}, we first present  our main result in the special case of Weyl polynomials and then state the general main result for arbitrary isotropic distributions of zeros. In Section \ref{sec:transport_map}, we describe the limiting distribution in terms of the push-forward under a transport map and apply the results to explicit examples in Section \ref{sec:examples}. Sections \ref{sec:Further_perspective} and \ref{sec:free_prob} are devoted to the connections of heat-evolved random polynomials to PDE's (\S \ref{subsec:PDE_perspective}), optimal transport (\S \ref{subsec:OT}), and free probability (\S \ref{sec:free_prob}). Up to this point, we aim to keep the presentation free of technicalities. The proof of the general main theorem will be given in Section \ref{sec:proof_main_thm} and proofs of its implications can be found in Section \ref{sec:proof_nu_t}.

\section{Main results}
\label{sec:main_results}
Let us now be more precise, fix notation and rediscover the semicircle and circular law in the world of (random) polynomials.

\subsection{Heat-flow operator and notation}
Given a complex ``time'' $s\in \C$, the action of the heat-flow operator on a polynomial  $P(z)$  is defined by the terminating series
$$
\exp\left\{-\frac {s} {2} \partial_z^2 \right\} P(z) =  \sum_{j=0}^\infty \frac 1 {j!}\left(-\frac {s}{2}\right)^j \partial_z^{2j} P(z),
$$
where $\partial_z$ denotes differentiation in the complex variable $z$.
More precisely, in the course of this paper, we will use the following notation. If $z= x+\ii y$ is a \textit{complex variable}, then the Wirtinger derivatives (or the Cauchy--Riemann operators) will be denoted by
\begin{align}\label{eq:Wirtinger}
\partial_z  = \frac 12 \left(\frac{\partial}{\partial x} - \ii \frac{\partial}{\partial y}\right),
\qquad
\partial_{\bar{z}}  = \frac 12 \left(\frac{\partial}{\partial x} + \ii \frac{\partial}{\partial y}\right).
\end{align}
On the other hand, the notation $\frac{\partial}{\partial x}, \frac{\partial}{\partial y},\ldots$ will be used to denote partial derivatives in the \textit{real} variables $x,y,\ldots$.

In the special case when $P(z) = z^n$ and $s=1$, the heat-evolved polynomial leads to the classical (probabilist) Hermite polynomials $\He_0(z) = 1, \He_1(z) = z, \He_2(z)=z^2-1,\ldots$ given by
\begin{equation}\label{eq:hermite_poly_def}
\He_n(z)
=
\exp\left\{-\frac 1 2  \partial_z^2\right\} z^n
=
n! \sum_{m=0}^{[n/2]}  \frac{(-1)^m }{m!  2^m} \cdot \frac{z^{n-2m}}{(n-2m)!},
\qquad n\in \N_0.
\end{equation}

The \textit{empirical distribution of zeros} of an algebraic polynomial $P(z)$ of degree $n\in \N$, i.e.\ the probability measure assigning to each zero the same weight $1/n$, will be denoted by
\begin{equation}\label{eq:empirical_distr_zeros_def}
 \llbracket  P \rrbracket \coloneqq \frac 1n \sum_{\substack{z\in \C:\, P(z) = 0}} \delta_z.
\end{equation}
By convention, the zeros are always counted with multiplicities.\footnote{Also, we agree that if $P$ is constant, we define $\llbracket P \rrbracket \coloneqq0$. For random polynomials that we consider, such as Weyl polynomials, the probability of this event converges to $0$ as $n\to\infty$.}
For instance, it is well known~\cite{gawronski_strong_asymptotics} that  the limiting empirical distribution $ \llbracket  \He_n(\sqrt {\tfrac n t} z)\rrbracket  $ of rescaled Hermite polynomials converges weakly to the semicircle distribution $\mathsf{sc}_t$ on the interval $[-2\sqrt t,2 \sqrt t]$ with Lebesgue density $x\mapsto \sqrt{4t - x^2}/(2\pi t)$, for $t>0$.
This fact, together with the formula
$\exp\left\{-\frac{t}{2n} \partial_z^2\right\} z^n =  (\tfrac t n)^{n/2} \He_n(\sqrt{\tfrac n t }z)$ explains why it is natural to choose the time scaling 
$$s= t/n
$$ 
when applying the heat-flow operator to a polynomial of degree $n\in \N$. This time scaling will be used in the sequel.

\subsection{Weyl polynomials}
The first main result of the present paper verifies a  modification of the heat-flow conjecture in which the characteristic  polynomial of the Ginibre random matrix is replaced by the so-called Weyl polynomial, which has the same limiting distribution of zeros.  The \textit{Weyl polynomials} are defined by
\begin{equation}\label{eq:weyl_poly_def}
W_n(z) = \sum_{k=0}^n \frac{(\sqrt n \, z)^k}{\sqrt{k!}}\xi_k,
\end{equation}
where $\xi_0,\xi_1,\ldots$ are i.i.d.\ complex-valued, non-degenerate (i.e., non-deterministic) random variables with
$$
\E \log (1+ |\xi_0|)<\infty.
$$

It is known~\cite[Theorem~2.3]{KZ14} that the empirical distribution of zeros of the Weyl polynomials converges, as $n\to\infty$, to the uniform distribution on the unit disk $\bD$ in the following sense. Let  $\cM(\C)$ be the space of finite measures on $\C$ endowed with the weak topology (which is generated by the L\'evy--Prokhorov metric, for example).\footnote{If $\mu_1,\mu_2,\ldots$ and $\mu$ are random elements with values in $\cM(\C)$, then $\mu_n$ converges weakly to $\mu$ in probability (respectively, almost surely) if and only if $\int_\C f \dd \mu_n \to \int_\C f \dd \mu$ in probability (respectively, almost surely) for every bounded continuous function $f:\C\to\R$; see~\cite{berti_pratelli_rigo_almost_sure_weak_conv}.} Then, $\llbracket W_n\rrbracket $, viewed as a random element in the metric space $\cM(\C)$, converges weakly a.s.\ to the uniform distribution on the unit disk $\{|z|\leq 1\}$. The next theorem, which can be seen as a random-polynomials analogue of the heat-flow conjecture,  describes the limit distribution for zeros of the \textit{heat-evolved} Weyl polynomials $\exp\{-\frac t{2n} \partial_z^2\} W_n$, as $n\to\infty$.  For an illustration, see Figure~\ref{fig:weyl_heat_flow}.

\begin{theorem}\label{theo:main_weyl}
For every $t\in \C$, the empirical measure of zeros of the polynomial $\exp\{-\frac t{2n} \partial_z^2\} W_n$, viewed as a random element with values in $\cM(\C)$, converges in probability to a deterministic probability measure $\nu_t$ that can be described as follows.
\begin{itemize}
\item[(i)] If $t \in (0,1)$, then $\nu_t$ is the uniform distribution on the ellipse
\begin{equation}\label{eq:ellipse_E_t_def}
\mathcal E_t = \left\{z= x+\ii y\in \C: \frac{x^2}{(1+t)^2} + \frac{y^2}{(1-t)^2} \leq 1\right\}.
\end{equation}
\item[(ii)] If $t\geq 1$, then $\nu_t$ is the Wigner law supported on the interval $[-2\sqrt t,+2 \sqrt t]$.
\item[(iii)] If $t = |t| \eee^{\ii \phi} \in \C$ is complex, then $\nu_t$ is the push-forward of $\nu_{|t|}$ under the rotation  map $z\mapsto z \eee^{\ii \phi/2}$.
\end{itemize}
\end{theorem}

The action of the heat flow on the plane Gaussian analytic function (which is an entire function from which Weyl polynomials with Gaussian coefficients can be obtained by truncation) has been studied in our paper~\cite{GAF-paper}.

\begin{figure}[t]
	\centering
	\includegraphics[width=0.32\columnwidth]{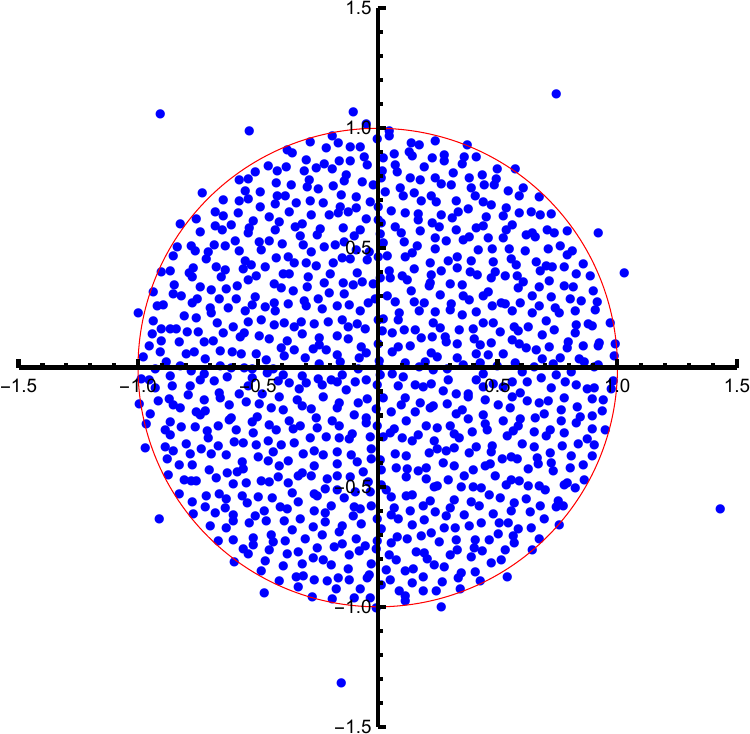}
	\includegraphics[width=0.32\columnwidth]{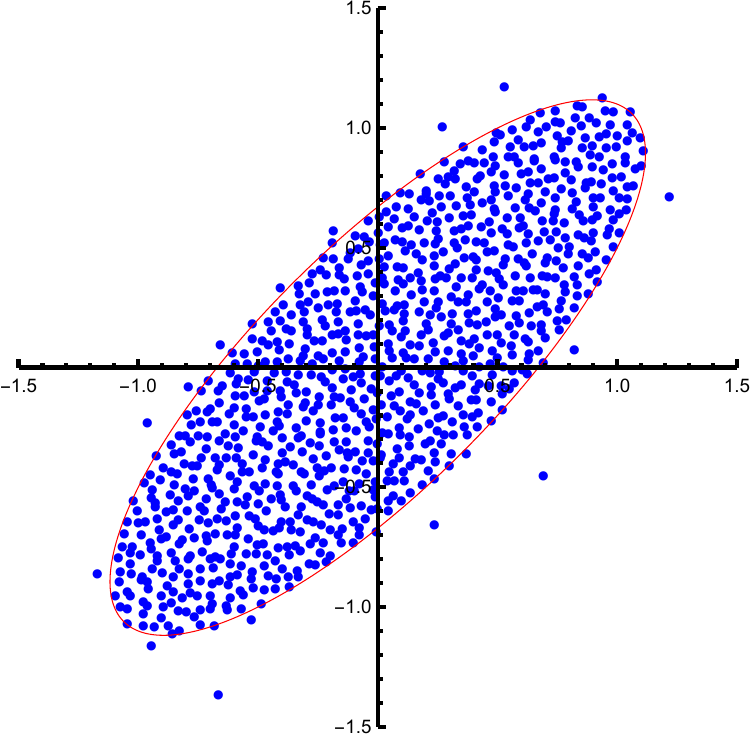}
	\includegraphics[width=0.32\columnwidth]{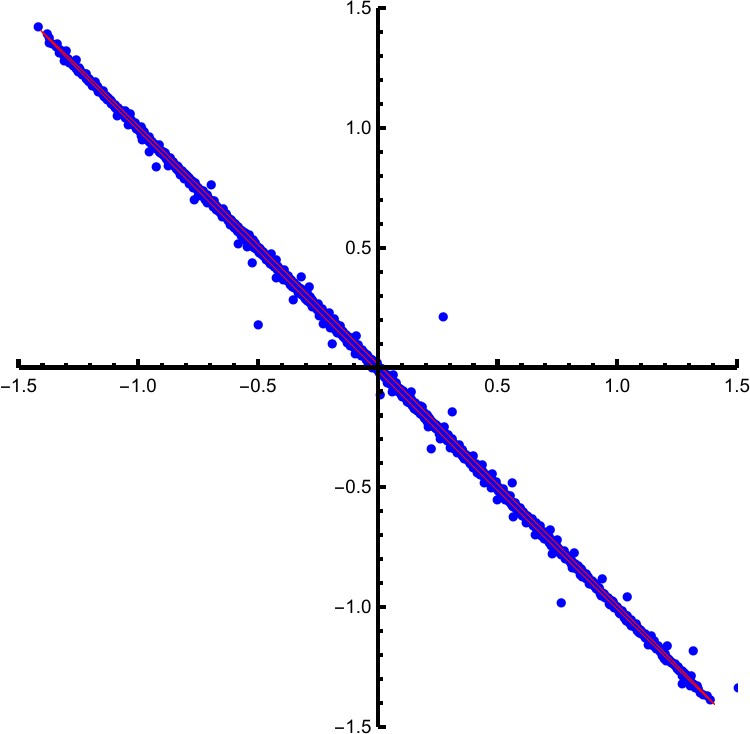}
	\caption{Zeros of $\exp\{-\frac t{2n} \partial_z^2\} W_n$ with $t=0$, $t=\ii/2$, $t=-99/100\cdot \ii$. The degree is $n=1000$.}
\label{fig:weyl_heat_flow}
\end{figure}

\subsection{Polynomials with general rotationally invariant distributions of zeros}\label{subsec:main_sec}
We are now going to study the zeros of heat-evolved polynomials with an arbitrary \textit{rotation invariant} initial distribution of zeros.
To this end, we consider the following general random polynomials
\begin{equation}\label{eq:P_n_general_g_def}
P_n(z)=\sum_{k=0}^n\xi_k a_{k;n} z^k
\end{equation}
which have been studied in \cite[\S 2.5]{KZ14}. From now on we suppose that the following conditions hold:
\begin{enumerate}
\item[(A1)]  \label{cond:A1} $\xi_0,\xi_1,\ldots$ are i.i.d.\ complex-valued, non-degenerate random variables with
\begin{align}\label{eq:condition_A1}
\E \log (1+ |\xi_0|)<\infty.
\end{align}
\item[(A2)] \label{cond:A2} $a_{k;n}$ with  $n\in \N$ and $k\in \{0,\ldots, n\}$  are deterministic complex coefficients such that for some continuous function $g:[0,1] \to \R$
    %(the exponential profile of the coefficients)
    we have
\begin{equation}\label{eq:condition_A2}
\lim_{n\to\infty} \sup_{k\in \{0,\ldots, n\}} \left|\frac 1n \log |a_{k;n}| - g\left(\frac kn\right)\right|=0
\end{equation}
Moreover, the function $g$ is concave (hence, it has left and right derivatives $g_-'$ and $g_+'$) and satisfies $g_-'(1) > -\infty$.
\end{enumerate}

The function $g$ is the ``exponential profile'' of the coefficients in the sense that  $|a_{k;n}| = \eee^{n g(k/n)+ o(n)}$, as $n\to\infty$, and the error term is uniform in $k\in \{0,\ldots, n\}$.
The next theorem describes the asymptotic distribution of roots of $P_n$, as $n\to\infty$. A somewhat  more general form of this result can be found in~\cite[Theorem~2.8]{KZ14}.

\begin{theorem}[Kabluchko--Zaporozhets]\label{theo:distribution_roots_general_g_time_0}
Let $P_n$ be given by~\eqref{eq:P_n_general_g_def} and suppose that conditions    \hyperref[cond:A1]{(A1)} and    \hyperref[cond:A2]{(A2)} hold with some function $g$. Then, the sequence of empirical measures $\llbracket P_n\rrbracket $ (which are viewed as random elements with values in $\cM(\C)$) converges in probability  to a certain rotationally invariant, compactly supported, deterministic probability measure $\nu_{0}$ which is characterized by its logarithmic potential
\begin{align}\label{eq:nu_g}
U_0(z)=\int_\C\log|z-w|\nu_0(\dint w)=\sup_{\alpha\in[0,1]}\big(g(\alpha)+\alpha\log|z|\big)-g(1),
\qquad
z\in \C\backslash\{0\}.
\end{align}

The radial part of the measure $\nu_0$ is the push-forward of the Lebesgue measure on $[0,1]$ under the map $\alpha \mapsto \eee^{-g'_-(\alpha)}$, that is $\nu_0(\{|z|\leq \eee^{-g'_-(\alpha)}\}) = \alpha$ for all $\alpha \in (0,1]$.  %where $g_-'(\alpha)$ denotes the left derivative of $g$.
\end{theorem}

The function $g$ plays the role of a tuning parameter, defined by \eqref{eq:condition_A2}. By choosing a suitable $g$ one can achieve an arbitrary rotationally invariant distribution $\nu_0$ of roots, up to a minor requirement; see~\cite[Theorem~2.9]{KZ14}. The additional assumption $g'(1)\neq -\infty$ in    \hyperref[cond:A2]{(A2)} is technical only, as it ensures $\nu_0$ to have bounded support. The assumptions of~\cite[Theorem~2.8]{KZ14} allow $\nu_0$ to have infinite support and cover even the case when $P_n$ is an infinite Taylor series. To avoid unnecessary technicalities, we prefer to consider only polynomials here. Modifications needed to treat random entire functions will be discussed below; see Section \ref{subsec:entire_functions}.

\begin{example}\label{ex:Kac}
For the \emph{Kac polynomials} $K_n(z) = \sum_{k=0}^n \xi_k z^k$ we have $g(\alpha) \equiv 0$ and $\nu_0$ is the uniform distribution on the unit circle $\{|z|=1\}$. The maximum in~\eqref{eq:nu_g} is attained at  $\alpha_0(z)=\mathbbm 1_{(1,\infty]}(|z|)$ for $|z|\neq 1$. For $|z|=1$, the maximum is attained at all $\alpha \in [0,1]$.
\end{example}
\begin{example}\label{ex:Weyl}
For the \emph{Weyl polynomials}~\eqref{eq:weyl_poly_def} we have $g(\alpha) = -\frac 12 (\alpha \log \alpha - \alpha)$ (with $g(0)\coloneqq0$) and $\nu_0$ is the uniform distribution on the unit disk $\bD$. The maximum in~\eqref{eq:nu_g} is attained at $\alpha_0(z)=\min(|z|^2,1)$, for all $z\in\C$.
\end{example}

\begin{example}\label{ex:LO}
More generally, we can consider \emph{Littlewood--Offord polynomials} whose exponential profile is given by
\begin{equation}\label{eq:litllewoof_offord_exp_profile}
g(\alpha) = -\beta(\alpha \log \alpha - \alpha),
\qquad
0< \alpha \leq 1,
\qquad
g(0)\coloneqq0,
\end{equation}
where $\beta>0$ is a parameter. The maximum in~\eqref{eq:nu_g}  is attained at $\alpha_0(z)=\min \{|z|^{1/\beta}, 1\}$ and it follows that
$$
\nu_{0}(\{|z|\leq r\})
=
\min \{r^{1/\beta}, 1\},
\quad
r>0,
\qquad
\frac{\nu_{0}(\dint z)}{\dint z} = \frac 1 {2\pi \beta} |z|^{\frac 1\beta - 2},
\quad
|z|\leq 1.
$$

Three particular families of random polynomials satisfying    \hyperref[cond:A1]{(A1)} and    \hyperref[cond:A2]{(A2)} with exponential profile~\eqref{eq:litllewoof_offord_exp_profile} are given by
\begin{equation}\label{eq:littlewood_offord_polys}
L_n^{(1)}(z)=\sum_{k=0}^n  \frac{\xi_k}{(k!)^{\beta}} (n^\beta z)^k
,\;\;\;
L_n^{(2)}(z)=\sum_{k=0}^n  \frac{\xi_k}{k^{\beta k}} (n^\beta \eee^\beta z)^k
,\;\;\;
L_n^{(3)}(z)=\sum_{k=0}^n  \frac{\xi_k}{\Gamma(\beta k+1)} (\beta^\beta z)^k.
\end{equation}
The first family has been introduced and studied by Littlewood and Offord~\cite{littlewood_offord1,littlewood_offord2}. The case $\beta= 1/2$ corresponds to the Weyl polynomials.  \emph{Exponential random polynomials} $P_n(z) = \sum_{k=0}^n  \xi_k (n \, z)^k/k! $
%$P_n(z) = \sum_{k=0}^n \frac{(n \, z)^k}{k!} \xi_k$
correspond to $\beta = 1$. The \emph{absolute values} of their roots are uniformly distributed on $[0,1]$ in the large $n$ limit.
 \end{example}

In the next theorem we describe the asymptotic  distribution of roots of the heat-evolved polynomial
\begin{equation}\label{pnzt}
P_n(z;t)\coloneqq\eee^{-\frac t{2n} \partial_z^2} P_n(z).
\end{equation}

\begin{theorem}[First Main Theorem]\label{theo:main_general_g}
Let $P_n$ be given by~\eqref{eq:P_n_general_g_def} and suppose that conditions    \hyperref[cond:A1]{(A1)} and    \hyperref[cond:A2]{(A2)} hold.
Then, for every $t \in \C\backslash\{0\}$, the empirical measure of zeros of the polynomial $\exp\{-\frac t{2n} \partial_z^2\} P_n$, viewed as a random element with values in $\cM(\C)$, converges in probability to a certain deterministic, compactly supported  probability measure $\nu_{t}$.
%\begin{itemize} \item[(i)] If $t > 0$ is real, then t
The measure $\nu_t$ is uniquely determined by its logarithmic potential, which may be computed as follows. Define
\begin{equation}\label{eq:Psi_def}
\Psi(z) = \frac 14 \Re\left(z^2 - z \sqrt{z^2 - 4}\right)  + \log\left|\frac{z + \sqrt{z^2 - 4}}{2}\right| -\frac 12,
\qquad
z\in \C
\end{equation}
and then
\begin{equation}\label{eq:f_def}
f_t(\alpha,z) = g(\alpha)+ \alpha\log \sqrt{|t|\alpha} + \alpha \Psi\left(\frac{z}{\sqrt{t\alpha}}\right),
\end{equation}
for $\alpha\in (0,1],$ $z\in \C$ and $t\in\C\backslash\{0\}$. In the excluded case $\alpha = 0$, we continuously extend $f_t$ to $f_t(0,z)\coloneqq \lim_{\alpha\downarrow 0} f_t(\alpha, z)
=g(0)$.
Finally, define
\begin{equation}\label{eq:var_identity}
U_{t}(z) \coloneqq \sup_{\alpha \in [0,1]} f_t(\alpha,z) -g(1), \qquad z\in \C.
\end{equation}
Then the logarithmic  potential of $\nu_t$ is computed as
\begin{equation}\label{eq:log_potential_limit}
\int_{\C} \log |z - y| \, \nu_{t}(\dd y) = U_{t}(z), \qquad z\in \C.
\end{equation}

\end{theorem}

\begin{remark}\label{rem:Psi}
In \eqref{eq:Psi_def}, $\sqrt{z^2 - 4}$ is defined to be holomorphic outside a branch cut along the real axis from $-2$ to $2$ and so that $\sqrt{z^2 - 4} \sim z$ as $|z|\to \infty$. We use a similar branch choice in all similar expressions throughout the paper. It is worth noting that the function $\Psi$ is the logarithmic potential of the semicircle law $\mathsf{sc}_1$ and we will comment more on both logarithmic potentials $\Psi$ and $U_t$ in Section \ref{sec:Prop_log_pot} below.
\end{remark}

\begin{remark}\label{rem:Wirtinger_logpot_Stieltjes}
Denoting by $\Delta$ the distributional Laplacian, the logarithmic potential satisfies the distributional Poisson equation $\nu_{t} = \frac 1 {2\pi} \Delta U_{t}$; see~\cite[Theorem~3.7.4]{ransford}. In particular, $U_t$ determines the measure $\nu_t$ uniquely. Furthermore, we may write $\Delta=4\partial_z\partial_{\bar z}$ in terms of the Wirtinger derivatives defined in \eqref{eq:Wirtinger}.

The Stieltjes transform\footnote{Note the convention of the sign.} of $\nu_t$ is given by
$$
m_{t}(z) \coloneqq  \int_{\C} \frac{\nu_{t}(\dint y)}{z - y}   = 2\partial_z U_{t},
$$
where the last equality is understood in distributional sense. Indeed, since the function $1/z$ is locally integrable on $\C$, so are Stieltjes-transforms $m_t$, hence they are Schwartz-distributions with weak derivative $\partial_{\bar z}m_t=\pi\nu_t$.
\end{remark}

\begin{remark}
For $z\neq 0$ we have  $f_0(\alpha,z)\coloneqq \lim_{t\to 0} f_t(\alpha, z) = g(\alpha)+\alpha\log|z|$ and $U_0(z) = \sup_{\alpha\in[0,1]}(g(\alpha)+\alpha\log|z|)-g(1)$, which is consistent with Theorem~\ref{theo:distribution_roots_general_g_time_0}.
\end{remark}

\begin{remark} \label{rem:rotation}
If $t = |t| \eee^{\ii \phi}\in \C$ is complex, then $\nu_{t}$ is the push-forward of $\nu_{|t|}$ under the rotation map $z\mapsto z \eee^{\ii \phi/2}$. Indeed, it is easy to  check that $U_t(z \eee^{\ii \phi/2}) = U_{|t|}(z)$.  Therefore, when studying the properties of $\nu_t$, there is no loss of generality in assuming that  $t>0$ for simplicity. All results in the sequel have natural extensions to complex $t\in\C$, e.g. explicitly mentioned in Remark \ref{rem:push_forward_rotated}. However, for differential equations in Sections \ref{subsec:local_pushforward}, \ref{subsec:PDE_perspective}, and \ref{subsec:proof_PDE} it is necessary to use complex $t$ and thus, we will explicitly distinguish between $t>0$ and complex $t$ throughout the paper.
\end{remark}

\begin{remark}
Although Theorem \ref{theo:main_general_g} is an important result, much of the work in the paper is to extract information from the theorem about the limiting distribution $\nu_t$. Notably, in Theorem \ref{theo:collapse_to_wigner}, we will find that after a certain time $t_{\mathrm{Wig}}$, the measure $\nu_t$ becomes \textit{exactly} equal to the semicircular distribution, and in Section \ref{sec:transport_map}, we will obtain local and global transport theorems relating $\nu_t$ to $\nu_0$.
\end{remark}

In the special case of Weyl polynomials (for which $\nu_0$ is the uniform distribution on the unit disk) we recover Theorem~\ref{theo:main_weyl}; see Section~\ref{Weyl.sec} for a detailed analysis of this case.  On the other hand, even in the simple looking case of the Kac polynomials (for which the initial distribution of zeros is uniform on the unit circle), the structure of the measure $\nu_{t}$ is quite intricate; see Figure~\ref{fig:examples_endroots}. A detailed analysis of the Kac case will be carried out in a separate work.

\begin{figure}[ht]
 \includegraphics[width=.33\textwidth]{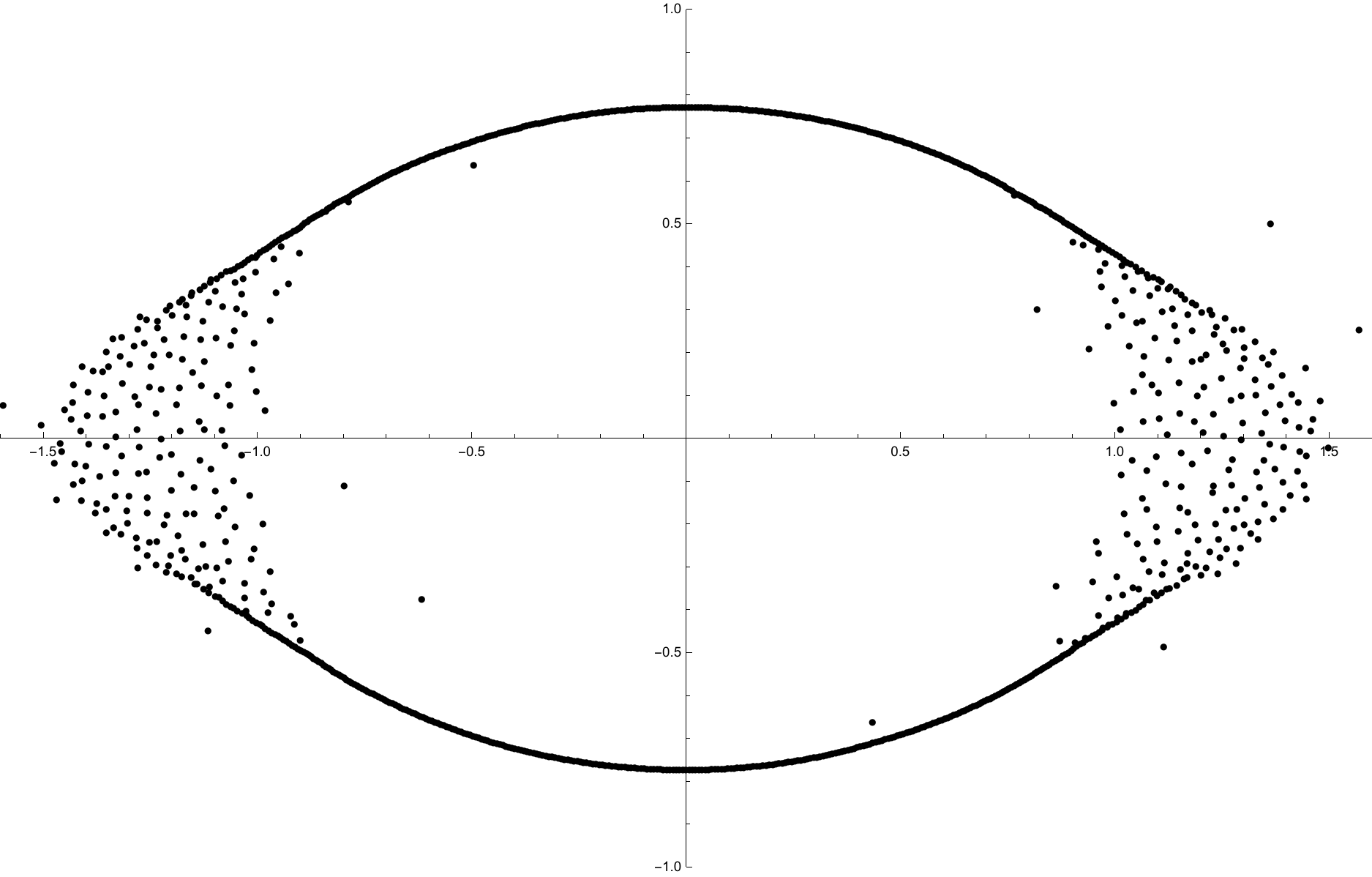}\includegraphics[width=.33\textwidth]{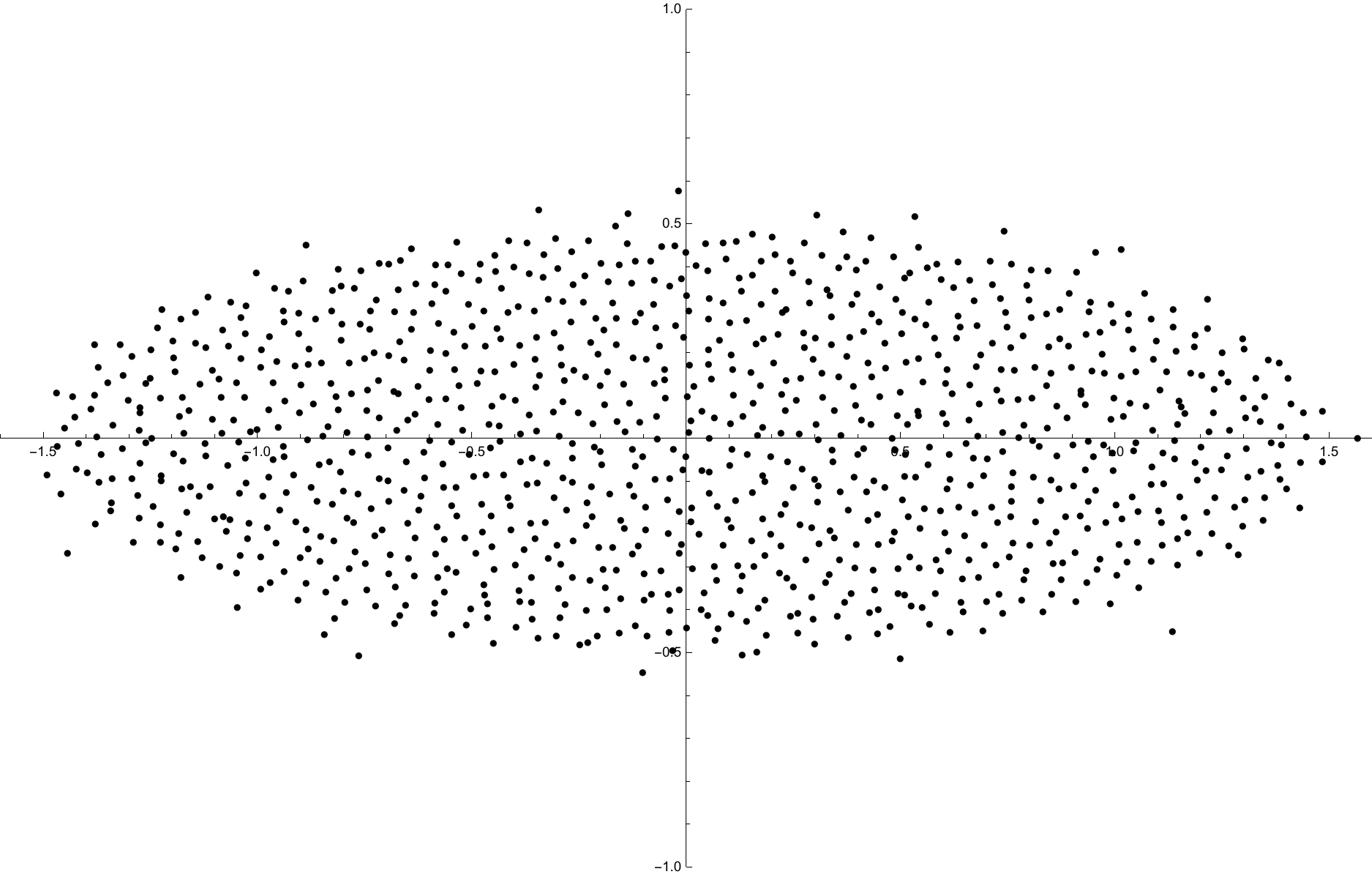}\includegraphics[width=.33\textwidth]{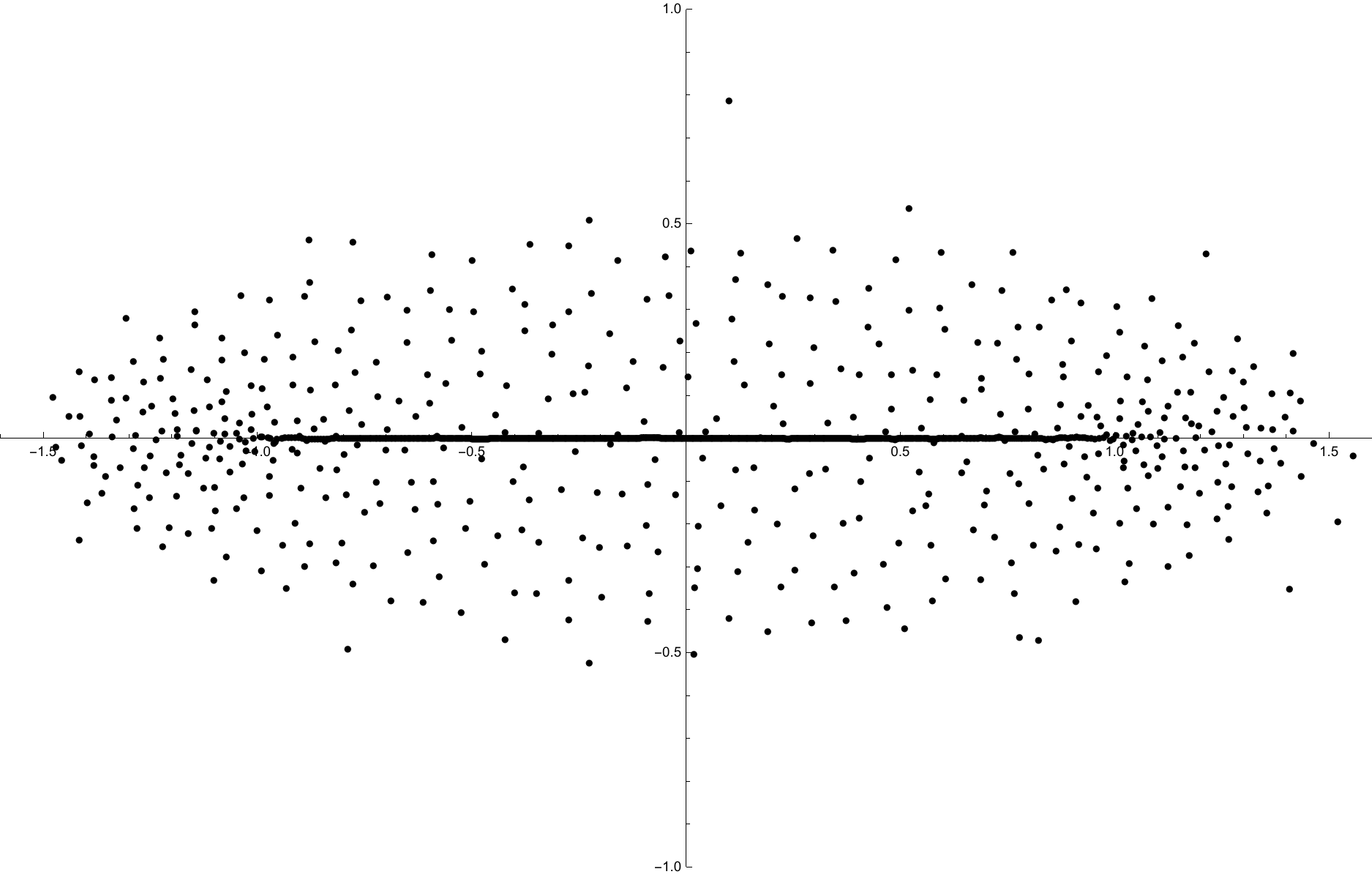}
   \caption{Samples of the empirical root distribution $\llbracket P_n(z;t)\rrbracket $ of heat-evolved random polynomials of degree $n=1000$ after time $t=1/2$. Here, the heat evolution started with Kac polynomials (left) from Example \ref{ex:Kac}, Weyl polynomials (center) from Example \ref{ex:Weyl} and Littlewood--Offord polynomials (right) from Example~\ref{ex:LO}. Theorem \ref{theo:main_general_g} describes the limiting distributions. Animated versions of this figure (among others) can be found in ``poly\_heat\_flow\_animated.pdf'' of the supplementary files on arxiv.org.} \label{fig:examples_endroots}
   \end{figure}

\subsubsection*{Collapse to the Wigner law}
As (real and positive) time $t$ grows, the roots of heat-evolved random polynomials tend to be drawn towards the real line. In fact, there exists a threshold value $t_{\mathrm{Wig}}$ after which $\nu_t$ becomes the semicircle distribution $\textsf{sc}_t$ of variance $t$ on the real line.
\begin{theorem}\label{theo:collapse_to_wigner}
Let $t>0$ and suppose that conditions    \hyperref[cond:A1]{(A1)} and    \hyperref[cond:A2]{(A2)} hold.  Then,  $\nu_t$ is the semicircle distribution $\textsf{sc}_t$ of variance $t$ if and only if
\begin{equation}\label{eq:t_Wig_formula}
t\ge t_{\mathrm{Wig}}\coloneqq\sup_{\alpha\in(0,1)}\exp \left( 2 \cdot \frac{g(\alpha)-g(1)}{1-\alpha}+\frac{\alpha}{1-\alpha}\log\alpha+1\right).
\end{equation}
For $0< t < t_{\mathrm{Wig}}$, the distribution $\nu_t$ is not concentrated on the real line.
\end{theorem}

Note that the function to be maximized in the exponential can be extended to a continuous function on $[0,1]$ and hence the supremum in~\eqref{eq:t_Wig_formula} exists and is finite. In the proof of Theorem~\ref{theo:collapse_to_wigner} we will demonstrate that for $t\ge t_{\mathrm{Wig}}$ and all $z\in\C$ the maximum of the function $f_t(\cdot,z)$ in~\eqref{eq:f_def}  is attained at $\alpha=1$.
%that is $f_t(\alpha,z) \leq f_t(1,z)$ for all $\alpha \in (0,1)$.

\begin{example}\label{ex:collapse_wigner_examples}
A simple calculation shows that $t_{\mathrm{Wig}}=\eee$ for Kac polynomials (the supremum in~\eqref{eq:t_Wig_formula} is attained as $\alpha\downarrow 0$) and $t_{\mathrm{Wig}}=1$ for Weyl polynomials (the function in the exponent vanishes). The latter proves Part~(ii) of Theorem~\ref{theo:main_weyl}.
More generally, for Littlewood--Offord polynomials with $g(\alpha) = - \beta(\alpha \log \alpha - \alpha)$  we have $t_{\mathrm{Wig}}=\exp(1-2\beta)$ for $\beta\le 1/2$  and $t_{\mathrm{Wig}}=1$ for $\beta\ge 1/2$ (in the former case the supremum is attained at $\alpha=0$ and in the latter case as $\alpha \uparrow 1$).

\end{example}

\subsubsection*{Universality conjecture}\label{subsec:conjecture_any_rot_invar_initial_distr}
Given any rotationally invariant probability measure $\nu$ on $\C\backslash\{0\}$ satisfying $\int_{0}^{\eps} \nu (\{|z| < r\}) r^{-1} \dint r <\infty$ for some $\eps>0$, we can find a concave function $g:[0,1]\to \R$ for which $\nu = \nu_0$, that is the zeros of $P_n$ are asymptotically distributed according to $\nu$; see~\cite[Theorem~2.9]{KZ14}. The following is an extended version of the heat-flow conjecture from \cite[\S 2]{hallho}.

\begin{conjecture}\label{conj:Universality}
The conclusion of Theorem~\ref{theo:main_general_g} remains in force if $P_n(z)$ is replaced by any  sequence of polynomials $Q_n(z)$ whose roots are asymptotically distributed according to $\nu= \nu_0$ and are not ``too evenly spaced along a curve.'' For example:
\begin{enumerate}[(a)]
\item Let $Q_n(z) \coloneqq \prod_{k=1}^n (z-X_k)$, where $X_1,X_2,\ldots$ are i.i.d.\ random variables with distribution $\nu=\nu_0$. In particular if $X_k$ are uniformly distributed on $\{|z|\leq 1\}$, then for every $t\in \C$, the empirical distribution $\llbracket \exp\{-\frac t{2n} \partial_z^2\} Q_n\rrbracket $ converges weakly to the uniform measure on the ellipse $\mathcal E_t$ (see also Figure \ref{fig:pairing_roots}).
\item If $Q_n(z)$ is the characteristic polynomial of the Haar random matrix (which is distributed uniformly on the orthogonal group $O(n)$ or the unitary group $U(n)$), then for every $t\in \C$, the empirical distribution $\llbracket \exp\{-\frac t{2n} \partial_z^2\} Q_n\rrbracket $ converges weakly to the same limit $\nu_{t}$ as in the case of Kac polynomials.
\item Let $m\in\N$ and $\{G_{n;k}\}_{n\in\N,k\le m}$ be a family of independent Ginibre $n\times n$-random matrices, defined as non-Hermitian random matrices having independent complex standard normal entries. If $Q_n(z)$ is the characteristic polynomial of the product $G_{n;1}\cdot \ldots \cdot G_{n; m}$ or of $(G_{n;1})^m$, then for every $t\in \C$, the empirical distribution $\llbracket \exp\{-\frac t{2n} \partial_z^2\} Q_n\rrbracket $ converges weakly to the same limit $\nu_{t}$ as in the case of Littlewood--Offord polynomials with $\beta = m/2$.
\end{enumerate}
\end{conjecture}

Note that without heat flow, i.e.\ for~$t=0$, the limiting empirical distributions of $Q_n$ are known. Indeed,  part~(a) is trivial, for part~(b) we refer to \cite{forrester2010log, Mehta,dyson1962statistical,diaconis1994eigenvalues}\footnote{Let us remark at this point that Kac polynomials are closely related to  the eigenvalues of a ``truncated'' Haar-unitary matrix (CUE) obtained by deleting one row and column from the Haar matrix, see \cite{krishnapur2009random,poplavskyi2018exact, forrester2019generalisation}.} and for part~(c) see for instance \cite{ake, burda,GT10,GKT,jalowy,kopel,nemish}.

An example with roots ``evenly spaced along a curve'' violating the conjecture will be discussed in Section~\ref{subsec:evenly_spaced_roots}. Roots forming a \textit{two-dimensional} lattice are \emph{not} considered evenly spaced in the sense of this conjecture and should satisfy it; see Figure~\ref{fig:pairing_roots}. The two-dimensional lattice is remarkably stable under the heat flow. This is related to the fact that the Weierstrass $\sigma$-function (which is the ``simplest'' entire function whose zeros form a lattice) is essentially invariant under the heat flow; see~\cite{GAF-paper}.

\subsection{Properties of the logarithmic potential}\label{sec:Prop_log_pot}
Let us make some comments on the function $\Psi(z)$.
 With the convention $\sqrt{z^2 - 4} \sim z$ as $|z|\to \infty$ of Remark \ref{rem:Psi}, the inverse of the Joukowsky map $z\mapsto z + \frac 1z$ has two branches:  the function $z\mapsto (z + \sqrt{z^2-4})/2$ which is a conformal map between $\C \backslash[-2,2]$ and $\{|z|>1\}$, and the function $z\mapsto (z - \sqrt{z^2-4})/2$ which is a conformal map between $\C \backslash[-2,2]$ and $\{0<|z|<1\}$. Knowing this, it  is easy to check that the function $\Psi(z)$ is well-defined and continuous on the whole of $\C$, including the interval $[-2,2]$, where it does not matter how to choose the sign of the square roots in~\eqref{eq:Psi_def}.

\begin{figure}[ht]
 \includegraphics[width=.33\textwidth]{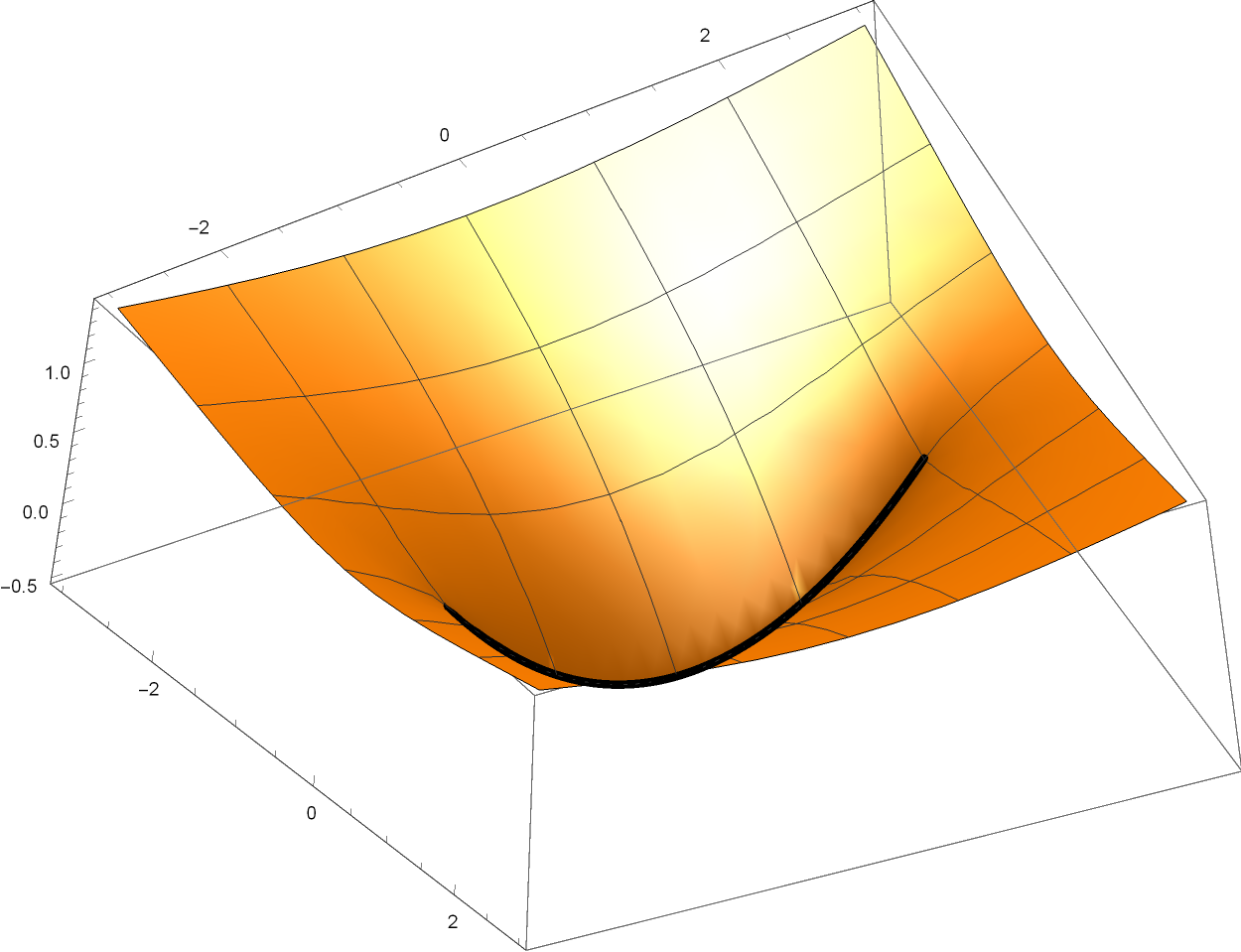}\includegraphics[width=.33\textwidth]{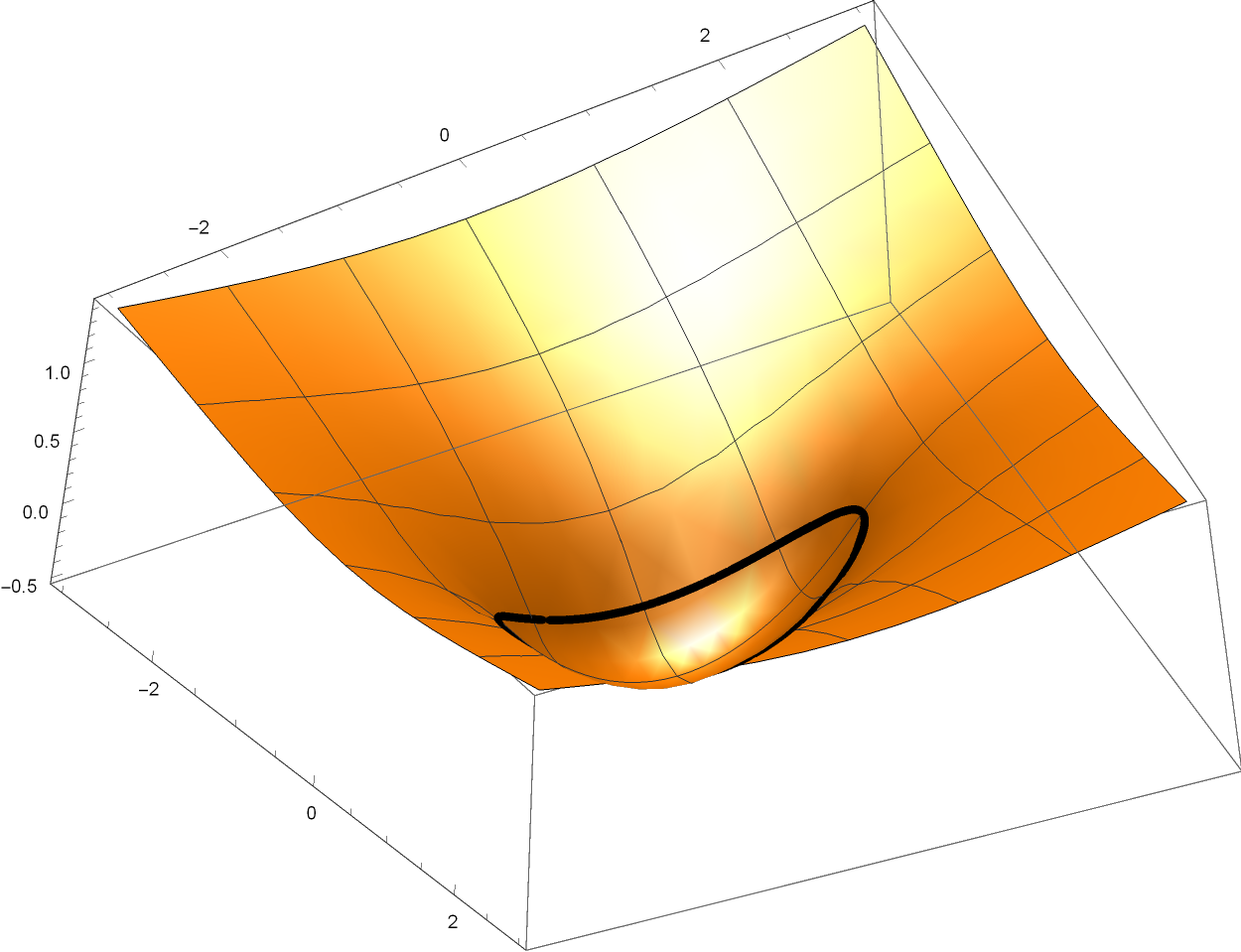}\includegraphics[width=.33\textwidth]{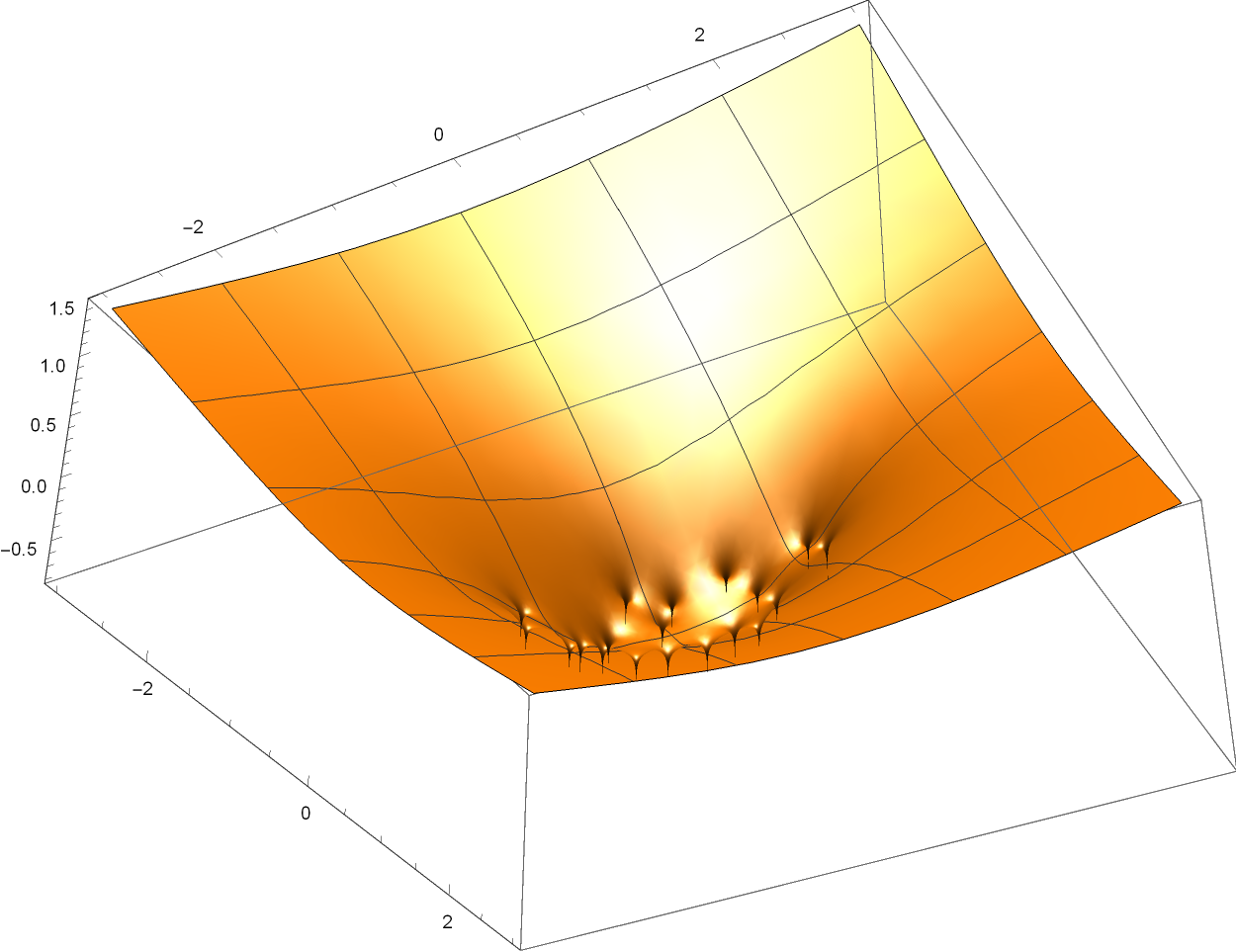}
   \caption{The logarithmic potential $\Psi$ of the semicircle law (left), the logarithmic potential $U_t$ of the elliptic law (center) for $t=1/2$, and its approximation (see Theorem~\ref{theo:main_weyl}) for finite $n=20$, i.e. the logarithmic potential of $\llbracket P_{n}(z;t)\rrbracket $ (right). In the proof of Theorem \ref{theo:main_general_g} it will be crucial that the logarithmic potentials at finite $n$ approximate the limiting logarithmic potentials very well, except at the positions of their singularities.} \label{fig:log_pots}
   \end{figure}

It is well known~\cite[Theorem~5.1 on p.~240]{saff_totik_book} that the logarithmic potential of the semicircle distribution $\mathsf{sc}_{t}$ on $[-2\sqrt t,2\sqrt t]$ is given by
\begin{equation}\label{eq:wigner_log_potential}
\int_{-2\sqrt t }^{2\sqrt t } \log |z-y| \, \mathsf{sc}_{t} (\dd y) =  \frac 12 \log t+ \Psi\left(\frac{z}{\sqrt t}\right),
\qquad z\in \C.
\end{equation}
Moreover, by~\cite[Ex.~3.1.1 on p.~94]{hiai_petz_book}, the Stieltjes transform of $\mathsf{sc}_{t}$ is
\begin{equation}\label{eq:wigner_stieltjes}
\int_{-2\sqrt t }^{2\sqrt t} \frac{\mathsf{sc}_{t}(\dd y)}{z-y} = \frac 1{2t} \left(z - \sqrt{z^2 - 4t}\right), \qquad z\in \C \backslash{[-2\sqrt t ,2 \sqrt t]}.
\end{equation}
In particular, the function $\Psi(z)$ is harmonic outside of $[-2,2]$. %More generally, for every $t>0$ and $\alpha\in (0,1]$,
Observe that the function $z\mapsto \log\sqrt{t\alpha}+\Psi(z/\sqrt{t\alpha})$ appearing in~\eqref{eq:f_def} is the logarithmic potential of the semicircle law $\textsf{sc}_{t\alpha}$ supported on $[-2\sqrt {t \alpha}, +2\sqrt {t \alpha}]$, where $t>0$ and $\alpha\in (0,1]$.

\begin{proposition}\label{prop:log_potential_U_t_cont}
For all $t\in \C \backslash\{0\}$,  the function $U_t(z)$ is continuous and subharmonic  in $z\in \C$.
\end{proposition}

Even though the proofs of this and the following proposition are short and elementary (in the sense that they do not rely on Theorem \ref{theo:main_general_g}), we postpone them to Section \ref{logPotProp.sec}. % in order to keep the introduction free of technicalities.

\begin{proposition}\label{prop:log_potential_large_z}
%and $\lim_{\alpha \uparrow 1} g'(\alpha) \neq -\infty$ (which ensures that $\nu_0$ has bounded support).
Fix $t>0$. Then, for all sufficiently large  $|z|>C_{g,t}$ (where $C_{g,t}$ is a constant depending on $g$ and $t$ only), the logarithmic potential of $\nu_{t}$ is given by
$$
U_{t}(z) = \frac 12 \log t + \Psi\left(\frac{z}{\sqrt t}\right),
\qquad |z|> C_{g,t},
$$
which coincides with the logarithmic potential of the Wigner law on $[-2\sqrt {t}, +2\sqrt {t}]$. Furthermore, the analytic moments of $\nu_t$ coincide for all $t$ with the moments of the semicircular distribution of variance $t$:
\[
\int_\mathbb{C} z^k \,\nu_t(dz) = \int_\mathbb{R} x^k \,\operatorname{sc}_t(dx), \quad k\in\mathbb{N}.
\]
\end{proposition}
\begin{remark}
Let $t>0$.
Taking $\alpha = 1$ in~\eqref{eq:log_potential_limit} and~\eqref{eq:f_def} we obtain the following lower bound which is valid for all $z\in \C$:
$$
U_{t}(z) \geq  \frac 12 \log t + \Psi\left(\frac{z}{\sqrt t}\right).
$$
\end{remark}

\section{The transport map and push-forward theorems}
\label{sec:transport_map}

In this section we introduce a \textquotedblleft transport
map\textquotedblright\ $T_t$ that encodes the way we expect the zeros of
heat-evolved polynomials to move in time. Also, this map will help us to identify maximizers $\alpha$ for $f_t(\alpha,z)$ and %pairs $(\alpha,z)$ where $\partial f_{t}(\alpha,z)/\partial\alpha$ equals zero,
it will simplify the Stieltjes transform $m_t$. % at certain \textquotedblleft regular points\textquotedblright ,
Most importantly,  it gives rise to a global push-forward theorem  (see Theorem \ref{theo:pushforward} below), stating that
\begin{align*}
\nu_t=(T_t)_\#\nu_0 \text{ if }t<t_{\mathrm{sing}},
\end{align*}
%for $t<t_{\mathrm{sing}},$ the measure $\nu_{t}$ is globally the push-forward of $\nu_{0}$ under $T_{t}$
for some explicit critical time $t_{\mathrm{sing}}$ at which singularities start to develop. Moreover, for general $t$ we will prove a local push-forward result near regular (non-singular) points, see Theorem \ref{theo:local_push_forward}.% as the push-forward of the initial measure

\subsection{Heuristics}\label{subsec:heuristics}
Let us motivate the definition of the transport map.\footnote{Another motivation from the point of view of Hamilton--Jacobi equations can be found in Section \ref{subsec:PDE_perspective}.}
Consider any polynomial $P_n(z)$ of degree $n$ and set $P_n(z;t)=\eee^{-\frac t{2n} \partial_z^2} P_n(z)$. Suppose that for some $t_0$, the roots of $P_n(z;t_0)$ are distinct. Then, by the implicit function theorem, it is possible to label the zeros of $P_n(z;t)$, for $t$ near $t_0$, as $z_1(t), \ldots, z_n(t)$ so that these zeros are distinct and depend analytically on $t$. Then the analytic functions $z_1(t),\ldots, z_n(t)$ satisfy the following system of ODE's:
\begin{equation}\label{eq:ODE_for_poly}
\partial_t z_j(t) =  \frac 1 n\sum_{\substack{k=1,\ldots, n:\\ k\neq j}} \frac{1}{z_j(t) - z_k (t)},
\qquad
j\in \{1,\ldots, n\},
\end{equation}
for $t$ near $t_0$\footnote{
Observe that the system~\eqref{eq:ODE_for_poly} need not have a well-defined solution for \emph{all} $t\in \C$ because some of the $z_j(t)$'s may collide.}. Recall that  if we use the notation $\partial_t$, it means that the time $t$ is considered to be complex, whereas $\frac{\partial}{\partial t}$ denotes the partial derivative in the real variable $t$.
These equations are well-known; see \cite{tao_blog1,tao_blog2,csordas_smith_varga,hallho,MarcusFPP,GAF-paper,menon_miracles_burgers,menon_complex_burgers}. See also~\cite{rodgers_tao}, for a fascinating connection to the Riemann hypothesis, and \cite{hallho}, where it is also observed that, after differentiating, \eqref{eq:ODE_for_poly} becomes the much studied (rational) \emph{Calogero--Moser system}
\begin{equation}
\partial_t^{2}z_{j}(t)=-\frac{2}{n^{2}}\sum_{\substack{k=1,\ldots, n:\\ k\neq j}}\frac{1}
{(z_{j}(t)-z_{k}(t))^{3}}.\label{eq:secondDeriv}%
\end{equation}

Assume now that the nearest neighbors of $z_{j}(t)$ are at the
typical\ distance of order $1/\sqrt n$ from $z_{j}(t)$ (as it is the case for i.i.d.\ points from  a bounded two-dimensional  density). Then, passing to the formal hydrodynamic limit $n\to\infty$ in \eqref{eq:secondDeriv}, we expect to have $\partial_t^2z_j(t)\approx0$.
After all, the points $z_{k}(t)$ that are far from $z_{j}(t)$ make a small
contribution to the second derivative because the sum in \eqref{eq:secondDeriv}
is normalized by a factor of $1/n^{2}$. The points
$z_{k}(t)$ that are near $z_{j}(t)$ also make a small contribution because
$1/n^{2}$ is small compared to the inverse cube of the nearest-neighbor distance $1/\sqrt{n}.$

We conclude that as long as the zero $z_{j}(t)$ remains in
a region where $\nu_{t}$ has a bounded two-dimensional  density, %the second derivative of $z_{j}(t)$ with respect to $t$ should be small, meaning that $z_{j}(t)$ should
it should move approximately linear in $t$ and applying \eqref{eq:ODE_for_poly} at $t=0,$ we see that the velocity is given by the Stieltjes transform of the initial distribution.
More precisely, we expect
\begin{equation}
z_{j}(t)\approx z_{j}(0)+t\int_{\mathbb{C}}\frac{1}{z_{j}(0)-w}~\nu_0
(\dint w)=z_j(0)+tm_0(z_j(0)),\label{eq:zjApprox}%
\end{equation}
see Figure \ref{fig:dynamics_roots}  (center and right panels).

\begin{figure}[ht]
\includegraphics[width=.33\textwidth]{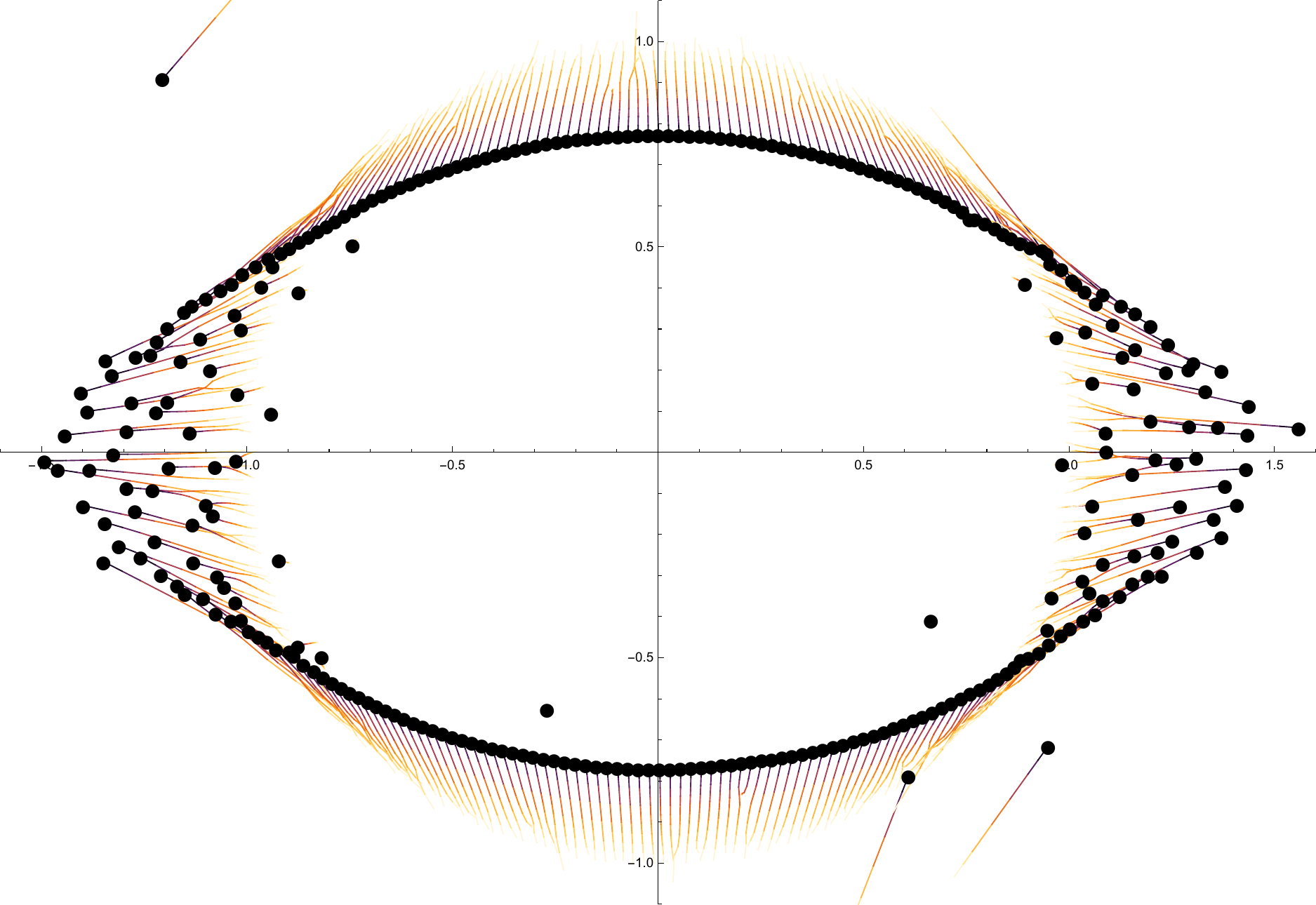}\includegraphics[width=.33\textwidth]{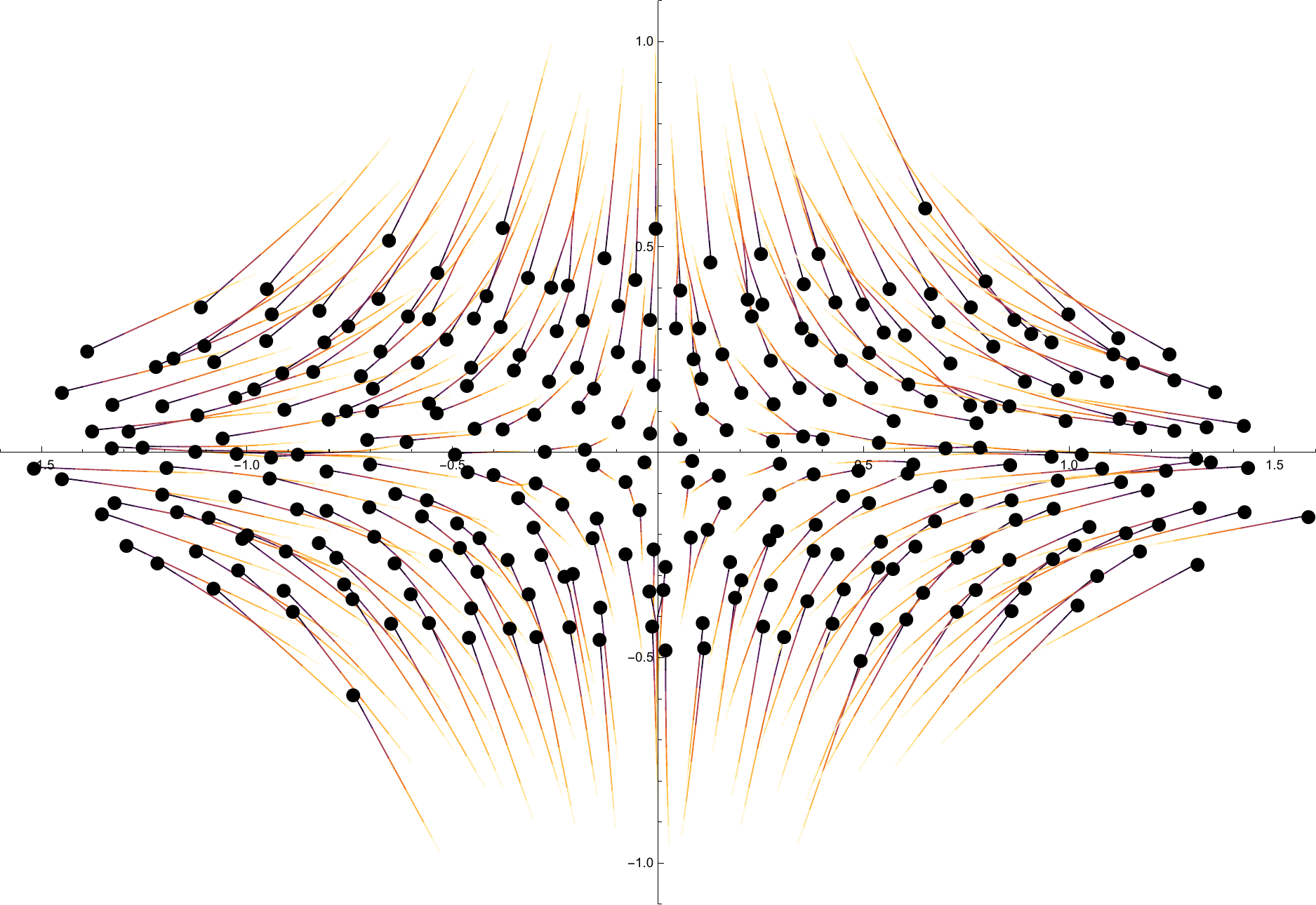}\includegraphics[width=.33\textwidth]{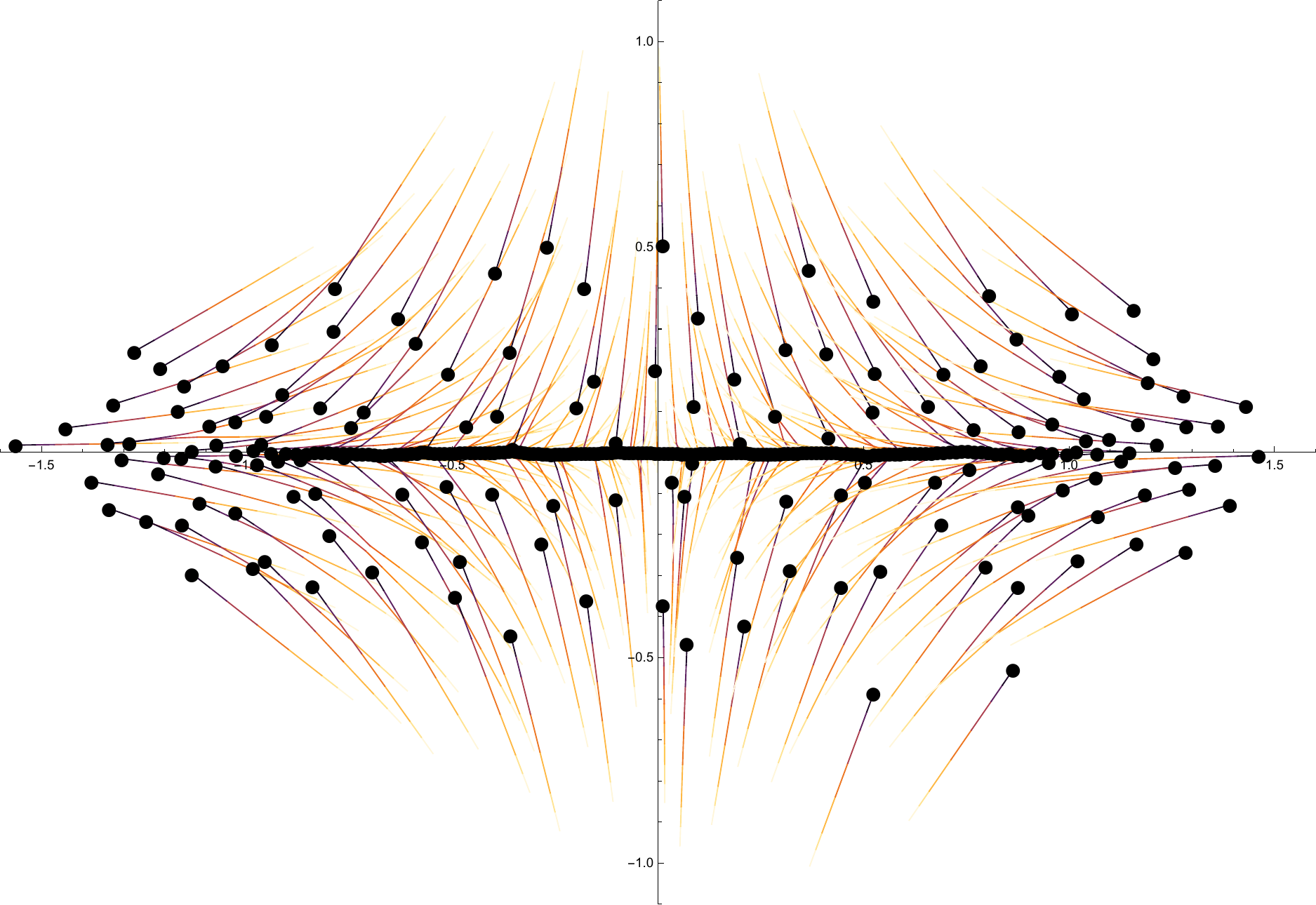}
\caption{Dynamics of the roots of the heat-evolved random polynomial $P_n(z;t)$ degree $n=300$ from $t=0$ until $t=1/2$. Again, we compare Kac polynomials (left) from Example \ref{ex:Kac}, Weyl polynomials (center) from Example \ref{ex:Weyl} and Littlewood--Offord polynomials (right) from Example \ref{ex:LO}. For Weyl polynomials, Theorem \ref{theo:pushforward} explains that the limiting distribution are push-forwards under $T_t$, which according point into the direction of the Stieltjes transform. } \label{fig:dynamics_roots}
   \end{figure}

We emphasize, however, that this behavior is not expected to persist once
$z_{j}(t)$ runs into a region where $\nu_{t}$ fails to have a density with
respect to the Lebesgue measure. If, for example, a positive fraction of the
zeros are concentrating onto a curve in the plane, then the typical spacing
along this curve will be order $1/n$ rather than $1/\sqrt{n}$ and the above
argument is no longer valid if $z_{j}(t)$ lies in the curve; see Figure~\ref{fig:dynamics_roots} (left). We shall exclude singular components of $\nu_t$ by assuming $t<t_{\mathrm{sing}}$ or by introducing  the notion of regular points below.

Passing to the formal hydrodynamic (large $n$) limit suggests that the velocity of a zero located at position $z$ at time $t$ should be given by the Stieltjes transform $m_t(z)$, the limit of the right-hand side of~\eqref{eq:ODE_for_poly}.
%Numerical simulations (see Figure~\ref{fig:dynamics_roots}) suggest that, in the large $n$ (hydrodynamic) limit, the zeros move along certain well-defined deterministic curves.
In general, we conjecture that if at time $0$ some zero was located near $w$, then at time $t>0$ this zero will  be located near $T_t(w)$ for certain deterministic transport maps $T_t :\C \to \C$. The measure $\nu_t$ is then the push-forward of $\nu_0$ under the map $T_t$. According to the previous arguments,  the velocity of zeros turns out to be constant in time as long as $t$ does not exceed a certain critical time $t_{\mathrm{sing}}$ after which singularities or shock waves start to develop.

\subsection{Definition and properties of the transport map}\label{subsec:transport_map}
Let us now introduce a family of maps $T_t$, which we believe to be these transport maps for $t< t_{\mathrm{sing}}$. %These maps will also allow us to simplify and clarify our main result.
%Although we are not able to prove that the trajectories of zeros have well-defined limits in \eqref{eq:zjApprox}, we will rigorously  establish the above-mentioned push-forward property.
We suppose that    \hyperref[cond:A1]{(A1)} and    \hyperref[cond:A2]{(A2)} hold together with the following additional condition:
\begin{itemize}
\item[(A3)] \label{cond:A3} The function $g$ is infinitely  differentiable on $(0,1)$ and $g''(\alpha)<0$ for all $\alpha \in (0,1)$.
If $g'(0)=+\infty$, we additionally assume that $\lim_{\alpha\downarrow 0} \alpha \eee^{g'(\alpha)} = 0$.
% and $\lim_{\alpha \downarrow 0} g'(\alpha) = 0$.
\end{itemize}
Under    \hyperref[cond:A1]{(A1)}--\hyperref[cond:A3]{(A3)}, the probability measure $\nu_0$ can be described as follows. For every $z\in \mathbb C$, the function $\alpha \mapsto g(\alpha) + \alpha \log |z|$, defined for $\alpha\in [0,1]$, has a unique maximizer  which we call $\alpha_0(z)$. It is given by
\begin{align}\label{eq:alpha_0}
\alpha_0(z)
=
\begin{cases}
0, & \text{ if } |z| \leq \eee^{- g'(0)},\\
(g')^{-1}(-\log |z|) & \text{ if } \eee^{- g'(0)} < |z| <  \eee^{- g'(1)},\\
1, & \text{ if } |z| \geq  \eee^{- g'(1)}.\\
\end{cases}
\end{align}
Here, $g'(0) = \lim_{\alpha \downarrow 0} g'(\alpha)$ (which may be $\infty$) and $g'(1) =\lim_{\alpha \uparrow 1} g'(\alpha)$  (which we excluded to be $-\infty$) are understood as one-sided derivatives.
The probability measure $\nu_0$ is supported on the annulus
\begin{equation}\label{eq:R_0_def}
\mathcal R_0 \coloneqq \{z\in \C: \eee^{- g'(0)} < |z| < \eee^{- g'(1)}\}.
\end{equation}
Note that, for $z\in  \mathcal R_0$, $\alpha_0(z)$ is the unique solution to $\eee^{-g'(\alpha)} = |z|$ in $(0,1)$.
From this equation and the last statement of Theorem~\ref{theo:distribution_roots_general_g_time_0} it follows that
\begin{equation}\label{eq:alpha_0_measure}
\alpha_0(r)=\nu_{0}\left(\left\{w\in\C: |w| < \eee^{-g'(\alpha_0(r))}\right\}\right)
=\nu_0(B_{r}(0)),
\end{equation}
for all $r>0$. Here, $B_r(0) = \{z\in \C: |z|<r\}$.  Thus, $\alpha_0$ is the distribution function of the radial part of the probability measure $\nu_0$. Also, let us stress that the Stieltjes transform of $\nu_0$ is given by
$$
m_0(z) :=  \int_{\C} \frac{\nu_{0}(\dint y)}{z - y} =  \frac{\alpha_0(z)}{z}, \qquad z\in \C\backslash \{0\}.
$$
This follows from the fact that the Stieltjes transform of the uniform distribution on the circle $\{|z| = r\}$ equals $1/z$ outside the circle and vanishes inside the circle.

\begin{definition}\label{def:transport_map}
Fix $t \geq 0$. The \emph{transport map} $T_t:\C\backslash  \{0\}\to \C$ is defined by
%More generally, it will be instructive in the sequel to change the coordinates to
\begin{align}\label{eq:T}
T_t(w)=w+t\frac{\alpha_0(w)}{w} = w + t m_0(w).
\end{align}
\end{definition}
%As in \eqref{eq:T}, we shall often drop the parentheses of the map $T_t(w)=T_tw$ for notational simplicity as long as there is no risk of confusion.

We emphasize that this definition is in accordance with the heuristic \eqref{eq:zjApprox}. Also note that $T_t$ equals the Joukowsky-type function $w \mapsto w + t/w$ for $|w|> \eee^{-g'(1)}$, but in contrast to the Joukowsky map, $T_t$ need not be conformal on $\mathcal R_0$ because of the presence of the term $\alpha_0(w)$ in~\eqref{eq:T}. Furthermore, $T_t$ transports towards the real axis and away from the imaginary axis, since
$$
\Im (T_t(w))=\left(1-\frac {t\alpha_0(w)}{|w|^2}\right)\Im w %\leq \Im w
\quad
\text{ and }
\quad
\Re (T_t(w))=\left(1+\frac {t\alpha_0(w)}{|w|^2}\right)\Re w. %\geq \Re w.
$$
If $g'(0)=+\infty$, the condition $\lim_{\alpha\downarrow 0} \alpha \eee^{g'(\alpha)} = 0$ which appears in    \hyperref[cond:A3]{(A3)} is equivalent to  $\lim_{s\downarrow 0}\alpha_0(s)/s = 0$, which means that we can define $T_t(0) = 0$ by continuity. The next proposition lists some elementary properties of the maps $T_t$.
In particular, as long as the time $t$ is sufficiently small, $T_t$ is a bijection which maps circles to ellipses.

\begin{proposition}\label{prop:T_t_bijective}
Suppose that conditions    \hyperref[cond:A1]{(A1)},    \hyperref[cond:A2]{(A2)},    \hyperref[cond:A3]{(A3)} hold. Let $t$ be real and such that
\begin{align}
0<t< t_{\mathrm{sing}}
\coloneqq\left(\sup_{s \in  (\eee^{- g'(0)}, \eee^{- g'(1)})} \left\lvert \left(\frac {\alpha_0(s)}{s}\right)' \right\rvert\right)^{-1}
=\inf_{\alpha\in(0,1)}\left|\frac{g''(\alpha)\eee^{-2g'(\alpha)}}{1+\alpha g''(\alpha)}\right|
.\label{eq:t_sing}
\end{align}
Then, the following hold.
\begin{enumerate}[(a)]
\item The map $T_t:\C\backslash\{ 0\}\to\C\backslash\{0\}$ is bijective.
\item For every $s>0$, $T_t$ maps the circle $\{|w|=s\}$ to the ellipse $\partial \mathcal E_{s,t}$, where $\partial A$ denotes the boundary of a set $A$ and
\begin{equation}\label{eq:def_E_r_mathcal_ellipse_0}
\mathcal E_{s,t}\coloneqq\left\{z=x+iy\in\C:\frac{x^2}{\left(s+t\frac{\alpha_0(s)}{s}\right)^2}+\frac{y^2}{\left(s-t\frac{\alpha_0(s)}{s}\right)^2}\leq 1\right\}.
\end{equation}
\item The ellipses $\mathcal E_{s,t}$ are nested, that is for $0< s < s'< \infty$ we have $\mathcal E_{s,t} \subset \mathcal E_{s',t}$. Their boundaries $\partial \mathcal E_{s,t}$ are disjoint  and foliate $\C\backslash\{0\}$ meaning that $\cup_{s>0} \partial \mathcal E_{s,t} = \C\backslash\{0\}$.
\item $T_t$ is a bijection between the annulus $\mathcal R_0$, see~\eqref{eq:R_0_def},  and the domain
\begin{align*}
\mathcal R_t
&\coloneqq
\left\{z=x+iy\in\C: |z| > \eee^{- g'(0)}, \frac{x^2}{\left(\eee^{-g'(1)} + t \eee^{g'(1)}\right)^2}+\frac{y^2}{\left(\eee^{- g'(1)}-t\eee^{g'(1)}\right)^2} <  1\right\}\\
&=
\bigcup_{s\in (\eee^{- g'(0)},  \eee^{- g'(1)})} \partial \mathcal E_{s,t}.
\end{align*}
\end{enumerate}
\end{proposition}

The definition of $t_{\mathrm{sing}}$ in \eqref{eq:t_sing} guarantees that the semi-axes $s\pm t\alpha_0(s)/s$ in \eqref{eq:def_E_r_mathcal_ellipse_0} will have positive derivative with respect to $s$, for $t<t_{\mathrm{sing}}$.

\subsection{The global push-forward theorem}

With the above properties at hand, we can state the main result of this section, which provides a fairly complete description of the probability measures $\nu_t$ for $0<t< t_{\mathrm{sing}}$ (before the singularities start to develop).

\begin{theorem}\label{theo:pushforward}
Suppose that conditions    \hyperref[cond:A1]{(A1)},    \hyperref[cond:A2]{(A2)},    \hyperref[cond:A3]{(A3)} hold and let $0<t< t_{\mathrm{sing}}$ as defined in \eqref{eq:t_sing}. Then, the following hold.
\begin{enumerate}[(a)]
\item The probability  measure $\nu_t$ is the push-forward of $\nu_0$ under $T_t$:
\begin{equation}\label{eq:push_forward}
(T_t)_\#\nu_{0}=\nu_{t}.
\end{equation}
In particular, $\nu_t$ is concentrated on the domain $\mathcal R_t$. As another consequence,  we have
$$
\nu_t(\mathcal E_{s,t}) = \alpha_0(s),
\qquad s>0.
$$
\item The distribution $\nu_t$ has a globally bounded Lebesgue density $p_t$.
\item The Stieltjes transform $m_{t}(z) = \int_{\C} \frac{1}{z - y} \nu_{t}(\dint y)$ of $\nu_t$ is continuous on $\C$ and satisfies
\begin{align}\label{eq:Stieltjes_pushforward}
m_{t}(T_t(w)) = m_0(w) = \frac{\alpha_0(w)}{w}= \frac{\partial}{\partial t} T_t(w), \quad \text{ for all } w\in
\C\backslash\{0\}.
\end{align}
\item For every $z= T_t(w) \in \mathcal R_t$, the function $\alpha \mapsto f_t(\alpha, z)$ has a unique maximizer $\alpha_t(z) \in (0,1)$.  Moreover, we have
\begin{align}\label{eq:alpha_t}
\alpha_t(T_t(w))=\alpha_0(w),
\quad
\text{ for all } w\in \C.
\end{align}
\item The logarithmic potential $U_t$ of $\nu_{t}$ belongs to the class  $C^1$ (that is, it is real continuously differentiable on $\C$) and satisfies
\begin{align}\label{eq:Ut(Ttw)}
U_{t}(T_t(w))=U_{0}(w)+\frac{\alpha_0(w)^2}{2}\Re\left(\frac{t}{w^2}\right),
\quad
\text{ for all }
w\in \C\backslash\{0\}.
\end{align}
\end{enumerate}
\end{theorem}

\begin{remark}\label{rem:push_forward_rotated}
According to Remark \ref{rem:rotation}, for complex $t=|t|\eee^{i\phi}\in\C$ with $|t|<t_{\mathrm{sing}}$, Equation~\eqref{eq:push_forward} continues to hold for the rotated transport map $T_t(w)\coloneqq\eee^{i\phi/2} T_{|t|}(w)$ and the rotated annulus $\mathcal R_t\coloneqq\eee^{i\phi/2}\mathcal R_{|t|}$. The remaining statements can also be easily adjusted.
\end{remark}

Although the formal derivations presented in Section \ref{subsec:heuristics} serve as motivation for Theorem \ref{theo:pushforward}, rigorously verifying such heuristics or following the proposed proof method outlined in \cite{hallho} poses considerable challenges. Instead, Theorem \ref{theo:main_general_g} introduces a novel approach to the problem and our proofs will be based on a detailed (complex) analysis of the variational identity \eqref{eq:var_identity} while making use of newly available tools (such as the explicit functions $\alpha_0$ and $\Psi$). The proofs will be pushed forward to Section \ref{subsec:proof_local}, after we are familiar with the technicalities of the previous sections.

\begin{remark}
The condition $\lim_{s\downarrow 0}\alpha_0(s)/s = 0$ in    \hyperref[cond:A3]{(A3)} does not follow from $g'(0) = +\infty$ as the example $g(\alpha) = \alpha - \alpha \log \alpha$ shows. In this case, we have $\alpha_0(s)=s$ for $s\in (0,1)$ and $\lim_{s\downarrow 0}\alpha_0(s)/s = 1$ even though  $g'(0) = +\infty$. The same example shows that Theorem~\ref{theo:pushforward} becomes incorrect without this condition. Indeed, formally applying \eqref{eq:t_sing} in this case would give $t_{\mathrm{sing}}= +\infty$, suggesting that \eqref{eq:push_forward} holds for all $t>0$. But this claim is incorrect, since Example~\ref{ex:collapse_wigner_examples} shows that $\nu_t$ is a Wigner law on $\mathbb R$ for $t\geq 1$, whereas $T_t$ does not map to $\mathbb R$. Actually, we shall see in Proposition~\ref{prop:littlewood_offord_push_forward_beta_greater_one_half} that $t_{\mathrm{sing}} = 0$, in the sense that \eqref{eq:push_forward} does not hold for any $t>0$.
\end{remark}

\begin{remark}\label{rem:pushforward_Burgers}
In Section \ref{subsec:PDE_perspective} below, we will see that the Stieltjes transform of $\nu_t$ satisfies an inviscid  Burgers' equation. The discontinuities of the Stieltjes transform correspond to the singularities (i.e.\ one-dimensional singular components) of the measure $\nu_t$. These singularities are analogous to the shock waves in the Burgers' equation.  In particular, \eqref{eq:t_sing} is similar to the formula for earliest critical time in Burgers' equation; see~\cite[Eq.~(2.41) on p.~37]{olver_book_PDE}, \cite[pp.~136-139]{evans_book_PDE}, \cite[Theorem~8.1]{craig_book}. See also Remark \ref{rem:Hopf-Lax} for an interpretation of \eqref{eq:Ut(Ttw)} from a PDE-perspective.
\end{remark}

Consider particles moving (with constant velocities and without any interaction) in the complex plane such that at time $0$ the velocity of particles at position $w$ equals $\alpha_0(w)/w$. If some particle starts at $w$, its position at time $t$ is $T_t(w)$. The particles which started on the circle $\{|z|=s\}$ will be located on the ellipse $\partial \mathcal E_{s,t}$ at time $t$. For $0<t< t_{\mathrm{sing}}$, these ellipses do not intersect and form a foliation of $\C\backslash\{0\}$; see Proposition \ref{prop:T_t_bijective}. At time $t_{\mathrm{sing}}$, two different ellipses intersect for the first time, where particles stick together and singularities start to develop. Note that \eqref{eq:Stieltjes_pushforward} is the hydrodynamic limit of \eqref{eq:ODE_for_poly} and it states that the velocity of particles at location $z=T_t(w)$ and time $t$ is given by the Stieltjes transform $m_t(z)$. This suggests the following conjecture, which is supported also by the push-forward property, see also Figure~\ref{fig:dynamics_roots} and Figure~\ref{fig:pairing_roots}.

\begin{conjecture}
In the hydrodynamic limit, the motion of any individual root under the heat flow is described by the maps $T_t$ for $0<t< t_{\mathrm{sing}}$. That is to say, when $n$ is large, the zeros $z_j(t)$ will, with high probability, stay close to the curves $T_t(z_j(0))$. In particular, the roots should, to good approximation, move with constant velocity along straight lines.
\end{conjecture}

\begin{figure}[ht]
\includegraphics[width=.49\textwidth]{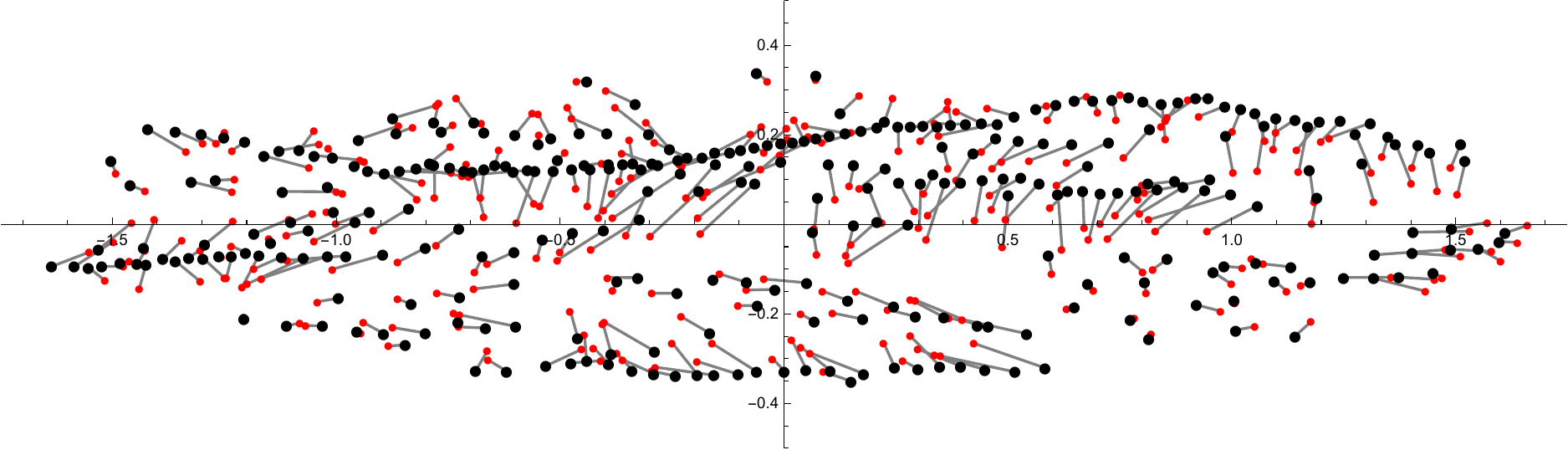}
\includegraphics[width=.49\textwidth]{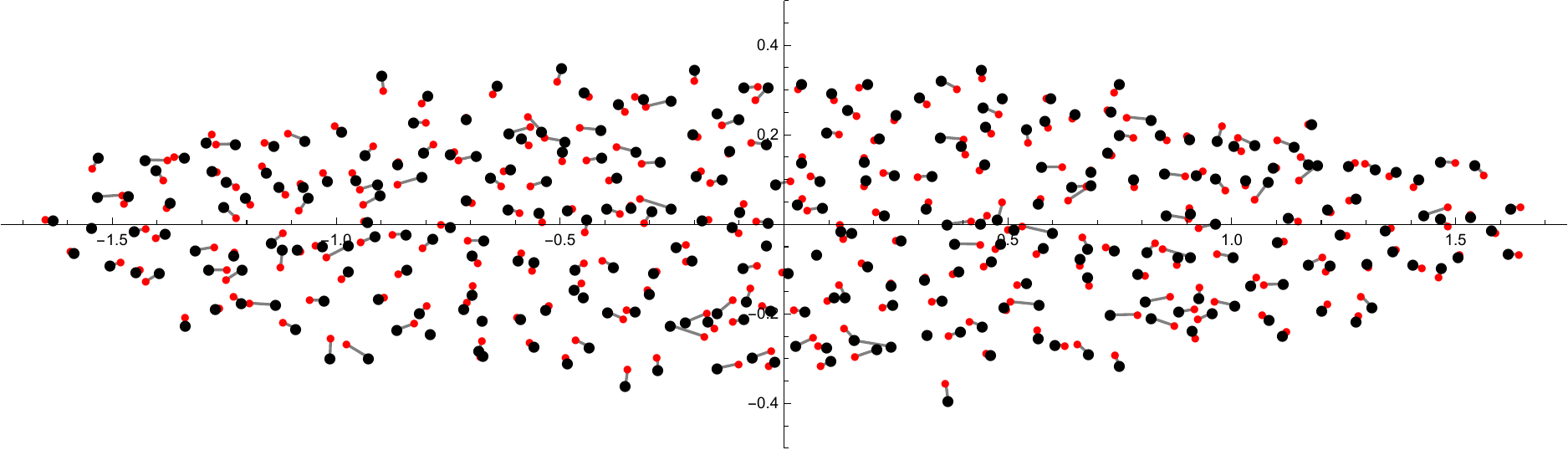}
\includegraphics[width=.49\textwidth]{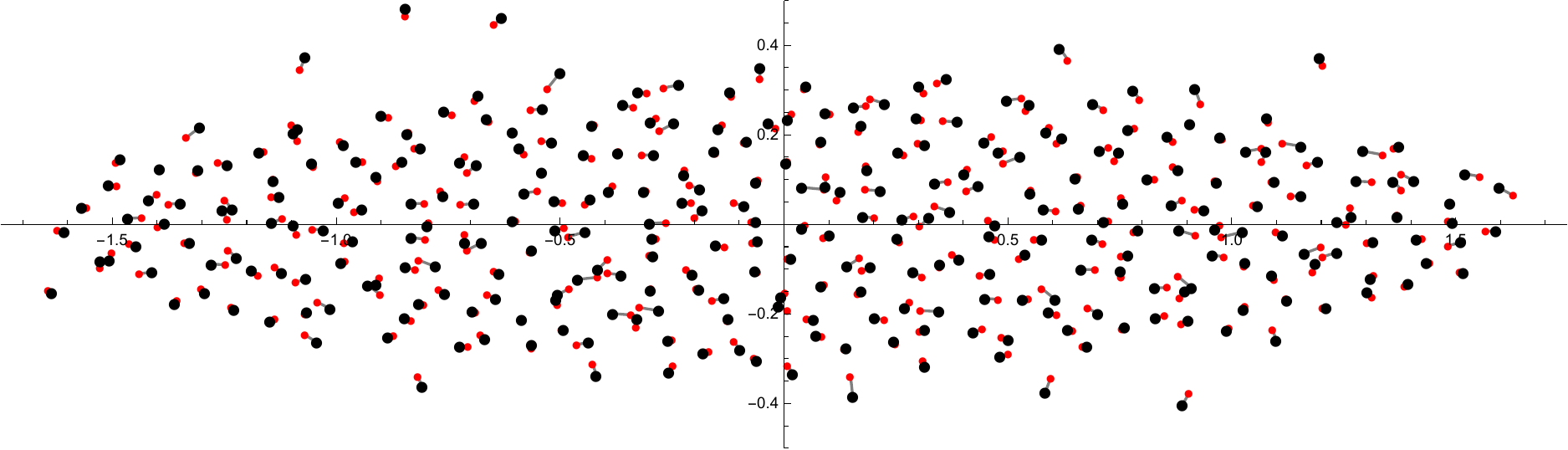}
\includegraphics[width=.49\textwidth]{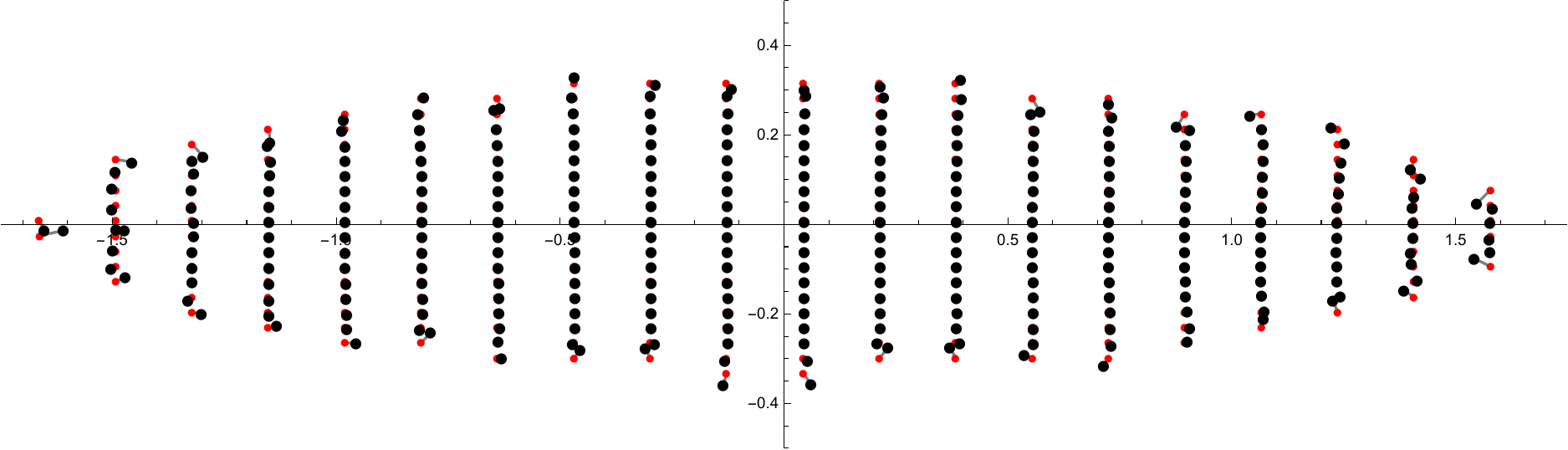}
\caption{``Pairing'' of the roots $z_j(t)$ (black dots) with the predicted locations of the transport map $T_t(z_j(0))$ (red dots) after time $t=2/3$. For $n=300$, we compare the heat evolution of a polynomial with i.i.d.\ roots (top left), the characteristic polynomial of Ginibre matrices (top right), Weyl polynomial (bottom left) and regular lattice in the unit disk (bottom right). The pairing appears to improve as the point process at time $0$ gets more organized. Note the curious line structures which appear for i.i.d.~zeros are also barely visible for eigenvalues (top right), but not for Weyl polynomials.} \label{fig:pairing_roots}
   \end{figure}

An interesting question is how to make the conjecture precise. That is, how close are the roots $z_j(t)$ for finite $n$ to the pushed-forward positions $T_t(z_j(0))$? We believe that the errors should be of order $1/\sqrt{n}$, which is the typical distance between $z_j(t)$ and the root closest to it. Supporting this claim,  it is shown  in \cite[Theorem 1.3]{GAF-paper} that a zero $z(t)$ of the heat-evolved Gaussian analytic function is essentially moving according to $T_t(z(0))$ up to a random error of order $1$. If we approximate the GAF by its Taylor polynomials and rescale appropriately, we obtain the Weyl polynomials. Taking the rescaling into account, the order 1 errors for the GAF case should  translate into order $1/\sqrt{n}$ errors in the case of the Weyl polynomials. Figure~\ref{fig:pairing_roots} indicates that the quality of pairing may depend on the local structure of the point process formed by the roots at time $0$.

We also note that although the transport map $T_t$ is defined for all $t$, the push-forward result in Theorem \ref{theo:pushforward} definitely does not hold for all $t$. In the case of the Weyl polynomials, for example,
$(T_t)_\#\nu_{0}$ will be the uniform measure on an ellipse for all $t>0$ except for $t=1$. By contrast,
the limiting root distribution $\nu_t$ for the Weyl polynomials collapses to the Wigner semicircle law on the real axis for all $t>1$. (See Section \ref{Weyl.sec}.) It is then an interesting question as to whether we can define a
\textit{different} map $\tilde T_t$ for $t>t_\mathrm{sing}$ for which a push-forward theorem does hold.

\begin{openproblem}
% and $t>0$
Define natural transport maps $\tilde T_t$ for all $t>0$ such that $\nu_t = (\tilde T_t)_\#\nu_{0}$.
\end{openproblem}

Such a map should reflect the way we expect the individual zeros to move for arbitrary times, and this
will not be by straight-line motion once $t>t_{\mathrm{sing}}$. Although this problem
is open in general, we will construct such a map in the case of Littlewood--Offord polynomials in Proposition \ref{prop:littlewood_offord_push_forward_beta_greater_one_half}.

For every initial distribution $\nu_0$ we have $t_{\mathrm{sing}}\le t_{\mathrm{Wig}}$. For (rescaled) Weyl polynomials we have $t_{\mathrm{sing}} = t_{\mathrm{Wig}}$.  The next proposition asserts that in all other cases the inequality is strict meaning that the singularities appear strictly before the collapse to the semicircle law.

\begin{proposition}\label{prop:t_sing<t_wig}
Assume    \hyperref[cond:A1]{(A1)}-\hyperref[cond:A3]{(A3)}. We have $t_{\mathrm{sing}}=t_{\mathrm{Wig}}$ if and only if $\nu_0$ is the uniform measure on the disk $B_{r}(0)$ for some $r>0$.
\end{proposition}

\subsection{The local push-forward theorem and the density at regular points}\label{subsec:local_pushforward}

If $t<t_{\mathrm{sing}}$, then Proposition \ref{prop:T_t_bijective} shows that $T_t$ is globally invertible and using \eqref{eq:alpha_t} it is easy to check that the inverse is given by
\begin{align}\label{eq:T_inverse}
T_t^{-1}(z)=\frac{z+\sqrt{z^2-4t\alpha_t(z)}}{2},
\end{align}
where the root is chosen such that it is positive for $z>2\sqrt{t\alpha_t(z)}$ (see also Section \ref{sec:Prop_log_pot}).
Even if $t\ge t_{\mathrm{sing}}$, this inverse is still well-defined locally near so called \textquotedblleft regular points\textquotedblright. (See, for example, the case of the Littlewood--Offord polynomials in Section \ref{LO.sec}.) In this section, we will show that the measure $\nu_{t}$ near a regular point can be expressed \textit{locally} as the push-forward of $\nu_{0}$ under the transport map and we present a simple formula for its density.

\begin{definition}\label{def:regular_points}
Let $\mathcal D\subset \C\times \C$ be a connected open set and let $\alpha_t(z)$ be a smooth function defined on $\mathcal D$ such that that $z/\sqrt{t \alpha_t(z)}\notin [-2,2]$ for all $(z,t)\in \mathcal D$ and one of the following cases occurs:
\begin{itemize}
\item[(D0)] For all $(z,t)\in \mathcal D$, the supremum in~\eqref{eq:var_identity} is attained at $\alpha_t(z) = 0$ only, or
\item[(D1)] For all $(z,t)\in \mathcal D$, the supremum in~\eqref{eq:var_identity} is attained at $\alpha_t(z) = 1$ only, or
\item[(D2)] For all $(z,t)\in \mathcal D$, the supremum in~\eqref{eq:var_identity} is attained at a unique $\alpha_t(z)\in (0,1)$.
\end{itemize}
A point $(z,t)\in\mathbb{C}\times\C$ is called a \emph{regular point}, if there exists a domain $\mathcal D\ni(z,t)$ as above.
\end{definition}

For regular points $(z,t)$, we define $T_t^{-1}(z)$ via \eqref{eq:T_inverse}. The restriction $z/\sqrt{t \alpha_t(z)}\notin [-2,2]$ ensures that the branch of the complex root of \eqref{eq:T_inverse} is unique. % (recall also that this avoids problems in the definition \eqref{eq:Psi_def} of $\Psi$).
%Note that if $(z,t)$ is a regular point, then $ \frac{\partial f_{t}}{\partial z}(\alpha_{t}(z),z)=0$.
Moreover, Theorem \ref{theo:pushforward} (d) implies that $\mathcal D=\{(z,t):z\in\mathcal R _t, |t|<t_{\mathrm{sing}}\}$ are regular points.

\begin{proposition}\label{prop:Stieltjes}
Assume    \hyperref[cond:A1]{(A1)} and    \hyperref[cond:A2]{(A2)}. The Stieltjes transform of $\nu_t$ at a regular point $(z,t) \in \mathcal D$ is given by
\begin{equation}\label{eq:m_t_as_function_of_alpha_t}
m_t(z) \coloneqq  \int_{\C} \frac{\nu_{t}(\dint y)}{z - y} =  \frac 1 {2t} \left(z - \sqrt{z^2 - 4t\alpha_t(z) }\right)=\frac{\alpha_t(z)}{T_t^{-1}(z)}.
\end{equation}
For $t=0$, the right-hand side is defined by continuity as $\alpha_0(z)/z$.
Conversely, we have
\begin{equation}\label{eq:alpha_t_as_function_of_m_t}
\alpha_t(z) = m_t(z) (z - t m_t(z)), \qquad (z,t) \in \mathcal D.
\end{equation}
If we additionally assume    \hyperref[cond:A3]{(A3)}, then $\alpha_0$ is well defined, $\alpha_t(z)=\alpha_0(T_t^{-1}(z))$ and analogously to \eqref{eq:Stieltjes_pushforward} it holds
\begin{align}\label{eq:m_t_alpha_0}
m_t(z)=\frac {\alpha_0(T_t^{-1}(z))}{T_t^{-1}(z)}=m_0(T_t^{-1}(z)).
\end{align}
\end{proposition}

In particular, \eqref{eq:alpha_t_as_function_of_m_t} shows that the number $m_t(z) (z - t m_t(z))$ is real. Note that $m_t(z)$ vanishes if $\alpha_t(z) \equiv 0$ on $\mathcal D$, while for $\alpha_t(z)>0$ it coincides,  at every fixed $(z,t)\in \mathcal D$, with the Stieltjes transform of the Wigner law on $[-2\sqrt{\alpha_t(z) t}, 2\sqrt{\alpha_t(z) t}]$ multiplied by $\alpha_t(z)$. With the help of Remark \ref{rem:Wirtinger_logpot_Stieltjes}, it is now easy to compute the density of $\nu_t$ in terms of the function $\alpha_t(z)$ by differentiating with $\partial_{\bar z}$.
\begin{corollary}\label{cor:density}
Assume    \hyperref[cond:A1]{(A1)} and    \hyperref[cond:A2]{(A2)} and let $t\in \C$. The restriction of the measure $\nu_{t}$ to the domain $\cD_t \coloneqq \{z\in \C: (z,t) \in \cD\}$ has a Lebesgue density given by
$$
p_t(z) =\frac{1}{\pi}\partial_{\bar{z}}\frac{\alpha_{t}(z)}{T_t^{-1}(z)}= \frac{1}{\pi} \frac{\partial_{\overline{z}} \, \alpha_t(z)}{\sqrt{z^2 - 4 t\alpha_t(z)}}.
$$
%which equals $\frac{1}{\pi}\partial_{\bar{z}}\frac{\alpha_{0}(T_t^{-1}z)}{T_t^{-1}z}$ if    \hyperref[cond:A3]{(A3)} holds.
\end{corollary}
Note that the fact that this function is real positive is a property of $\alpha_t$ and at least for $t=0$ it is easy to see that $p_0(w)=\frac 1 \pi \partial_{\bar w}\alpha_0(w)/w=\frac{1}{2\pi|w|}\alpha_0'(|w|)\ge 0$. Under (D0) or (D1), we have $p_t(z) = 0$, which is consistent with the harmonicity of  $U_{t}$ on  $\mathcal D_t$.
We finish this section with the aforementioned local push-forward theorem.

\begin{theorem}%[Local push-forward Theorem]
 \label{theo:local_push_forward} If    \hyperref[cond:A1]{(A1)},    \hyperref[cond:A2]{(A2)},    \hyperref[cond:A3]{(A3)} hold, then the restriction of $\nu_{t}$ to $\cD_t = \{z\in \C: (z,t) \in \cD\}$ is the
push-forward of the restriction of $\nu_{0}$ to $T_t^{-1}(\cD_t)$, more precisely
\[
\left.  \nu_{t}\right\vert _{\cD_t}=(T_{t})_\#(\left.  \nu_{0}\right\vert _{T_t^{-1}(\cD_t)}).
\]

\end{theorem}

\section{Examples}\label{sec:examples}
%\begin{example}
In this section, we will analyze several examples. In the case of the Littlewood--Offord polynomials, which include the Weyl polynomials as a special case, we compute explicitly the times $t_{\mathrm{sing}}$ and $t_{\mathrm{Wig}}$. In the Littlewood--Offord case with $\beta > 1/2$, we will also construct an extended transport map $\tilde T_t$ which gives a push-forward theorem for $t\ge t_{\mathrm{sing}}$. We also look at an extension of our results to random analytic functions,  at a family of polynomials whose zeros concentrate onto an annulus and at polynomials with evenly spaced roots.

\subsection{Weyl polynomials}\label{Weyl.sec}
Let us look more carefully at the special case $\beta = 1/2$ which corresponds to Weyl polynomials and is especially explicit. For Weyl polynomials, we have $g(\alpha) = -\frac 12 (\alpha \log \alpha - \alpha)$ for all $\alpha \in (0,1]$. It follows from~\eqref{eq:alpha_0} that  $\alpha_0(r)=\min(r^2,1)$, for all $r>0$, and that the initial distribution $\nu_0$ is the uniform distribution on $\{|z|\leq 1\}$.  The transport maps $T_t$, $t\geq 0$,  are given by
$$
T_t(w)
=
\begin{cases}
w+t\bar w=(1+t)\Re w +\ii (1-t)\Im w, &\text{ if } |w|\leq 1,\\
w+t/w, &\text{ if } |w|\geq  1.
\end{cases}
$$
In the following we assume that Condition~\eqref{eq:t_sing} holds, which in our case takes the form $0 < t < 1$. Note that the push-forward of $\nu_0$ under $T_t$ (which is a linear self-map of $\R^2$) is the uniform distribution on the ellipse $\mathcal E_t$ defined in~\eqref{eq:ellipse_E_t_def}.
%The supremum of $f_t(\alpha, z)$ over $\alpha\in [0,1]$ will be identified explicitly in ...
%For $0< t <1$,
The inverse of $T_t$ is given by
\begin{equation}\label{eq:T_t_inverse_weyl_polys}
T_t^{-1}(z)
=
\begin{cases}
\frac{\Re z}{1+t}+\ii \frac{\Im z}{1-t}, &\text{ if }z\in \mathcal E_t,\\
\frac 1 2(z+\sqrt{z^2-4t}), &\text{ if }z\notin \mathcal E_t,
\end{cases}
\qquad
0 < t < 1.
\end{equation}
The branch of the root is again chosen such that $T_t^{-1}(z)\to\infty$ as $z\to\infty$. An application of Theorem~\ref{theo:pushforward}, Part~(d), implies that the unique maximizer of $\alpha\mapsto f_t(\alpha,z)$ is given by
$$
\alpha_t(z)=
\begin{cases}\Big(\frac{\Re z }{1+t}\Big)^2+\Big(\frac{\Im z}{1-t}\Big)^2, &\text{ if }z\in \mathcal E_t,\\
1, &\text{ if }z\notin \mathcal E_t,
\end{cases}
\qquad
0 < t < 1.
$$

The logarithmic potential and the Stieltjes transform of the circular law are given by
$$
U_0(z)=
\begin{cases}
\frac 1 2 |z|^2-\frac 1 2, &\text{ if }|z|\le 1,\\
\log|z|, &\text{ if }|z|\ge 1,
\end{cases}
\qquad
m_0(z)=
\begin{cases}
\bar z, &\text{ if }|z|\le 1,\\
1/z, &\text{ if }|z|\ge 1.
\end{cases}
$$
Parts~(c) and~(e) of Theorem~\ref{theo:pushforward} yield the following known formulas~(see~\cite[p.~689]{girko_elliptic} and~\cite[Lemma~5.3.12]{hiai_petz_book} for the case $z\in \mathcal E_t$; the complementary case is difficult to locate in the literature).

\begin{lemma}\label{lem:log_pot_ellipse}
For every $t\in (0,1)$, the logarithmic potential and the Stieltjes transform of the uniform distribution on the ellipse $\mathcal E_t$ defined in~\eqref{eq:ellipse_E_t_def} are given by
\begin{align*}
U_{t}(z)
&=
\frac{1}{\pi (1-t^2)} \int_{\mathcal E_t} \log|z-y| \dd y
=
\begin{cases}
\frac 12 \left(\frac{(\Re z)^2}{1+t} + \frac{(\Im z)^2}{1-t}\right) - \frac 12  &\text{ if } z \in \mathcal E_t,\\
\frac 12 \log t + \Psi\left(\frac{z}{\sqrt t}\right) &\text{ if } z\notin \mathcal E_t,
\end{cases}
\\
m_t(z)
&=
\frac{1}{\pi (1-t^2)} \int_{\mathcal E_t} \frac {\dint y} {z-y}
=
\frac{\alpha_0(T_t^{-1}(z))}{T_t^{-1}(z)}
=
\begin{cases}
\frac{\Re z}{1 + t} - \ii \frac{\Im z}{1-t}, &\text{ if } z\in \mathcal E_t,\\
\frac 1{2 t} ( z - \sqrt{z^2 - 4t}), & \text{ if } z\notin \mathcal E_t.
\end{cases}
\end{align*}
Note that for $z\notin \mathcal E_t$ the logarithmic potential and the Stieltjes transform are the same as for the Wigner law supported on $[-2\sqrt t, +2\sqrt t]$.
\end{lemma}

In the above analysis we assumed that $0<t<1$. For $t\geq 1$, the measure $\nu_t$ becomes one-dimensional; see  Theorem \ref{theo:main_weyl}. At $t=1$, the push-forward property still holds (note that $T_1(w) = 2\Re w$ maps the uniform distribution on the unit disk to the Wigner law on $[-2,2]$), but  for $t>1$ the distribution $\nu_t$  does not coincide with the push-forward $(T_t)_\#\nu_{0}$, at least if we define $T_t$ in the above way. The following definition seems more appropriate in view of Theorem~\ref{theo:main_weyl}:
$$
\tilde T_t(w) = 2 (\Re w) \sqrt t,
\qquad
t>1.
$$

\subsection{Littlewood--Offord polynomials}\label{LO.sec}
In this section we  analyze the Littlewood--Offord polynomials from Example~\ref{ex:LO}. We recall that $g(x)=-\beta (x\log x - x)$ for all $x\in (0,1]$ and $g(0)=0$, where $\beta>0$ is a parameter of the model.  The initial distribution $\nu_0$ is rotationally invariant and satisfies $\nu_{0}(\{|z|\leq r\}) =  \alpha_0(r) = \min \{r^{1/\beta}, 1\}$, for all $r>0$. Among other things, the next proposition shows that $t_{\mathrm{sing}} = 0$ is possible meaning that the singularities may start to develop immediately at time $0$.

\begin{proposition}\label{prop:littlewood_offord_t_sing}
The following assertions hold for the Littlewood--Offord model.
\begin{itemize}
\item[(i)] If $0<\beta \leq 1/2$, then $t_{\mathrm{sing}} = \beta/(1-\beta)$ and $t_{\mathrm{Wig}}=\exp(1-2\beta)$.  In particular,  for all $t\in (0, \beta/(1-\beta))$, Theorem~\ref{theo:pushforward} implies that $\nu_{t}$ is the push-forward of $\nu_{0}$ under the map
$$
T_t(w) = w + t|w|^{1/\beta}/w,  \qquad |w|\leq 1.
$$
\item[(ii)] For $1/2  < \beta < 1$, we have $t_{\mathrm{sing}} = 0$ and $t_{\mathrm{Wig}}=1$. (This case will be studied in Proposition~\ref{prop:littlewood_offord_push_forward_beta_greater_one_half}).
\end{itemize}
\end{proposition}
\begin{proof}
For $s\in (0,1]$ we have $\alpha_0(s)/s = s^{(1/\beta) - 1}$. For $\beta\neq 1$, the derivative of this function is $(1-\beta)/\beta \cdot s^{(1/\beta)-2}$, while for $\beta=1$ the derivative vanishes. Note that $g'(0) = +\infty$ and the condition $\lim_{s\downarrow 0}\alpha_0(s)/s = 0$ is satisfied if and only if $0 < \beta < 1$. For $0<\beta \leq 1/2$, the supremum of the derivative is $(1-\beta)/\beta$ meaning that $t_{\mathrm{sing}} = \beta/(1-\beta)$. This means that we can apply Theorem~\ref{theo:pushforward} if and only if $0 < \beta\leq  1/2$.  For $\beta\in  (1/2, 1)$, the supremum is infinite and hence $t_{\mathrm{sing}} = 0$. For $\beta\geq 1$ we cannot apply Theorem~\ref{theo:pushforward}. For claims about $t_{\mathrm{Wig}}$ see Example~\ref{ex:collapse_wigner_examples}.
\end{proof}

Note that for $\beta >1/2$, the singularities start to develop immediately at time $t_{\mathrm{sing}} = 0$. Still, it is possible to obtain a complete description of the measure $\nu_t$. Recall that $t_{\mathrm{Wig}}=1$ meaning that $\nu_t$ is the Wigner law $\textsf{sc}_t$ if $t\geq 1$. In the following result, we describe $\nu_t$ for $0<t<1$; see also Figure~\ref{fig:littlewood_offord_heat_flow_beta_greater_frac12}.

\begin{proposition}\label{prop:littlewood_offord_push_forward_beta_greater_one_half}
Let $g(x)\coloneqq-\beta (x\log x - x)$ for some parameter $\beta>1/2$. Then, for every $t\in (0,1)$, the probability measure $\nu_t$ is the push-forward of the probability measure $\nu_0$ (which satisfies $\nu_{0}(\{|z|\leq r\}) = r^{1/\beta}$, for $0<r<1$) under the modified transport map
$$
\tilde T_t(w) \coloneqq
\begin{cases}
\sqrt t \cdot  (2\Re w)|w|^{(1-2\beta)/(2\beta)}, & \text{ if } |w|<t^{\beta/(2\beta-1)},\\
w + t |w|^{1/\beta}/w, &\text{ if } t^{\beta/(2\beta-1)}\leq |w| \leq 1,\\
w + t/w, &\text{ if } |w|>1.
\end{cases}
$$
The measure $\nu_t$ has a singular part which is $t^{1/(2\beta-1)}$ times the Wigner semicircle law on the interval $[-2t^{\beta/(2\beta-1)}, 2t^{\beta/(2\beta-1)}]$. The absolutely continuous part of $\nu_t$ is the push-forward of the restriction of $\nu_0$ to the annulus $\{t^{\beta/(2\beta-1)}<|w|<1\}$ by the standard push-forward map $w\mapsto w + t |w|^{1/\beta}/w$.
\end{proposition}

We conjecture that, in the hydrodynamic limit, the modified transport maps $\tilde T_t$ describe the motion of individual zeros. That is, a zero which starts at some point $w$ in the unit disk moves linearly with constant in time speed $|w|^{1/\beta}/w$ until it hits\footnote{This is a statement about the $n\to\infty$ limit only. For finite $n$, the zeros need not become exactly real.} the real axis at time $t(w)= |w|^{(2\beta-1)/\beta}$. For example, zeros starting close to the origin hit the real axis at small times since $\beta>1/2$. Note that the zeros having smaller absolute value than $|w|$ have already reached the real axis at time $t(w)$.  The position at which the zero that started at $w$ hits the real axis is  $2\Re w$. Then, the zero moves along the real axis; a detailed description of this motion will be provided below. The conclusions of Proposition \ref{prop:T_t_bijective} and Theorem~\ref{theo:pushforward} remain in force when restricted to the annulus $\{t^{\beta/(2\beta-1)}< |w| < 1\}$. For example, points $w$ with fixed $|w| = r \in (t^{\beta/(2\beta-1)}, 1)$ are mapped by $\tilde T_t$ to the ellipse $\partial \mathcal E_{r,t}$ with half-axes $r + t r^{(1/\beta)-1}$ and $r - t r^{(1/\beta)-1}$. These ellipses foliate the ellipse $\mathcal E_t$ with half-axes $1+t$ and $1-t$ with exception of the interval $[-2t^{\beta/(2\beta-1)}, +2t^{\beta/(2\beta-1)}]$. This slitted ellipse supports the absolutely continuous part of $\nu_t$; see Figure~\ref{fig:littlewood_offord_heat_flow_beta_greater_frac12}.  For every $w$ from the annulus  $\{t^{\beta/(2\beta-1)}< |w| < 1\}$ we have $\alpha_t(\tilde T_t(w)) = \alpha_0(w)= |w|^{1/\beta}$ and the Stieltjes transform  of $\nu_t$ at $\tilde T_t(w)$ is given by
$$
m_t(\tilde T_t(w)) = \alpha_0(w)/w = |w|^{1/\beta}/w.
$$

\begin{figure}[t]
	\centering
	\includegraphics[width=0.35\columnwidth]{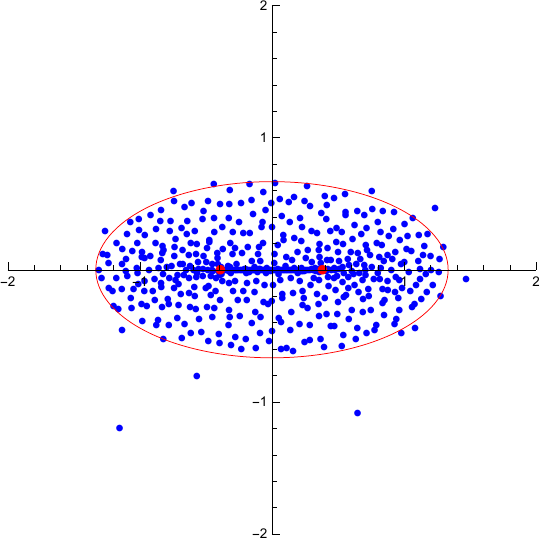}
\includegraphics[width=0.35\columnwidth]{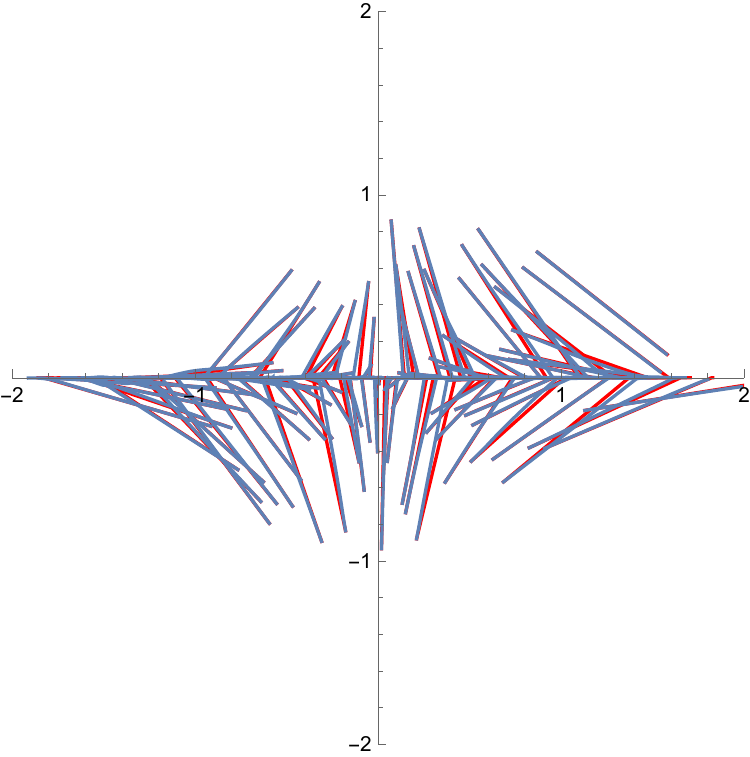}
\caption{Left: Zeros of a heat-evolved Littlewood--Offord polynomial
%$\exp\{-\frac t{2n} \partial_z^2\} L_n^{(1)}$
with $\beta= 3/4>1/2$ and degree is $n=500$  at time $t=1/3$. Right: Trajectories of selected roots under the heat flow (grey)  and under transport maps (red), for all $t\in [0,1]$.}
%Left: $t=1/3$. Right: $t=2/3$. The degree is $n=500$.}
\label{fig:littlewood_offord_heat_flow_beta_greater_frac12}
\end{figure}

Note that the Stieltjes transform of $\nu_t$ is defined on the slitted complex plane and has a jump discontinuity along the interval $[-2t^{\beta/(2\beta-1)}, 2t^{\beta/(2\beta-1)}]$. Given some $a$ in the interior of this  interval, denote the preimages of $a + \ii 0$ and $a - \ii 0$ under $\tilde T_t$ by $w_+$ and $w_-$, where $\Im w_+>0$ and $\Im w_-<0$.  These preimages are complex-conjugate numbers satisfying $|w_+| = |w_-| = t^{\beta/(2\beta-1)}$ and $a= 2\Re w_+ = 2\Re w_-$ (since the zeros that started at $w_+$ and $w_-$ hit the real axis at time $t$ and position $a$). The limits of the Stieltjes transform at $a$ are given by
\begin{align*}
m_t(a + \ii 0) &= |w_+|^{1/\beta}/w_+ = t^{1/(2\beta-1)}/w_+ = w_-/t,
\\
m_t(a - \ii 0) &= |w_-|^{1/\beta}/w_- = t^{1/(2\beta-1)}/w_-= w_+/t.
\end{align*}
Using the Stieltjes inversion formula, we can compute the density of the singular part of $\nu_t$ at $a$ as follows:
$$
p_t(a) = - \frac 1 \pi \Im m_t(a + \ii 0) = \frac 1{\pi t} \Im w_+ = \frac {\sqrt{4t^{2\beta/(2\beta-1)} -a^2}} {2\pi t^{2\beta/(2\beta-1)}} \cdot t^{1/(2\beta-1)}.
$$
This agrees with the density of the Wigner law on $[-2t^{\beta/(2\beta-1)}, 2t^{\beta/(2\beta-1)}]$ multiplied by $t^{1/(2\beta-1)}$.

We are now in position to explain how the zeros move after hitting the real line.
Consider some root located at position $a\in \R$ at time $t\in (0,1)$, where with $|a| < 2t^{\beta/(2\beta-1)}$. By symmetry, we assume $a\geq 0$.  We believe that the velocity of this root is given by the principal (Cauchy) value of the Stieltjes transform of $\nu_t$ at $a$, which, by the Sokhotski–Plemelj theorem, is
$$
v_t(a) = \frac{m_t(a + \ii 0) + m_t(a - \ii 0)}{2} = \Re m_t(a + \ii 0) = (\Re w_-) /t = a/(2t).
$$
So, the speed is $(a/2t)$. This is quite natural in view of the following observation. The roots which started near $w_+$ arrive at time $t$ at position $a$ with speed $|w_+|^{1/\beta}/w_+ = w_-/t$. Similarly, the roots that started at $w_-$ arrive at $a$ with speed $w_+/t$. Recall that $w_-/t$ and $w_+/t$ are complex-conjugate numbers with common real part $a/(2t)$. It is therefore quite natural that the roots which are already at $a$ (and have become real at times $<t$) move with speed $a/(2t)$.

Now, let $z(t)$ denote the position of some root that started at $z(0) = w$.  After the root becomes real, it satisfies the ODE $z'(t) = z(t)/ (2t)$. The general solution is given by  $z(t) = C \sqrt t$. The constant $C$ can be determined from the condition $z(|w|^{(2\beta-1)/\beta}) = 2\Re w$. This gives
$$
z(t) = \sqrt t \cdot  (2\Re w)|w|^{(1-2\beta)/(2\beta)} = : \tilde T_t(w) \quad \text{ for all } \quad  t\in [|w|^{(2\beta-1)/\beta},1].
$$
This derivation is non-rigorous but it agrees with numerical experiments. It turns out that the push-forward of $\nu_0$ restricted to the disk $\{|w|< t^{\beta/(2\beta-1)}\}$ agrees with the Wigner law on the interval $[-2t^{\beta/(2\beta-1)}, 2t^{\beta/(2\beta-1)}]$ multiplied by $t^{1/(2\beta-1)}$. This fact will be verified in the proof of Proposition~\ref{prop:littlewood_offord_push_forward_beta_greater_one_half} which will be given in Section~\ref{subsec:proof_pushforward}.

At time $t=1$, all zeros become real and their distribution is the Wigner law on $[-2,2]$. (Remember: we are talking about the infinite $n$ limit only, for finite $n$ the zeros are approximately real). The zero which started at $w$ is located at $(2\Re w)|w|^{(1-2\beta)/(2\beta)}$ at time $t=1$. Therefore, for $t>1$, we conjecture that $\tilde T_t(w) = \sqrt t \cdot (2\Re w)|w|^{(1-2\beta)/(2\beta)}$  for all $w$ in the unit disk.

In the above discussion we omitted the regime when $0<\beta \leq 1/2$ and $\beta/(1-\beta) < t < \exp(1-2\beta)$.  Numerical simulations show that in this regime the complexity of $\nu_t$ is similar to that of Kac polynomials (corresponding formally to $\beta = 0$) in the regime $0< t < \eee$.
%see Section~\ref{subsec:kac} below.

\subsection{Random entire functions}\label{subsec:entire_functions}
We expect that the main results of the present paper, including Theorem~\ref{theo:main_general_g}, extend in a straightforward way to random entire functions of the form $F_n(z)=\sum_{k=0}^\infty \xi_k a_{k;n} z^k$ satisfying appropriate modifications of the Assumptions    \hyperref[cond:A1]{(A1)}--   \hyperref[cond:A3]{(A3)} (as in~\cite{KZ14}). Most importantly, the analogue of    \hyperref[cond:A2]{(A2)} requires now that for every $c>0$,
\begin{equation}\label{eq:condition_coefficients_infinite_taylor_analogue}
\lim_{n\to\infty} \sup_{k\in \{0,\ldots, [cn]\}} \left|\frac 1n \log |a_{k;n}| - g\left(\frac kn\right)\right|=0,
\end{equation}
for some concave function  $g:[0,\infty) \to\R$. Note that in this setting, $n$ loses its role as the degree; now it is just a parameter appearing in~\eqref{eq:condition_coefficients_infinite_taylor_analogue}.

We then define the empirical distribution of zeros by assigning to each zero of $F_n$ the same weight $1/n$; the result is then a locally finite measure on $\C$.  The space of such measures is endowed with the vague topology. The analogue of Theorem~\ref{theo:main_general_g} states that the empirical distribution of zeros of $\exp\{-\frac t{2n} \partial_z^2\} F_n$ converges to certain locally finite measure $\nu_t = \frac 1 {2\pi} \Delta U_{t}$, where $U_t$ is given by~\eqref{eq:var_identity} and~\eqref{eq:f_def}, but now the supremum in~\eqref{eq:var_identity} has to be taken over $\alpha > 0$. One important difference to the polynomial case is that now $U_t$ loses its role as the logarithmic potential of $\nu_t$ (which is infinite, in general). That is to say, $U_t$ cannot be computed from $\nu_t$ as a convolution against the log function.

A family of examples covered in this way is the case $g(x)=-\beta(x\log x -x)$, $x>0$, with parameter $\beta\ge 1/2$. For each $n$, the resulting entire functions are scaled versions of the Littlewood--Offord entire functions, with
the case $\beta = 1/2$ giving a scaled version of the plane Gaussian analytic function (GAF). In the GAF case, the limiting empirical measure $\nu_0$ is $1/\pi$ times the Lebesgue measure on the plane. The limiting empirical measure of the heat-evolved functions is then $1/(\pi(1-t^2))$ times the Lebesgue measure, for $0\leq t <1$, with the restriction on $t$ needed to ensure that the heat flow is defined; see \cite{GAF-paper}.

For $\beta >1/2$, the limiting empirical measure at time zero satisfies $\nu_{0}(\{|z|\leq r\}) = r^{1/\beta}$ for all $r>0$. A push-forward theorem, analogous to Proposition~\ref{prop:littlewood_offord_push_forward_beta_greater_one_half} in the polynomial case,
holds for all $t>0$, with the definition of the extended transport map $\tilde T_t$ modified as follows:
$$
\tilde T_t(w) \coloneqq
\begin{cases}
\sqrt t \cdot  (2\Re w)|w|^{(1-2\beta)/(2\beta)}, & \text{ if } |w|<t^{\beta/(2\beta-1)},\\
w + t |w|^{1/\beta}/w, &\text{ if } t^{\beta/(2\beta-1)}\leq |w| <\infty.
\end{cases}
$$
(That is, there is no longer a region in which $\tilde T_t(w)$ is equal to $w+t/w$.)

\subsection{One more example}\label{martin.sec}
Let us finish this section with an example of random polynomials whose limiting zero distribution is supported on an annulus and where $0<t_{\mathrm{sing}}<t_{\mathrm{Wig}}$.

\begin{figure}[ht]\label{fig:Martin}
\includegraphics[width=.33\textwidth]{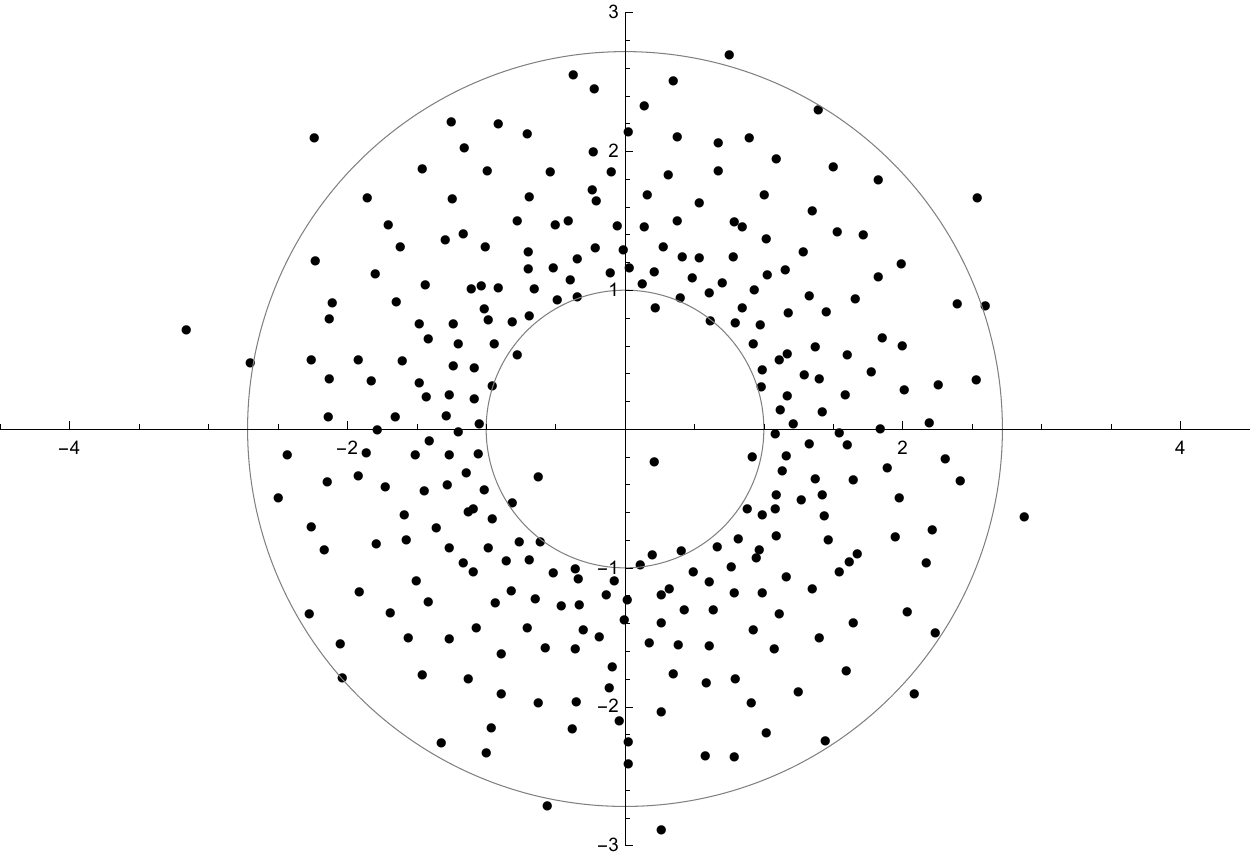}\includegraphics[width=.33\textwidth]{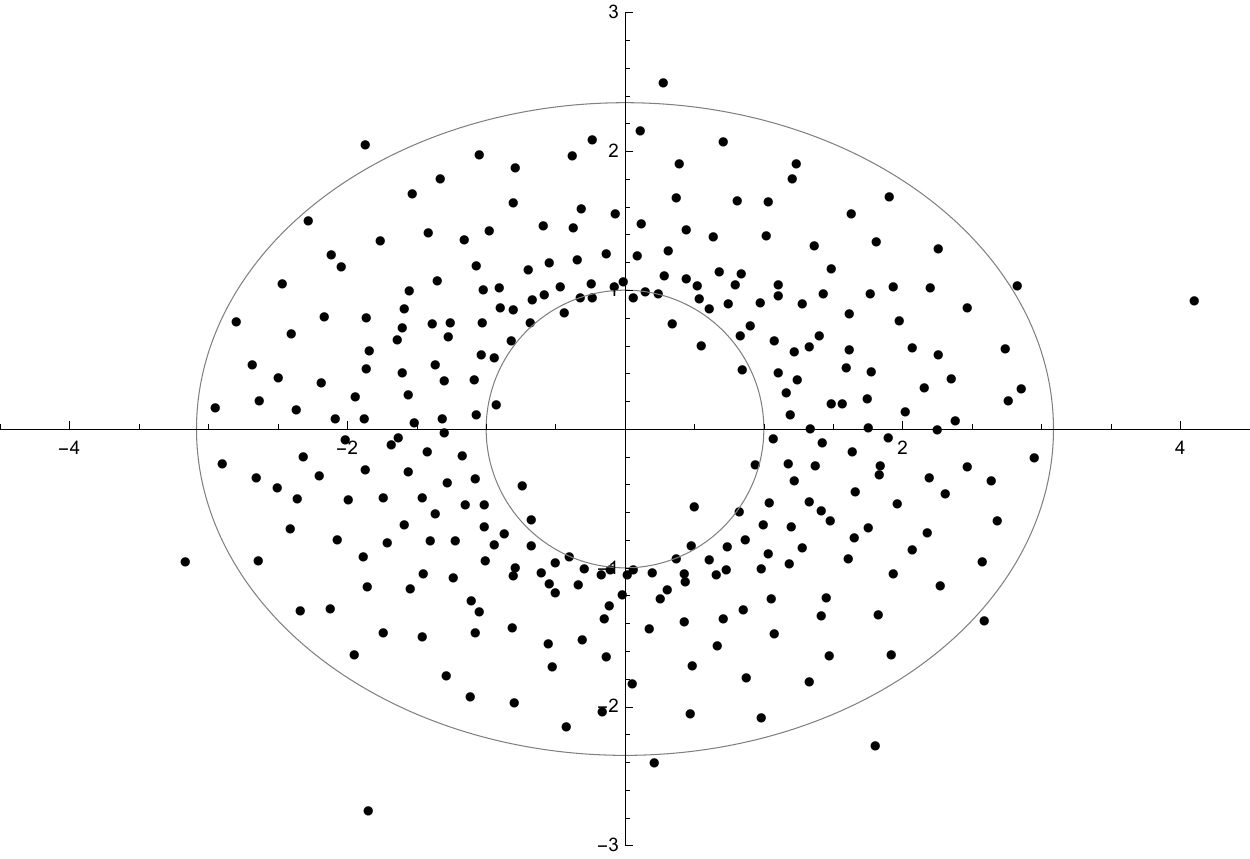}\includegraphics[width=.33\textwidth]{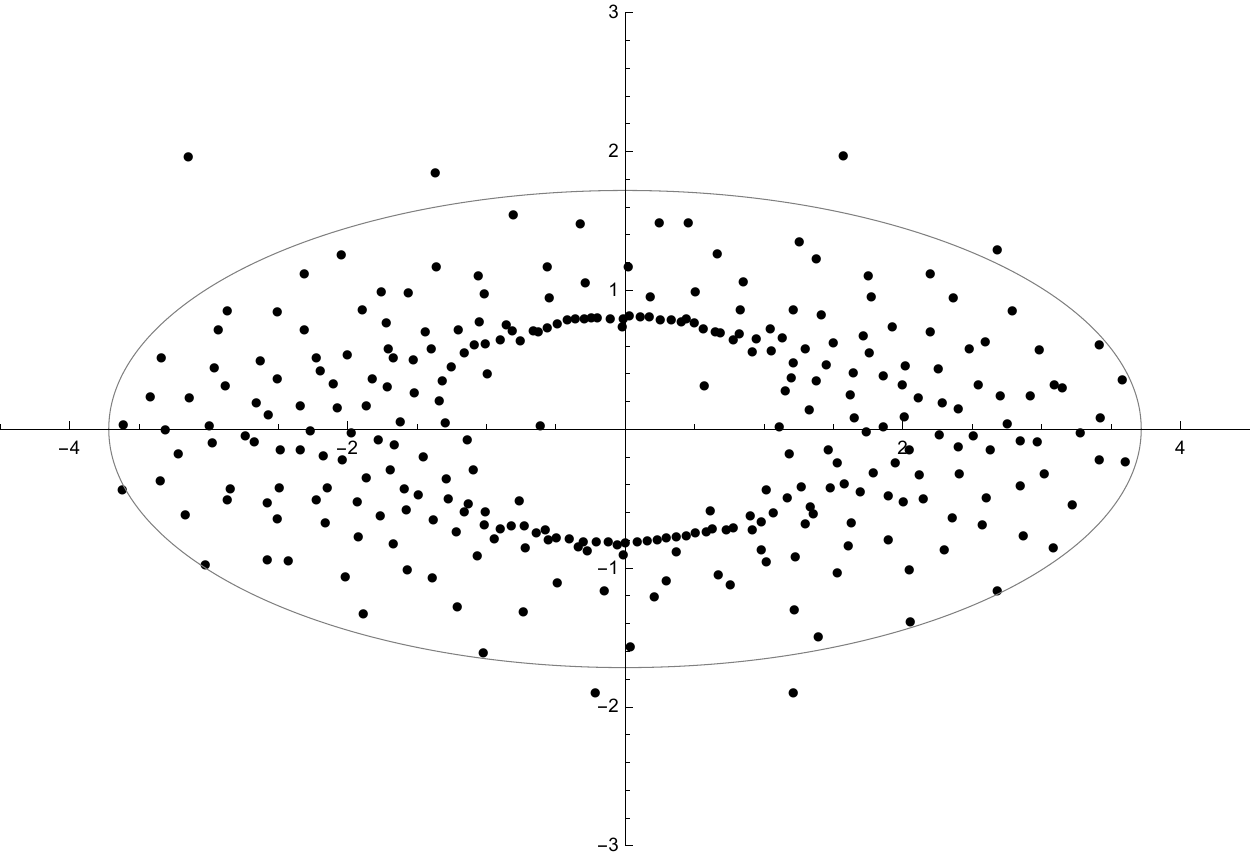}
\caption{Zeros of polynomials from Section \ref{martin.sec} undergoing the heat flow for $t=0,t=t_{\mathrm{sing}}=1$ and $t=\sqrt{t_{\mathrm{Wig}}}=\eee$.} \label{fig:zeros_on_annulus}
   \end{figure}

Consider a random polynomial $P_n(z)$ that is of the form~\eqref{eq:P_n_general_g_def} for the strictly concave function $g(\alpha)=-\alpha^2/2$, $\alpha \in [0,1]$, for example
$$
P_n(z) = \sum_{k=0}^n \xi_k \eee^{-k^2/(2n)} z^k.
$$
A similar infinite series has been studied in~\cite[\S 2.4]{KZ14}. A simple analysis using Theorem~\ref{theo:collapse_to_wigner} shows that $t_{\mathrm{Wig}}=\eee^2$. Moreover, from~\eqref{eq:alpha_0} it follows that $\alpha_0(z) = \log|z|$ for $z\in\mathcal R_0$, where $\mathcal R_0$ is the annulus $\{z\in \C: 1 < |z| < \eee\}$. By~\eqref{eq:alpha_0_measure}, $\nu_0$ is supported on $\mathcal R_0$ and its Lebesgue density is  $z\mapsto 1/(2\pi |z|^2)$.  It follows from~\eqref{eq:t_sing} that $t_{\mathrm{sing}}=1$. Theorem \ref{theo:pushforward} applies and yields the support
$$
\mathcal R_t=\left\{z=u+iv:u^2+v^2>1,\frac {u^2}{\eee^1+t\eee^{-1}}+\frac{v^2}{\eee^1-t\eee^{-1}}<1\right\}
$$
of the push-forward $\nu_t=(T_t)_{\#} \nu_0$ for $0<t<1$. Then, for $1<t<\eee^2$ singularities form at parts of the inner boundary, see Figure \ref{fig:Martin}, until the distribution collapses to the semicircle law $\mathsf{sc}_{\eee^2}$ at $t= \eee^2$.
%We omit the details and the explicit form of the density.

\subsection{Evenly spaced roots}\label{subsec:evenly_spaced_roots}
Consider the polynomials  $Q_n(z)= z^n - r^n$, with $r>0$ being constant. The roots of  $Q_n$ are asymptotically uniformly distributed on the circle $\{|z|= r\}$. Nevertheless, as we will show in the next proposition, the asymptotic distribution of zeros of $\exp\{-\frac t{2n} \partial_z^2\} Q_n$ is quite different from what we know for Kac polynomials; compare Figures~\ref{fig:examples_endroots} and~\ref{fig:evenly_spaced_example}.  Intuitively, the reason is that the roots of $Q_n$ are ``evenly spaced'' along a curve (namely, a circle). That is, for a fixed $n$, the spacings between the zeros are extremely regular (actually all exactly equal in this case). This example shows that certain conditions excluding evenly spaced roots are necessary in Conjecture~\ref{conj:Universality}.
\begin{figure}[ht]
	\centering
	\includegraphics[width=0.4\columnwidth, trim={0 40 0 40},clip]{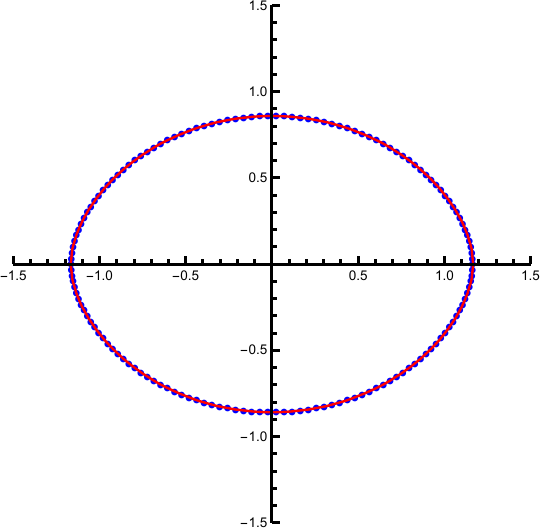}
	\includegraphics[width=0.4\columnwidth, trim={0 40 0 40},clip]{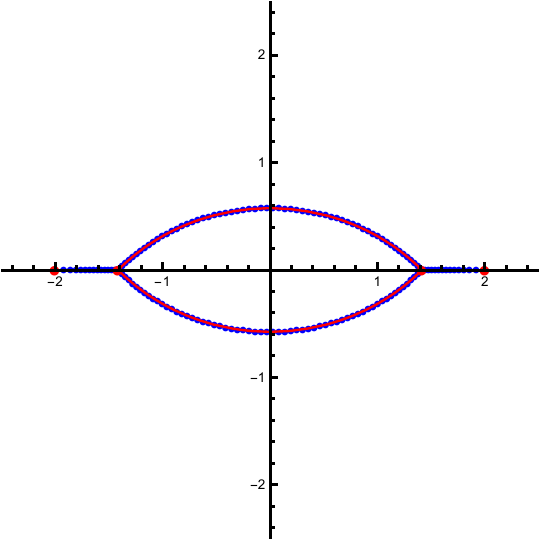}
	\caption{Zeros of $\eee^{-\frac t{2n} \partial_z^2} (z^n - 1)$ for $t<1/\eee$ (left) and $t=1\in (1/\eee,\eee)$ (right).  The level set $L_{1,t}$ is shown in red.   The degree is $n=150$.}
\label{fig:evenly_spaced_example}
\end{figure}

\begin{proposition}\label{prop:evenly_spaced_roots}
Fix some $r>0$ and $t\in \C\backslash\{0\}$. Then, as $n\to\infty$, the probability measure $\llbracket \eee^{-\frac t{2n} \partial_z^2} (z^n - r^n)\rrbracket $ converges weakly to the probability measure $\mu_{r,t}$ on $\C$ given by the distributional Laplacian
$$
\mu_{r,t} = \frac 1 {2\pi} \Delta \max\left\{\Psi\left(\frac z {\sqrt t}\right),\log \frac {r}{\sqrt{|t|}}\right\}.
$$
\begin{itemize}
\item[(a)] If $0< t\leq r^2/\eee$, then $\mu_{r,t}$ is supported on the level set
\begin{equation}\label{eq:def_level_set_L_r_t_evenly_spaced}
L_{r,t} \coloneqq \{z\in \C: \Psi(z/\sqrt t) = \log (r/\sqrt t)\}.
\end{equation}
%and its density with respect to the length measure on this level set is given by ...???
\item[(b)] If $r^2/\eee < t < r^2 \eee$, then $\mu_{r,t}$ is supported on the union of the level set $L_{r,t}$ and the non-empty intervals $[\sqrt{2 t \log (r^2 \eee /t)}, 2\sqrt{t}]$ and $[- 2\sqrt{t}, - \sqrt{2 t \log (r^2 \eee /t)}]$.
\item[(c)] If $t\geq r^2 \eee$, then $\mu_{r,t} = \mathsf{sc}_{t}$ is the Wigner law supported on the interval $[-2\sqrt {t}, +2\sqrt{t}]$.
\item[(d)] If $t = |t| \eee^{\ii \phi}\neq 0$ is complex, then $\mu_{r, t}$ is the push-forward of $\mu_{r, |t|}$ under the rotation  map $z\mapsto z \eee^{\ii \phi/2}$.
\end{itemize}
\end{proposition}

\section{Further perspectives}
\label{sec:Further_perspective}

\subsection{A PDE perspective}\label{subsec:PDE_perspective}

\subsubsection{Heuristics}\label{subsub:PDEheuristic}

Let us present a different argument for how the heuristic \eqref{eq:zjApprox} explains the definition of the transport map, that is: $t\mapsto T_t(w)$ can be seen as the characteristic curve of a Hamilton--Jacobi PDE.

Let $P_n(z;t)$ be \textit{any} polynomial (not necessarily a random polynomial) evolving according to the heat flow, as in \eqref{pnzt}.
As a simple consequence of the system of ODE's \eqref{eq:ODE_for_poly}, we can derive PDE's satisfied by the logarithmic potential $U^{(n)}_t$ and the Stieltjes transform $m^{(n)}_t$ of $\llbracket P_n(z;t)\rrbracket $, defined by
$$
U_{t}^{(n)}(z)\coloneqq \frac 1n \sum_{j=1}^n \log |z-z_j(t)|,
\qquad
m_{t}^{(n)}(z)\coloneqq \frac 1n \sum_{j=1}^n \frac 1 {z-z_j(t)}
=
2 \partial_z U_{t}^{(n)}(z).
$$
\begin{theorem}\label{theo:ODE_stieltjes_log_potential}
We have
\begin{align}
\partial_t U_{t}^{(n)}(z) + (\partial_z U_{t}^{(n)}(z))^2
&=
- \frac 1 {2n} \partial_z^2 U_{t}^{(n)}(z)
= \frac 1 {4n^2}\sum_{j=1}^n\frac{1}{(z-z_j(t))^2}, \label{eq:PDE_finite_n}\\
\partial_t m_{t}^{(n)}(z) + \frac12  \partial_z [(m_{t}^{(n)}(z))^2]
&=
- \frac 1 {2n} \partial_z^2  m_{t}^{(n)}(z)
= -\frac 1 {4n^2}\sum_{j=1}^n \frac{1}{(z-z_j(t))^3}, \label{eq:PDE_finite_n_stieltjes}
\end{align}
where $\partial_t$ and $\partial _z$ are the Wirtinger derivatives in \emph{complex} variables $t$ and $z$.
\end{theorem}

Passing formally to the $n\to\infty$ limit (cf.~\cite{hallho}), we arrive at the following PDE's for the logarithmic potential $U_t(z)$ and the Stieltjes transform  $m_t(z)$ of the measure $\nu_t$:
\begin{align}\label{eq:PDE_formal}
\partial_t U_t(z)
&=
-\left(\partial_z U_t(z)\right)^2,
\\
\partial_t m_t(z)
&=
- \frac 12\partial_z (m_t^2(z)) = - m_t(z)\partial_z m_t(z).
\label{eq:PDE_for_stieltjes_burgers_formal}
\end{align}
Let us stress that the time $t=\Re t  + \ii \Im t$ is considered as a \textit{complex} variable and that
$$
\partial_t  = \frac 12 \left(\frac{\partial}{\partial \Re t} - \ii \frac{\partial}{\partial \Im t}\right),
\qquad
\partial_{\bar{t}}  = \frac 12 \left(\frac{\partial}{\partial \Re t} + \ii \frac{\partial}{\partial \Im t}\right)
$$
are the Wirtinger derivatives in $t$.
Writing $z= x+ \ii y$ and taking the real and imaginary parts of the first  PDE gives
$$
\frac{\partial}{\partial \Re t} U_t(z) = -\frac 12  \left(\frac{\partial}{\partial x} U_t(z)\right)^2 +  \frac 12\left( \frac{\partial}{\partial y} U_t(z)\right)^2,
\qquad
\frac{\partial}{\partial \Im t} U_t(z) = -  \frac{\partial}{\partial x} U_t(z) \cdot \frac{\partial}{\partial y} U_t(z).
$$
If we are interested in real-valued times $t\in \R$ only, we obtain the Hamilton--Jacobi PDE
\begin{align}\label{eq:Re(PDE)}
\frac{\partial}{\partial t} U_t(z) = -\frac 12  \left(\frac{\partial}{\partial x} U_t(z)\right)^2 + \frac 12 \left(\frac{\partial}{\partial y} U_t(z)\right)^2,
\end{align}
matching the reasoning in \cite{hallho}.
% we expect that $U_{t}$ will satisfy the PDE???
% \[
%\frac{\partial U_{t}}{\partial t}+\frac{1}{2}\left(  \frac{\partial U_{t}%
%}{\partial x}\right)  ^{2}-\frac{1}{2}\left(  \frac{\partial U_{t}}{\partial
%y}\right)  ^{2}=0.
%\]
%
%\[
%\frac{\partial U_{t}}{\partial t}+H\left(  x,y,\frac{\partial U_{t}}{\partial
%x},\frac{\partial U_{t}}{\partial y}\right)  =0,
%\]

We denote the corresponding Hamiltonian by
\begin{equation}
H(x,y,p_{x},p_{y})=\frac{1}{2}p_{x}^{2}-\frac{1}{2}p_{y}^{2}%
\label{theHamiltonian}%
\end{equation}
and analyze this PDE by the method of characteristics, see \cite{evans_book_PDE, DHK22}. Hamilton's equations are given by %
\begin{align}
\frac{\partial x}{\partial t} &  =\frac{\partial H}{\partial p_{x}}=p_{x};\quad\frac{\partial y}%
{\partial t}=\frac{\partial H}{\partial p_{y}}=-p_{y};\nonumber\\
\frac{\partial p_{x}}{\partial t} &  =-\frac{\partial H}{\partial x}=0;\quad\frac{\partial p_{y}}%
{\partial t}=-\frac{\partial H}{\partial y}=0.\label{eq:Hamilton}%
\end{align}
Then $p_{x}$ and $p_{y}$ are constants of motion and
\begin{equation}
x(t)=x(0)+tp_{x}(0);\quad y(t)=y(0)-tp_{y}(0).\label{xtyt}%
\end{equation}
We consider specifically solutions of the system \eqref{eq:Hamilton} in which the
initial momenta are chosen as $p_{x}(0)=\frac{\partial}{\partial x}U_{0}$ and $ p_{y}(0)=\frac{\partial
}{\partial y}U_{0}$. Under this assumption, the second Hamilton--Jacobi equation tells us that, as
long as the solution remains smooth, we have %
\[
\frac{\partial U_{t}}{\partial x}(z(t))=p_{x}(t);\quad\frac{\partial
U_{t}}{\partial y}(z(t))=p_{y}(t).
\]
where we denoted $z(t)=x(t)+iy(t)$. Thus,%
\begin{equation}
2\partial_z U_{t}(z(t))=p_{x}(t)-ip_{y}(t)=\frac{\partial}{\partial t}z(t).
\label{secondHJ}%
\end{equation}

Meanwhile, by \eqref{eq:ODE_for_poly}, we expect that the zeros of $P_{n}(z,t)$ should
move approximately along solutions of the ODE%
\begin{equation}
\frac{\partial }{\partial t}z(t)=\int_{\mathbb{C}}\frac{1}{z(t)-w}~\nu_{t}(\dint w)=2
\partial_zU_{t}(z)\vert _{z=z(t)}. \label{zeroODE}%
\end{equation}
But (\ref{secondHJ}) says precisely that the characteristic curves
$z(t)$ solve this equation. Thus, using (\ref{xtyt}), we
find that the solutions to (\ref{zeroODE}) are given by
\[
z(t)=z(0)+2t\partial_zU_{t}(z)\vert _{z=z(t)}=T_t(z(0)),
\]
which agrees with the curves in idea \eqref{eq:zjApprox} and Definition \ref{def:transport_map} of the transport map.

\subsubsection{Rigorous results}

The remainder on the right-hand side of \eqref{eq:PDE_finite_n} is hard to control directly. Thus, at a rigorous level, we will take a different approach, using Theorem \ref{theo:main_general_g} and the notion of regular points from Definition \ref{def:regular_points}, to prove the following PDE results. A key idea is that the function $f_t(\alpha,z)$ appearing in Theorem \ref{theo:main_general_g} satisfies the PDE \eqref{eq:PDE_formal} for each fixed value of $\alpha$, which may be verified by direct calculation.

\begin{theorem}\label{theo:PDE}
Suppose that    \hyperref[cond:A1]{(A1)} and    \hyperref[cond:A2]{(A2)} hold and let $ \mathcal D$ be a domain of regular points. Then, %the following assertions hold:
\begin{enumerate}[(a)]
\item
The logarithmic potential $U_{t}$ of the probability measure  $\nu_{t}$  satisfies the PDE
\begin{align}\label{eq:PDE}
\partial_t U_t(z)=-\left(\partial_z U_t(z)\right)^2, \qquad (z,t)\in \mathcal D.
\end{align}
\item
The Stieltjes transform $m_t(z)$ satisfies the (complex) inviscid Burgers PDE's
\begin{align}
\partial_t m_t(z)
&=
- \frac 12\partial_z (m_t^2(z)) = - m_t(z)\partial_z m_t(z), \qquad (z,t)\in \mathcal D, \label{eq:PDE_for_stieltjes_burgers}
\\
\partial_{\overline{t}}  m_t(z)
&=
- \frac 12  \partial_z [\overline{m_{t}^2(z)}]
=
- \overline{m_t(z)} \partial_z \overline{m_t(z)},
\qquad (z,t)\in \mathcal D.
\label{eq:PDE_for_stieltjes_burgers_conjugate}
\end{align}

\item The density $p_t$ of $\nu_t$ satisfies the continuity equation
\begin{align}\label{eq:cont_eqn}
\partial_t p_t(z) = -2\partial_z [(\partial_z U_t(z))p_t(z)],\qquad (z,t)\in \mathcal D.
\end{align}

%\item  The function $\alpha_t(z)$ satisfies the PDE
%\begin{equation}\label{eq:PDE_for_alpha_t}
%\partial_t \alpha_t(z)  = - m_t(z)\partial_z %\alpha_t(z), \qquad (z,t)\in \mathcal D.
%\end{equation}
\item If    \hyperref[cond:A1]{(A1)}--   \hyperref[cond:A3]{(A3)} hold, then $U_t$ is a ``global'' $\mathcal C^1$-solution to
\begin{align}\label{eq:PDE_global}
\partial_t U_t(z)=-\left(\partial_z U_t(z)\right)^2\quad \text{ for }(z,t)\in\C\times B_{t_{\mathrm{sing}}}(0).
\end{align}
\end{enumerate}
\end{theorem}

Recall that Assumptions  \hyperref[cond:A1]{(A1)} and    \hyperref[cond:A2]{(A2)} have been defined in Section \ref{subsec:main_sec} and Assumption    \hyperref[cond:A3]{(A3)} has been added in Section \ref{subsec:transport_map}. The proofs will be given in Section \ref{subsec:proof_PDE}.

Solutions to the inviscid Burgers' equation \eqref{eq:PDE_for_stieltjes_burgers} are known to develop so-called shock waves (discontinuities), which correspond to singularities in $\nu_t$ for $t\geq t_{\mathrm{sing}}$, see Remark \ref{rem:pushforward_Burgers}. 
%In the subsequent Section \ref{subsec:free_prob}, we will discuss Burgers' equation \eqref{eq:PDE_for_stieltjes_burgers} from a free probability perspective. 
%For interesting aspects of Burgers' equation we refer to~\cite{burgers_book,vergassola,she_aurell_frisch}.
Interestingly, the complex Burgers' equation~\eqref{eq:PDE_for_stieltjes_burgers} with initial datum being a Stieltjes transform has recently been observed as the limiting equation that describes the growth of multiple SLE's, see \cite{del_monaco, hotta_katori,hotta_schleissinger}.

\begin{remark}
Let us again suppose that the time $t$ is real and write down the PDE's satisfied by $m_t(z)$. Let $x = \Re z$, $y= \Im z$ and $u_t(z) = \Re m_t(z)$, $v_t(z) = \Im m_t(z)$.  First of all, we claim that $m_t(z)$ satisfies the second Cauchy--Riemann equation, or, which is the same, that the vector field $(u_t(z), -v_t(z))$ is rotation-free:
%%%Note the minus sign!
\begin{equation}\label{eq:m_second_cauchy_riemann}
\frac{\partial u_t}{\partial y}   +  \frac{\partial v_t}{\partial x} = 0.
%\qquad (z,t)\in \mathcal D,
\end{equation}
Indeed, since $m_t(z)  = 2\partial_z U_t = (\partial/\partial x) U_t - \ii (\partial/\partial y)U_t$ and $U_t$ is real-valued, it follows that $u_t = (\partial/\partial x) U_t$ and $v_t=  - (\partial/\partial y)U_t$, which yields~\eqref{eq:m_second_cauchy_riemann}. Taking the sum of~\eqref{eq:PDE_for_stieltjes_burgers} and~\eqref{eq:PDE_for_stieltjes_burgers_conjugate}, and using the identity $\partial_{t} + \partial_{\overline{t}}  = (\partial/\partial t)$, we obtain
$$
\frac{\partial}{\partial t} m_t(z)
=
- \frac 12  \partial_z [m_{t}^2(z)]
- \frac 12  \partial_z [\overline{m_{t}^2(z)}]
=
-\partial_z \Re [m_{t}^2(z)].
$$
Expanding the right-hand side and using~\eqref{eq:m_second_cauchy_riemann} gives
\begin{align}
\frac{\partial u_t}{\partial t}
&=
- u_t \frac{\partial u_t}{\partial x} + v_t \frac{\partial v_t}{\partial x}
=
- u_t \frac{\partial u_t}{\partial x} - v_t \frac{\partial u_t}{\partial y}
 ,\label{eq:PDE_m_t_Re}\\
\frac{\partial v_t}{\partial t}
&=
+ u_t \frac{\partial u_t}{\partial y} - v_t \frac{\partial v_t}{\partial y}
=
- u_t \frac{\partial v_t}{\partial x} - v_t \frac{\partial v_t}{\partial y}.\label{eq:PDE_m_t_Im}
\end{align}
These equations are the usual Burgers' equations in dimension $2$. In order to derive the classical Hopf--Lax solution of the  Burgers' equation, see~\cite{burgers_book,vergassola,she_aurell_frisch}, one  assumes that $(u_t, v_t)$ is a gradient of some potential. This is not true in our case: instead, $(u_t, -v_t)$ is the gradient of $U_t$.

\end{remark}

\begin{remark}[Relation to the Hopf--Lax solution]\label{rem:Hopf-Lax}
Recall that \eqref{eq:Ut(Ttw)} states
\begin{align*}
U_{t}(T_t(w))&=U_{0}(w)+\frac{\alpha_0(w)^2}{2}\Re\left(\frac{t}{w^2}\right)\\
&=U_0(w) + \frac{\Re(t)} 2 \left( \left(\frac{\partial}{\partial x} U_0(w)\right)^2 - \left(\frac{\partial}{\partial y}U_0(w)\right)^2 \right)  ,
\end{align*} which can be interpreted as the first Hamilton--Jacobi formula for solutions of Hamilton--Jacobi equations with quadratic Hamiltonian, see \eqref{theHamiltonian}.
Viscosity solutions of Hamilton--Jacobi equations with convex Hamiltonians are given by the Hopf--Lax formula, see \cite[10.3.4]{evans_book_PDE}, or by the Hopf formula for convex initial value, see \cite[\S~3.3]{Tran}. In our setting, the real part of \eqref{eq:PDE}, that is \eqref{eq:Re(PDE)}, belongs to the non-convex Hamiltonian $H$ and $U_0$ is non-convex due to its logarithmic asymptotic, hence the theory does not apply. Formally however, the Hopf--Lax formula in the coordinates $z=T_t(w)$ would give $U_t(T_t(w))=\inf_{\xi\in\C}(U_0(\xi)+\frac 1 {2t}\Re(\xi-T_t(w))^2)$, where the infimum does not exist due to non-convexity. Interestingly, its only saddle point is $\xi=w$, hence \eqref{eq:Ut(Ttw)} can be seen as the analogue of the Hopf--Lax formula. To the best of our knowledge, even the theory for non-convex Hamiltonians does not apply, since the Hamiltonian does not diverge to $+\infty$ (which would be necessary for optimal control solutions), our initial data $U_0$ is neither bounded nor has bounded gradient and the result \cite{Evans14} would only hold almost everywhere. Viscosity solutions to the PDE \eqref{eq:PDE} lie out of the scope of this paper and will be discussed in a forthcoming work.
\end{remark}

\subsection{An optimal transport perspective}\label{subsec:OT}

Let us define the \textit{logarithmic energy} of a collection of $n$ distinct
points $z_1,\ldots,z_n$ in $\mathbb{C}$ as
\[
E(z_{1},\ldots,z_{n})=\frac{1}{n^{2}}\sum_{k,l:~k\neq l}\log\left\vert
z_{k}-z_{l}\right\vert .
\]
Then an elementary calculation shows that the equation of motion \eqref{eq:ODE_for_poly} for
the roots of a polynomial under the heat flow can be rewritten as
\[
\partial_{t}z_{j}(t)=\frac{1}{n}(\partial_{z_{j}}E)(z_{1}(t),\ldots,z_{n}(t)).
\]
That is, the heat evolution of the roots is a sort of gradient flow with
potential energy given by the logarithmic energy and the Wirtinger derivatives
$\partial_{z_{j}}$ playing the role of the gradient. 

This calculation suggests that the evolution of the measures $\nu_{t}$ could
be seen as a sort of gradient flow on the Wasserstein space with respect to
the continuous version of the logarithmic energy $\iint\log\left\vert
z-y\right\vert ~\nu(dy)~\nu(dz).$ (This energy function is also known as the
free entropy \cite{Voiculescu93} in the context of free probability theory.)
A classical result of Jordan-Kinderlehrer-Otto \cite{JKO} states that the heat flow on the set of absolutely
continuous measures can be interpreted as the Wasserstein gradient flow with respect to the relative
entropy functional, see also \cite{Otto,erbar,Villani}. After all, such gradient flows are given by solutions to the continuity equation---the classical counterpart to what we have shown in \eqref{eq:cont_eqn}.

The preceding discussion leads naturally to questions about $t\mapsto\nu_{t}$
seen as a curve of distributions in the Wasserstein space, which we
investigate in the remainder of this section. 

It is well known in optimal transport that push-forwards under a map $T_t=(1-t)\mathrm{Id}+tT$ are most natural candidates for optimal transport maps between their corresponding marginals. In the following, we shall study $\nu_{t}=(T_t)_{\#}\nu_0$, $t\in[0,1]$, as a curve of distributions in the Wasserstein space. Let us recall some notions of optimal transport, see  \cite{AmbrosioGigli,Villani} for more information.

The Wasserstein distance $W_2$ between two compactly supported distributions $\mu$ and $\nu$ on $\C$ is defined via the optimal transport problem
\begin{align}\label{eq:Wasserstein}
 W^2_2(\mu,\nu)=\inf_{q\in\Pi}\int_{\C^2} |z-w|^2 q(\dint z, \dint w),
\end{align}
where $\Pi$ is the set of all couplings (or, transport plans) between $\mu$ and $\nu$.
If there exists an optimal coupling achieving the infimum in \eqref{eq:Wasserstein} and it is of the form $q_T=(\mathrm{Id},T)_\#\mu$, then we call $T$ an \emph{optimal transport map}. %The optimization problem \eqref{eq:Wasserstein} becomes the so called Monge Problem. %If the cost is given by the square of a distance, then \eqref{eq:Wasserstein} defines the $2$-Wasserstein metric (squared) and the space of distributions endowed with this metric becomes geodesic.
Existence and uniqueness follow from Brenier's Theorem \cite[Theorem 1.26]{AmbrosioGigli}.
\begin{theorem}[Brenier]\label{thm:Brenier}
Let $\mu,\nu$ be distributions on $\R^d$ with finite variance and let $\mu$ be regular, i.e. it does not charge hypersurfaces. Then, there exists a unique optimal transport map $T$ between $\mu$ and $\nu$, and $T$ is the gradient of a convex function $\Phi$.
\end{theorem}
 %For more on the optimal transport in Lorentzian spaces, we refer to \cite{CM20}.
A curve of distributions $(\mu_t)_{t\in[0,1]}$ is called (constant speed) Wasserstein geodesic if $W_2(\mu_s,\mu_t)=(t-s)W_2(\mu_0,\mu_1)$ for all $s,t\in [0,1]$. In fact, if $T$ is an optimal transport map between $\mu_0$ and $\mu_1$, then $\mu_t=((1-t)\mathrm{Id}+tT)_\# \mu_0$, $t\in[0,1]$, is a Wasserstein geodesic, see \cite[Theorem 2.10]{AmbrosioGigli} and Remark 2.13 therein. The following theorem shows that the heat flow gives rise to a transport map $T_t$ that transports the limiting root distribution of random polynomials to the semicircle law in an optimal way if and only if we started with Weyl polynomials (or those having the same root distribution).

%Moreover, if $T=T_1$ is injective, then it gives rise to an optimal transport plan $q_T=(\mathrm{Id},T)_\#\nu_{0}$ minimizing $\mathbf C_L(\nu_{0},\nu_{1})$ and hence (???), $(\nu_{t})_{t\in [0,1]}$ is a 2-Lorentzian-Wasserstein geodesic.

\begin{theorem}\label{theo:OT}
Let $t\in\C$ and $\nu_t$ be given as in Theorem \ref{theo:pushforward}. The transport map $T_t$ is optimal between $\nu_0$ and $\nu_t$ if and only if $\nu_0$ is the uniform measure on some centered disk, in which case the curve $(\nu_{t})_{t\in[0,1]}$ (in the space of probability measures) is a Wasserstein geodesic.
\end{theorem}

The proofs will be given in Section \ref{subsec:proof_OT}.

\begin{remark}
It would be interesting to see if replacing the cost functions $c(z,w)=|z-w|^2$ in the optimal transport problem could lead to an interpretation of  $\nu_{t}$ (with general isotropic $\nu_0$) as a geodesic  with respect to a different optimization problem.
%but we were unable to find a natural non-trivial cost function.
In view of the negative sign of $\partial_z$ in imaginary direction, one might be inclined to change the sign of the metric, or cost function, in the imaginary direction as well. The resulting cost function would be the Lorentz (or Minkowski) metric $\Re\big((z-w)^2\big)=\Re(z-w)^2-\Im(z-w)^2$.
This suggestion is in accordance with the fact that the Brenier map can be expressed as $T(w)=\nabla\Phi(w)=\nabla\big(\tfrac 1 2|w|^2-\phi(w)\big)$ for some $c$-concave Kantorovich potential $\phi$, while our transport map analogously reads $T_1(w)=2\partial_w\big(\tfrac 1 2\Re(w^2)+U_0(w)\big)$.
Then, our PDE \eqref{eq:PDE_formal} for the logarithmic potential is a Lorentzian version of the geodesic PDE for the Kantorovich potential, see \cite{lott,OttoVillani,Villani}.

However, in all our examples the Lorentz cost would only decrease if we replace the map $w\mapsto T_t(w)$ by its reflected version  $w\mapsto \overline{T_t(w)}$. Therefore, $T_t$ can only be optimal with respect to the Lorentz cost function if the target measure is concentrated on the real line (in which case, the optimal transport plans are given by the optimal transport plans of the Wasserstein space with transport costs differing only by a constant).\footnote{Even though optimality fails, it is still possible to follow the idea of duality (see \cite[\S 1.6.4]{Santa}), with the only difference of replacing the infimum in the definition of the $c$-conjugate \cite[Definition 1.10]{Santa} with a saddle point. We omit the details.}
\end{remark}

\section{A free probability interpretation of the heat flow}
\label{sec:free_prob}

In this section, we take $t>0$ for simplicity. In light of Theorem
\ref{theo:main_weyl}(iii) there is no real loss of generality in making this assumption.

\subsection{Background}

We continue to let $\nu_{0}$ denote the limiting root distribution of the
random polynomials $P_{n}$ and to let $\nu_{t}$ denote the limiting root
distribution of the heat-evolved polynomials $P_{n}(z,t).$ In this section, we
will give a free probability interpretation of $\nu_{t}$ as the Brown measure
of a certain element in a tracial von Neumann algebra, assuming that $\nu_{0}$
has a certain form.

Let us begin with an overview of recent developments on the evolution of zeros
of polynomials under the action of differential operators---and the free
probability interpretation of such results. (The history of this research
direction can be traced back to at least the Gauss--Lucas Theorem
\cite[Theorem (6,1)]{Marden}, stating that the roots of the derivative of any
polynomial are in the convex hull of the roots of the original polynomial.)
Let us now take a random polynomial $Q_{n}(z)=\prod(z-z_{j})$ with
i.i.d.~roots $z_{j}\sim\mu_{0}.$ In that case, the empirical measure of the
roots converges to $\mu_{0}$. It was shown in~\cite{K15} that also the
empirical distribution of roots of $Q_{n}^{\prime}$ converges to $\mu_{0}$.
This result was extended to higher derivatives of fixed order in
\cite{Byun,MV23} and negligibly growing order in \cite{MV22}, see also
\cite{Sub12,PR13,O16,OW19,Angst}.

The limiting distribution, however, changes in an interesting way when the
order of differentiation is comparable to the degree of the polynomial. For
every $n\in\mathbb{N}$ consider a \emph{real-rooted} polynomial $Q_{n}(z)$ of
degree $n$ such that $\llbracket Q_{n}\rrbracket\Rightarrow\mu_{0}$ for some
probability measure $\mu_{0}$ on $\mathbb{R}$ and such that all zeros of all
$Q_{n}$'s are in some fixed interval $[-C,C]$. It was shown
in~\cite{Steiner21,HK21,arizmendi_garza_vargas_perales} that for the
differentiation of order $k\sim tn$ with $0<t<1$ it holds that
$\llbracket\partial_{z}^{k}Q_{n}\rrbracket$ converges to $\mu_{0}%
^{\boxplus\frac{1}{1-t}}(\frac{\cdot}{1-t})$ in probability. Here,
$\mu\boxplus\nu$ denotes the free (additive) convolution of the probability
distributions $\mu$ and $\nu$; see~\cite{voiculescu_nica_dykema_book}. We then
write $\mu^{\boxplus\ell}=\mu\boxplus\dots\boxplus\mu$ for the free
convolution of $\ell\in\mathbb{N}$ identical distributions, which is known to
have an extension to arbitrary $\ell\in\mathbb{R}$ with $\ell\geq1.$

Repeated differentiation of order $k\sim tn$ of random polynomials $P_{n}$
satisfying \hyperref[cond:A1]{(A1)} and \hyperref[cond:A2]{(A2)}---which, of
course, have complex roots---has been studied in several papers
\cite{FengYao19,OSteiner,HK21,diff-paper}. Recently, a free probability
interpretation of these results has been given by Campbell, O'Rourke, and
Renfrew \cite{COR23}. Specifically, let the limiting root distribution of the
``squared'' polynomial $z\mapsto P_{n}(z^{2})$ be $\mu_{0}.$ Then the limiting
root distribution $\mu_{t}$ of the squared polynomial $(\partial_z^{tn}P_{n})(z^{2})$ is equal 
to $\mu_{0}^{\oplus\frac{1}{1-t}}( \frac{\cdot}{1-t}) ,$
where $\oplus$ refers to the (fractional) free convolution of radial measures
on $\mathbb{C},$ which comes from the operation of adding freely independent
$R$-diagonal elements. (For repeated differentiation of polynomials with
complex roots restricted to a fixed finite set of points we refer
to~\cite{bogvad_etal}.)

Let us now turn to the action of the heat-flow operator, which comprises
\emph{all} powers of $\partial_{z}^{2}$, while giving weights $1/j!$ to the
$j$-th power. The zero distribution of heat-evolved polynomials is the subject
of \cite{tao_blog1,tao_blog2,csordas_smith_varga,
MarcusFPP,GAF-paper,rodgers_tao,hallho,ZakharLeeYang}, most of which we
addressed already in detail. If all roots of all $Q_{n}$'s are real and
contained in some interval $[-C,C]$, it follows from \cite[Theorem
2.11]{ZakharLeeYang} and \cite[Theorem 1.1]{VW22} that for all $t>0$ the
empirical distribution of zeros of the heat-evolved polynomials $e^{-\frac
{t}{2n}\partial_{z}^{2}}Q_{n}(z)$ converges weakly to the free additive
convolution $\mu_{0}\boxplus\mathsf{sc}_{t}$ of $\mu_{0}$ and the Wigner law
on the interval $[-2\sqrt{t},2\sqrt{t}].$\footnote{The free convolution with the Wigner distribution also appears as the
hydrodynamic limit of Dyson's Brownian motion with randomness, see
\cite[\S ~4.3.1, 4.3.2]{AGZ} and \cite{VW22,tao_blog1,ZakharLeeYang}.} See also \cite{Marcus21, Mirabelli}
for a connection to finite free probability: The heat-evolved polynomial is
just the finite free convolution of the original polynomial with a Hermite polynomial.

We further note that, by a result of Voiculescu~\cite[Equation after
(3.18)]{voiculescu_nica_dykema_book}, the Stieltjes transform of $\mu
_{0}\boxplus\mathsf{sc}_{t}$ solves the so called \emph{free heat equation}
\begin{equation}
\frac{\partial}{\partial t}m_{\mu_{0}\boxplus\mathrm{sc}_{t}}(z)=-m_{\mu
_{0}\boxplus\mathrm{sc}_{t}}(z)\partial_{z}m_{\mu_{0}\boxplus\mathrm{sc}_{t}%
}(z),\quad t\geq0, \label{eq:free_heat_equation}%
\end{equation}
for $z$ outside the support of $\mu_{0}\boxplus\mathsf{sc}_{t}$; see also
\cite{Biane,craig_book,RogersShi,Uchiyama, voiculescu_nica_dykema_book}. The
name \textquotedblleft free heat equation\textquotedblright\ comes from the
fact that the semicircle distribution is the free probability analogue of the
Gaussian distribution, whose classical convolution with any initial real
distribution satisfies the classical heat equation.\footnote{Recall Section \ref{subsec:OT}, where we argued from a different perspective how the heat evolution of roots is connected to the free counterpart of the heat equation: a `gradient' flow with respect to the \emph{free} entropy.}

We now compare Voiculescu's PDE (\ref{eq:free_heat_equation}) to the PDE
satisfied by the Stieltjes transform $m_{t}(z)$ of $\nu_{t}.$ We note that
$m_{t}$ is holomorphic in $z$ outside the support of $\nu_{t}$ (since $\nu
_{t}$ is the $\bar{z}$-derivative of $m_{t}$). Then by
\eqref{eq:PDE_for_stieltjes_burgers_conjugate}, $m_{t}(z)$ is also holomorphic
in $t$ outside the support of $\nu_{t}.$ It then follows from
\eqref{eq:PDE_for_stieltjes_burgers} that if we restrict to $t>0,$ the
function $m_{t}(z)$ will satisfy Voiculescu's PDE (\ref{eq:free_heat_equation}%
) outside the support of $\nu_{t}.$ What is notable about
\eqref{eq:PDE_for_stieltjes_burgers} and
\eqref{eq:PDE_for_stieltjes_burgers_conjugate}, however, is that they apply
also to (regular) points \textit{inside} the support of $\nu_{t}.$

\subsection{Free probability interpretation of the evolution of $\nu_{t}$}

In general, the Brown measure of a noncommutative random variable does not
determine the $\ast$-distribution of that element. On the other hand, there is
natural class of noncommutative random variables, known as $R$-diagonal
elements, that have radial Brown measures and for which
the $\ast$-distribution \textit{is} determined by the Brown measure. (See the original article
of Nica and Speicher \cite{nicaspeicher} or Lecture 15 of their book \cite{nicaspeicherbook}.) 
In light
of the preceding discussion of the behavior of real-rooted polynomials under
heat flow, one might make the following guess for the behavior of random
polynomials: If the limiting root distribution $\nu_{0}$ of $P_{n}$ is the
Brown measure of an $R$-diagonal element $a_0$, then the limiting root
distribution $\nu_{t}$ of the heat evolved polynomial $P_{n}(z,t)$ will be the Brown measure of $a_0+x_{t},$
where $x_{t}$ is a semicircular element of variance $t$ that is freely
independent of $a_0$:%
\begin{equation}
\nu_{t}\overset{??}{=}\mathrm{Brown}(a_0+x_{t})\quad\text{(initial guess).}
\label{guess}%
\end{equation}
But this guess is not correct, even in the case of the Weyl polynomials, in
which case the $R$-diagonal element $a_0$ may be taken to be circular of
variance 1. In this case, both the limiting root distribution of $P_{n}(z,t)$
and the Brown measure of $a_0+x_{t}$ are both uniform on an ellipse---but not the
same ellipse. (See Example \ref{FreeWeyl.example} below.)

There is, nevertheless, a similar statement that can be made, as follows.

\begin{theorem}
\label{freeInterpret.thm} Let $\nu_{0}$ be the limiting root
distribution of the random polynomials $P_{n}$ satisfying \hyperref[cond:A1]{(A1)} and \hyperref[cond:A2]{(A2)}. Suppose there is some $s>0$ such that $\nu_{0}$ is the Brown
measure of an element $a_{s,0}$ of the form%
\begin{equation}
a_{s,0}:=b+x_{s/2}+iy_{s/2},\label{as0}%
\end{equation}
where $b$ is $R$-diagonal, $x_{s/2}$ and $y_{s/2}$ are semicircular with
variance $s/2,$ and where $b,$ $x_{s/2},$ and $y_{s/2}$ are freely independent. Then condition
\hyperref[cond:A3]{(A3)} also holds, the quantity $t_{\mathrm{sing}}$ in Proposition \ref{prop:T_t_bijective} is at
least $s$, and for all $t$ with $0<t<s,$ the measure $\nu_{t}$ is the Brown
measure of the element%
\begin{equation}
a_{s,t}=b+x_{(s+t)/2}+iy_{(s-t)/2},\label{zst}%
\end{equation}
where $x_{(s+t)/2}$ and $y_{(s-t)/2}$ are semicircular with the indicated
variances, assumed to be freely independent of each other and $b.$
Furthermore, if we are given an element $a_{s,0}$ of the form in (\ref{as0}), we can
choose the polynomials $P_{n}$ so that $\nu_{0}$ equals the Brown measure of
$a_{s,0}$.
\end{theorem}

The proof of this result is given in Section \ref{proof61thm}.

In what follows, we make use of the following result \cite[Example
5.3]{BianeLehner}: if $x_{\alpha}$ and $y_{\beta}$ are freely independent
semicircular elements with variances $\alpha$ and $\beta,$ the Brown measure
of $x_{\alpha}+iy_{\beta}$ is uniform on the
ellipse centered at the origin with semi-axes $2\alpha/\sqrt{\alpha+\beta}$
and $2\beta/\sqrt{\alpha+\beta}.$ We also note that the quantity
$x_{s/2}+iy_{s/2}$ in (\ref{as0}) is a circular element of variance $s,$ which
is $R$-diagonal and has a Brown measure that is uniform on a disk of radius
$\sqrt{s}.$ Since $a_{s,0}$ is the sum of two freely independent $R$-diagonal
elements, it is $R$-diagonal and its Brown measure can be computed by a result
of Haagerup--Larsen \cite[Proposition 3.5]{haageruplarsen}. Meanwhile, $x_{(s+t)/2}+iy_{(s-t)/2}$
is an elliptic element whose Brown measure is uniform on the ellipse with
semi-axes $(s+t)/\sqrt{s}$ and $(s-t)/\sqrt{s}.$

\begin{example}
\label{FreeWeyl.example}We consider the case of the Weyl polynomials and
compare the results of Theorem \ref{freeInterpret.thm} to our initial guess in
\eqref{guess}. In the Weyl polynomial case, we may take $b=0$ and $s=1$ in
(\ref{as0}). By Theorem \ref{theo:main_weyl}, the limiting root
distribution of $P_{n}(z,t)$ is, for $0<t<1,$ uniform on an ellipse with
semi-axes $1+t$ and $1-t.$ By contrast, in our initial guess (\ref{guess}), we
would take the element $a_0$ to be circular; that is, $a_0=x_{1/2}+iy_{1/2}$. Then
$a_0+x_{t}$ has the same distribution (after combining the \textquotedblleft%
$x$\textquotedblright\ terms) as $x_{1/2+t}+iy_{1/2}.$ Thus, the Brown measure
of $a_0+x_{t}$ is uniform on the ellipse with semi-axes $(1+2t)/\sqrt{1+t}$ and
$1/\sqrt{1+t}.$ In particular, the limiting root distribution of $P_{n}(z,t)$
collapses onto the $x$-axis as $t$ approaches 1, while the Brown measure of
$a_0+x_{t}$ only collapses onto the $x$-axis asymptotically as $t\rightarrow
\infty.$
\end{example}

In general, we may compare our initial guess in (\ref{guess}) to the result of
Theorem \ref{freeInterpret.thm}, if we assume that the element $a_0$ in
(\ref{guess}) has the form
\[
a_0=a_{s,0}=b+x_{s/2}+iy_{s/2}.
\]
In this case, $a_0+x_{t}$ has the same $\ast$-distribution as $b+x_{s/2+t}%
+iy_{s/2}.$ We are, therefore, led to compare the following two elements:%
\begin{align}
a_{s,t} &  =b+x_{(s+t)/2}+iy_{(s-t)/2}\label{element1}\\
b_{s,t} &  =b+x_{s/2+t}+iy_{s/2}.\label{element2}%
\end{align}
Although these elements do not have the same Brown measure, they are related
as follows.

\begin{proposition}
\label{twoElements.prop}For any element $b$ that is freely independent of $x_\cdot$
and $y_\cdot$ (but not necessarily $R$-diagonal), the analytic moments of $a_{s,t}$
equal the analytic moments of $b_{s,t}$: $\tau(a_{s,t}^{k})=\tau(b_{s,t}%
^{k}),$ for all $0<t<s$, where $\tau$ is the trace on the relevant von Neumann algebra. 
If, in addition, $b$ is $R$-diagonal, these analytic
moments are equal to the moments of the semicircular distribution of variance
$t$. That is,
\begin{equation}
\tau(a_{s,t}^{k})=\tau(b_{s,t}^{k})=\int_{\mathbb{R}}u^{k}~d\mathsf{sc}%
_{t}(u)=t^{k/2}\int_{\mathbb{R}}u^{k}~d\mathsf{sc}_{1}(u).\label{momentAnswer}%
\end{equation}
In this case,
both elements have the same Stieltjes transform near infinity as
$\mathsf{sc}_{t}$ and therefore both Stieltjes transforms satisfy Voiculescu's
PDE (\ref{eq:free_heat_equation}) near infinity.
\end{proposition}

\begin{proof}
If two elements $A$ and $B$ are freely independent, the analytic moments of
$A+B$ can be expressed as a polynomial function of the analytic moments of $A$
and of $B.$ Thus, to prove that $\tau(a_{s,t}^{k})=\tau(b_{s,t}^{k}),$ it
suffices to prove that $c=x_{(s+t)/2}+iy_{(s-t)/2}$ and $\,\,d=x_{s/2+t}%
+iy_{s/2}$ have the same analytic moments. To prove this, we will show that
the analytic moments of $x_{\alpha}+iy_{\beta}$ for $\alpha\geq\beta\geq0$ are
the same as the those of $x_{\alpha-\beta}.$ Once this is established, we will
simply note that $\alpha-\beta$ has the same value (namely $t$) for both $c$
and $d.$

To compute the analytic moments of $x_{\alpha}+iy_{\beta},$ we note that%
\begin{equation}
x_{\alpha}+iy_{\beta}\overset{d}{=}x_{\beta}+iy_{\beta}+u_{\alpha-\beta}%
\label{ellipticAndCircular}%
\end{equation}
where $x_{\beta},$ $y_{\beta},$ $u_{\alpha-\beta}$ are freely independent
semicircular elements with the indicated variances and $\overset{d}{=}$
denotes equality in $\ast$-distribution. Then we note that $x_{\beta
}+iy_{\beta}$ is circular, so that its analytic moments are the same as those
of the zero element. Thus, by the observation at the beginning of this proof,
the circular term on the right-hand side of (\ref{ellipticAndCircular}) does
not affect the analytic moments of $x_{\alpha}+iy_{\beta}.$

Finally, if $b$ is $R$-diagonal, its analytic moments are zero and the $b$
term does not affect the analytic moments of $a_{s,t}$ or $b_{s,t}.$ The
analytic moments of the remaining two terms is computed (for $t>0$) by the
result in the first paragraph of the proof as the moments of $x_{t}$, as claimed.
\end{proof}

\subsection{A random matrix model}

\begin{figure}[ptb]
\centering
\includegraphics[scale=0.33]
{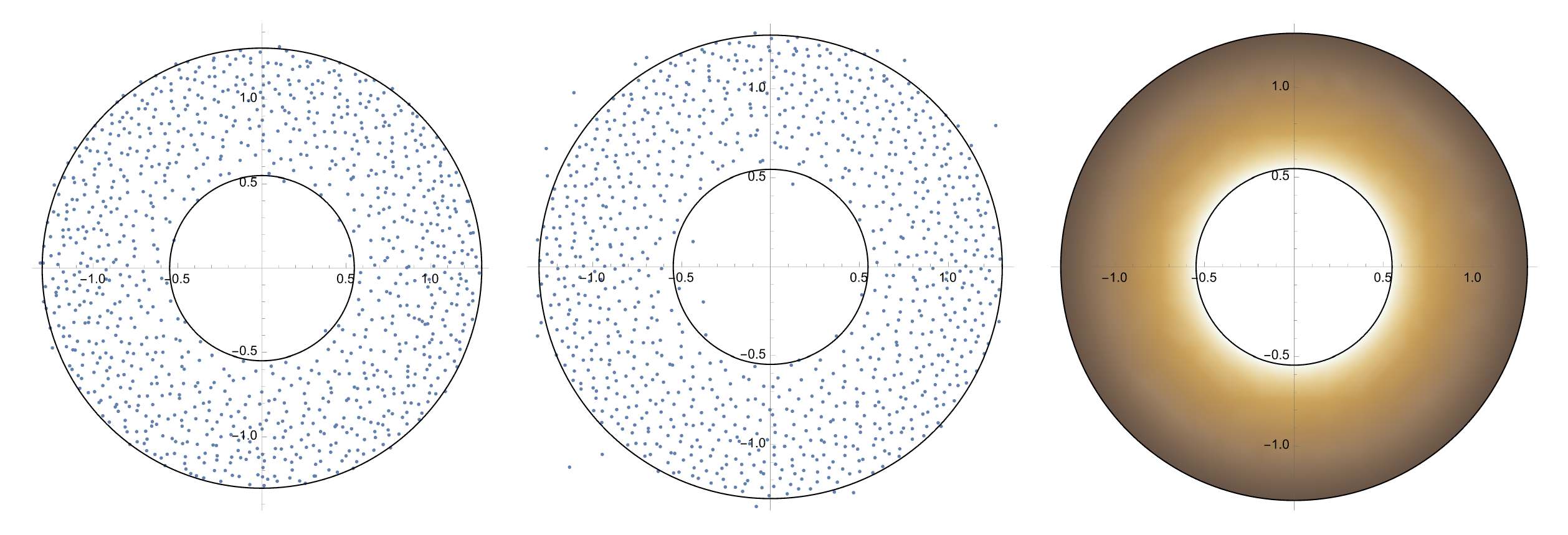}\caption{The eigenvalues of $A_{n,s,0}$ (left), the roots of
$P_{n}$ (middle), and the Brown measure of $a_{s,0}$ (right), with $B_n$ Haar
unitary and $s=0.7.$}%
\label{triptych1.fig}%
\end{figure}

Using results of \'{S}niady \cite{Sniady}, one can show that the Brown measure
of the element $a_{s,t}$ in (\ref{zst}) is the almost-sure limiting eigenvalue
distribution of a random matrix model of the form%
\begin{equation}
A_{n,s,t}=B_{n}+\sqrt{(s+t)/2}\,X_{n}+i\sqrt{(s-t)/2}\,Y_{n},\quad0<t<s,
\label{Znst}%
\end{equation}
where $B_{n}$ is an $n\times n$ bi-unitarily invariant random matrix that
converges almost surely in $\ast$-moments to $b$ and
where $X_{n}$ and $Y_{n}$ are GUEs, with $B_{n},$ $X_{n}$, and $Y_{n}$ taken
to be independent. (The elliptic matrix $\sqrt{(s+t)/2}\,X_{n}+i\sqrt
{(s-t)/2}\,Y_{n}$ can be expressed as the sum of a Ginibre matrix and another
elliptic matrix independent of the Ginibre matrix. Thus, Theorem 6 in
\cite{Sniady} is applicable.)

The heat flow conjecture of \cite[Conjecture 2.13]{hallho}---see also Conjecture \ref{conj:Universality} in the present paper---asserts that
applying the time-$t$ heat flow to the characteristic polynomial of
$A_{n,s,0}$ should give a new polynomial whose limiting root distribution
coincides with the limiting eigenvalue distribution of $A_{n,s,t},$ for
$-s<t<s.$ As a special case, when $B_{n}=0$ and $s=1,$ applying the time-$t$
heat flow to the characteristic polynomial of a Ginibre matrix should give a
polynomial whose roots are asymptotically uniform on the ellipse with
semi-axes $1+t$ and $1-t,$ for $0<t<1.$ While we currently do not know how to
prove this conjecture, Theorem \ref{freeInterpret.thm} provides a rigorous
random polynomial analogue of it. That is to say, we may choose the
polynomials $P_{n}$ so that their limiting root distribution equals the
limiting eigenvalue distribution of $A_{n,s,0}$ and then Theorem
\ref{freeInterpret.thm} says that the limiting root distribution of
$P_{n}(z,t)$ will equal the limiting eigenvalue distribution of $A_{n,s,t},$
for $0<t<s.$ See Figures \ref{triptych1.fig} and \ref{triptych2.fig}.

\begin{figure}[ptb]
\centering
\includegraphics[scale=0.32]{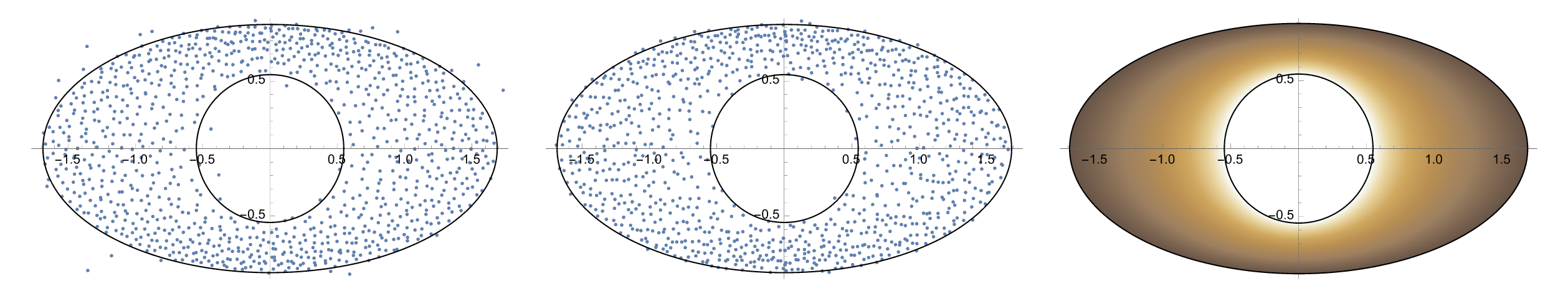}\caption{The eigenvalues of
$A_{n,s,t}$ (left), the roots of $P_{n}(z,t)$ (middle), and the Brown measure
of $a_{s,t}$ (right), with $B_n$ Haar
unitary, $s=0.7$, and $t=0.5.$}%
\label{triptych2.fig}%
\end{figure}

\subsection{The Haar unitary case\label{HaarUnitaryCase.sec}}

We now consider the example in which the element $b$ in (\ref{as0}) is a Haar unitary $u.$ In that case,
the Brown measure $\mu^{s}$ of $a_{s,0}$ has been computed explicitly by Zhong in
Section 8 of \cite{Zhong}. The measure $\mu^{s}$ is supported on the annulus
with inner radius $r_{1}=\sqrt{\max({0,1-s)}}$ and outer radius $r_{2}%
=\sqrt{1+s}$ and the radial cumulative distribution function of $\mu^{s}$ is
given in \cite[Theorem 8.8]{Zhong} as%
\begin{equation}
\mu^{s}(\{z\in\C: |z|\le r\})=\frac{r^{2}}{s}+\frac{1}{2}-\frac{\sqrt{4r^{2}+s^{2}}}{2s}\label{HaarCDF}
\end{equation}
for $r_{1}\leq r\leq r_{2}.$ We can invert the radial CDF to find the radial
quantile function $Q^{s}$ of $\mu^{s}$ as%
\[
Q^{s}(\alpha)=\frac{1}{\sqrt{2}}\sqrt{1-s+2s\alpha+\sqrt{1-2s+s^{2}+4s\alpha}%
}.
\]
Then, in accordance with \eqref{eq:alpha_0}, the \textit{derivative} of the exponential profile $g_{s}$ that produces
the measure $\mu^{s}$ is the negative of the logarithm of $Q^{s}$:
\[
g_{s}^{\prime}(\alpha)=-\frac{1}{2}\log\left[  1+s(2\alpha-1)+\sqrt
{(s-1)^{2}+4s\alpha}\right]  +\frac{1}{2}\log2.
\]
We can then integrate to obtain $g_{s}(\alpha)$ itself by first integrating by
parts (with $u=g_{s}^{\prime}(\alpha)$ and $dv=d\alpha$) and then making the
substitution $y=\sqrt{(s-1)^{2}+4s\alpha}$.
The result is
\begin{align*}
g_{s}(\alpha)  &  =\frac{1}{8s}\Big(4s\alpha-2y-y^{2}\log\left[
\frac{1}{2}\left(  1+y+s(2\alpha-1\right)  \right] \\
&  +(s-1)^{2}\log(1-s+y)+(s+1)^{2}\log(1+s+y)\Big){\Big .}.
\end{align*}

The limiting root distribution of the polynomials $P_{n}$ computed with this
exponential profile is, by construction, the Brown measure $\mu^{s}$ of the
element $a_0$ in (\ref{as0}), with $b$ taken to be a Haar unitary. The measure
$\mu^{s}$ is also the limiting eigenvalue distribution of the random matrix
model in (\ref{Znst}), with $t=1$ and $B_{n}$ taken to be a Haar-distributed
unitary matrix. See Figure \ref{triptych1.fig}.

When $u$ is a Haar unitary, the Brown measure $\mu_{t}^{s}$ of the element $a_{s,t}$ in (\ref{zst}) is
the push-forward of $\mu^{s}$ under a transport map $T_{t}^{s},$ given
explicitly as%
\[
T_{t}^{s}(w)=w+\frac{t}{s}\bar{w}\left(  1-\frac{\sqrt{4\left\vert
w\right\vert ^{2}+s^{2}}-s}{2\left\vert w\right\vert ^{2}}\right)  .
\]
(Apply Eq. (8.19) in \cite{Zhong} with $\alpha=0,$ $t=s,$ $\gamma=t,$ and
$s=\left\vert w\right\vert ^{2}-1.$) Theorem \ref{freeInterpret.thm} says that
if $P_{n}$ is chosen with exponential profile as above, then the limiting root
distribution of $P_{n}(z,t)$ will coincide with $\mu_{t}^{s}.$ Furthermore, as
noted above, the limiting eigenvalue distribution of the random matrix model
(\ref{Znst}) (with $B_{n}$ taken to be a Haar-distributed unitary matrix) will
also coincide with $\mu_{t}^{s}.$ Thus, the limiting root distribution of the
heat evolution of $P_{n}$ agrees with the limiting eigenvalue distribution of
$A_{n,s,t},$ giving a rigorous random-polynomial version of the heat flow
conjecture in \cite[Conjecture 2.13]{hallho}. See Figure \ref{triptych2.fig}.

Finally, using the formula in \eqref{HaarCDF} for $\alpha_0(r)$, it is not hard to compute
$t_{\mathrm{sing}}$ in this setting as
\[
t_{\mathrm{sing}}=s+\frac{s^2}{(s+1)^2+1}.
\]
Note that in this case, the value of $s$ in Theorem \ref{freeInterpret.thm} is maximal, in the sense that 
a Haar unitary cannot be decomposed as the free sum of an $R$-diagonal element and a circular element. 
(The Brown measure of such a sum is absolutely continuous on the plane.) Nevertheless, $t_{\mathrm{sing}}$ is strictly
greater than $s$ in this case.

\subsection{Proof of Theorem \ref{freeInterpret.thm}}\label{proof61thm}

To prove Theorem \ref{freeInterpret.thm}, we will use results of Zhong
\cite{Zhong} (drawing on ideas in \cite{HHfamily}). In Section 7 of
\cite{Zhong}, Zhong shows that the Brown measure of $a_{s,t}$ is the
push-forward of the Brown measure of $a_{s,0}$ under a certain map $\Phi
_{s,t}$. (Here we adapt Zhong's results to our notation; what we call $s$ is
denoted $t$ in Zhong, while what we call $t$ is denoted $\gamma$ in Zhong.)
Comparing Equation (7.8) in \cite{Zhong} to Definition \ref{def:transport_map}
in the present paper
we see that Zhong's map $\Phi_{s,t}$ coincides with our
transport map $T_{t}^{s}.$ Thus, by Theorem C in \cite{Zhong} and Theorem
\ref{theo:pushforward} in the present paper, the limiting root distribution
$\nu_{t}$ of $P_{n}(z,t)$ equals the Brown measure of $a_{s,t},$ for $0\leq
t<\min(t_{\mathrm{sing}},s)$---because both measures equal the push-forward of $\nu_0$ under the map $\Phi_{s,t}$.

First, we show that \hyperref[cond:A3]{(A3)} holds. According to \cite[Theorem B]{Zhong}, the Brown measure $\nu_0$ of $a_{s,0}$ has no atom 
at zero. Then by \cite[Theorem 4.4(iii)]{haageruplarsen}, 
\[
\nu_0\{\left\vert z\right\vert\leq [S_{\left\vert a_{s,0} \right\vert^2}(\alpha-1)]^{-1/2} \}=\alpha
\]
for $0<\alpha<1$. According to \cite[Theorem 4.4(ii)]{haageruplarsen}, $S_{\left\vert a_{s,0} \right\vert^2}(\alpha-1)$ is analytic (and positive) for $\alpha\in(0,1)$.  
Thus, we can find an analytic function $g$ such that
\begin{equation}
e^{-g'(\alpha)}=[S_{\left\vert a_{s,0} \right\vert^2}(\alpha-1)]^{-1/2}.\label{gfroms}
\end{equation}
Then we compute from \eqref{gfroms} that
\[
g''(\alpha)=\frac{1}{2}\frac{S'_{\left\vert a_{s,0} \right\vert^2}(\alpha-1)}{S_{\left\vert a_{s,0} \right\vert^2}(\alpha-1)}.
\]
Since, by \cite[Theorem 4.4(ii)]{haageruplarsen}, $S'_{\left\vert a_{s,0} \right\vert^2}(\alpha-1)<0$ on $(0,1)$, we find that $g$ is analytic with $g''<0$ on $(0,1)$.

It remains only to show that $\alpha e^{g'(\alpha)}$ tends to zero as $\alpha$ tends to zero. By \eqref{eq:alpha_0_measure}, this claim amounts to showing that $\alpha_0(r)/r$ tends to zero as $r$ tends to zero. But by \cite[Theorem 13]{BYZ}, the density of $\nu_0$ with respect to the Lebesgue measure on the plane is bounded above by $1/(\pi s)$, so that $\alpha_0(r)$ is bounded above by $r^2/s$ and the result follows.

Second, we show that $t_{\mathrm{sing}}$ is at least $s.$ In general, the transport map $T_t$ maps
the circle of radius $r$ to the ellipse with semi-axes $\operatorname{ax}_+$ and  $\operatorname{ax}_-$ given 
by
\[
\mathrm{ax}_{\pm}=r\pm t\frac{\alpha_{0}(r)}{r}.
\]
Restating the definition of $t_{\mathrm{sing}}$ from Proposition \ref{prop:T_t_bijective}, changing $s$
to $r$ in to avoid a notational collision, we have
\[
t_{\mathrm{sing}}=\left( \sup_{r \in(\eee^{- g^{\prime}(0)}, \eee^{-
g^{\prime}(1)})} \left|  \left( \frac{\alpha_{0}(r)}{r}\right) ^{\prime
}\right| \right) ^{-1}.
\]
The definition of $t_\mathrm{sing}$ ensures that $d(\operatorname{ax}_+)/dr$ and $d(\operatorname{ax}_-)/dr$ are positive for $t<t_\mathrm{sing}$. 

Assume that $\nu_{0}$ is the Brown measure of an element of the form in
(\ref{as0}), for some $s>0$ and let $T_{t}^{s}$ be the associated transport
map. Then $T_{t}^{s}$ maps the circle of radius $r$ to an ellipse with
semi-axes%
\[
\mathrm{ax}_{\pm}(s,t,r)=r\pm t\frac{\alpha_{0}^{s}(r)}{r}.
\]
Now, Zhong shows \cite[Theorem 7.8]{Zhong} that even when $t=s,$ the semi-axes
$\mathrm{ax}_{\pm}(s,s,r)$ are weakly increasing functions of $r.$ But then we
note that
\begin{align*}
\mathrm{ax}_{\pm}(s,t,r)  &  =\frac{t}{s}\left(  \frac{s-t}{t}r+r\pm
s\frac{\alpha_{0}^{s}(r)}{r}\right) \\
&  =\frac{s-t}{s}r+\frac{t}{s}\mathrm{ax}_{\pm}(s,s,r).
\end{align*}
Since $\mathrm{ax}_{\pm}(s,s,r)$ is weakly increasing in $r$, we see that
$\frac{\partial}{\partial r}\mathrm{ax}_{\pm}(s,t,r)>0$ for all $0<t<s,$ so that
\[
-\frac{1}{t}<\left(  \frac{\alpha_{0}^{s}(r)}{r}\right)  ^{\prime}<\frac{1}{t}
\]
for $0<t<s$, showing that $t_{\mathrm{sing}}$ is at least $s$.

Third, we prove that the measure $\nu_t$ is, for $0<t<s$, the Brown measure of the element $a_{s,t}$ in \eqref{zst}. This result holds by the argument given in the first paragraph of this subsection.

Last, we show that if $\nu_0$ is the Brown measure of an element of the form $a_{s,0}$ in \eqref{as0}, then we can construct a 
random polynomial satisfying \hyperref[cond:A1]{(A1)} and \hyperref[cond:A2]{(A2)} (and then, automatically, \hyperref[cond:A3]{(A3)}, as shown above) with limiting root distribution $\nu_0$. Using the bound on the density of $\nu_0$ from \cite[Theorem 13]{BYZ}, we see that \cite[Theorem 2.9]{KZ14} applies, and the desired result holds. 

\section{Proof of Theorem~\ref{theo:main_general_g}}
\label{sec:proof_main_thm}
We begin with the proof of our main theorem. For simplicity, we will assume that $\P[\xi_0=0] = 0$, so that the degree of $P_n$ is exactly $n$. Without this assumption, the degree of $P_n$ is $n$ minus a geometric random variable, and essentially the same proofs apply with more complicated notation.

\subsection{Proof of Theorem~\ref{theo:main_general_g}}
Consider first the case when $t>0$.  Let $\Delta$ be the Laplace operator understood in the distributional sense; see, e.g., \cite[Section~3.7]{ransford}.
%Recall its Green function: $\frac 1 {2\pi} \Delta \log |z-z_0| = \delta_{z_0}$.
The empirical distribution of zeros of $\eee^{-\frac t{2n} \partial_z^2} P_n$ can be written as
\begin{align}\label{eq:distr_Poisson_finite_n}
\llbracket \eee^{-\frac t{2n} \partial_z^2} P_n\rrbracket  = \frac 1 {2\pi n} \Delta \log \left|\eee^{-\frac t{2n} \partial_z^2} P_n\right|;
\end{align}
see~\cite[Theorem~3.7.8]{ransford}.
Our main task is to compute the limit of this expression as $n\to\infty$. More precisely, we will show that (with a mode of convergence to be made precise below)
\begin{equation}\label{eq:proof_main_theo_g_aim}
\frac 1 {n} \log \left|\eee^{-\frac t{2n} \partial_z^2} P_n(z)\right|
\ton
\sup_{\alpha \in [0,1]} f_t(\alpha,z)
=
U_{t}(z) + g(1)
%=???
%?? \lim_{n\to\infty} \int_{\C} \log |z-y|\; \llbracket \eee^{-\frac t{2n} \partial_z^2} P_n\rrbracket  (\dd y),??
\end{equation}
where $f_t$ is defined as in~\eqref{eq:f_def}. Then, we will argue that the  Laplacian can be interchanged with the large $n$ limit, which would imply that $\llbracket \eee^{-\frac t{2n} \partial_z^2} P_n\rrbracket $ converges to $\nu_t\coloneqq \frac 1{2\pi} \Delta U_t$.
Finally, we will show that $U_t$ is the logarithmic potential of $\nu_t$.  The term $g(1)$ appears in~\eqref{eq:proof_main_theo_g_aim} because the polynomial $P_n$ need not be monic and its leading coefficient is $a_{n,n} = \xi_n \eee^{n g(1) + o(n)}$.

Let us now be more precise. By the formula
\begin{equation}\label{eq:heat_flow_z_n_hermite_poly}
\exp\left\{-\frac{s}{2} \partial_z^2\right\} z^n =  s^{n/2} \He_n\left(\frac{z}{\sqrt s}\right),
\qquad
s>0,
\end{equation}
we have
\begin{equation}\label{eq:heat_flow_applied_to_P_n_with_general_g}
\mathbf G_n(z)
\coloneqq
\eee^{-\frac t{2n} \partial_z^2} P_n (z)
=\sum_{k=0}^n \xi_k a_{k;n} \left(\frac t n\right)^{k/2} \He_k\left(\frac {\sqrt n z} {\sqrt t}\right)
=
\sum_{k=0}^n \xi_k \mathbf c_{k;n}(z),
\end{equation}
where
\begin{equation}\label{eq:c_k_n_def}
\mathbf c_{k;n}(z)
\coloneqq
a_{k;n}\left(\frac t n \right)^{k/2} \He_k\left(\frac{\sqrt n z}{\sqrt t}\right)
=
\eee^{n g(k/n)+o(n)}\left(\frac t n \right)^{k/2} \He_k\left(\frac{\sqrt n z}{\sqrt t}\right)
\end{equation}
with an $o(n)$-term which is uniform in $k\in \{0,\ldots, n\}$. Let us determine the rate of exponential growth of the term $\mathbf c_{k;n}(z)$ having in mind the regime in which $k$ grows linearly with $n$.
\begin{lemma}\label{lem:exonential_rate_c_k_n}
Fix some compact set $K\subset \C\backslash[-2\sqrt{t},2\sqrt{ t}]$. Write $\alpha \coloneqq k/n$. Then, as $n\to\infty$,
%The following estimates will be uniform in as long as $z\in K$. Fix some $\eps$.
\begin{align}\label{eq:summand_asymptotic}
\frac 1n \log \left|\mathbf c_{k;n}(z)\right|
%=
%\left|\eee^{n g(k/n) + o(n)} \left(\frac t n\right)^{k/2} \He_k\left(\frac {\sqrt n z} {\sqrt t}\right)\right|
=
g(\alpha) + \alpha \log \sqrt{\alpha t}+ \alpha \Psi\Big(\frac{z}{\sqrt{t\alpha}}\Big)+o(1)
=
f_t(\alpha, z) + o(1)
\end{align}
uniformly in $k\in \{0,\ldots, n\}$ and $z\in K$. Here,  $f_t(\alpha, z)$ and $\Psi(z)$ are defined as in~\eqref{eq:f_def} and~\eqref{eq:Psi_def}.
\end{lemma}

For the proof, we need the following well-known Plancherel--Rotach result for the asymptotics of the Hermite polynomials 
outside the ``oscillatory'' interval $[-2,2]$.
\begin{lemma}\label{lem:hermite_asymptotics}
Then, for every $z\in \C\backslash[-2,2]$ we have
\begin{equation}\label{eq:hermite_poly_asymptotics}
\lim_{k\to\infty} \frac 1k \log \left|\frac{\He_k(\sqrt k\, z)}{k^{k/2}}\right| =  \Psi(z).
\end{equation}
Moreover, the convergence is uniform as long as $z$ stays in any compact subset of $\C\backslash[-2,2]$.
\end{lemma}
\begin{proof}
See~\cite[Theorem~8.22.9 (b)]{szegoe_book} for the real $z$ case and~\cite[p.~622]{van_assche_some_results} for the proof of local uniformity on $\C\backslash[-2,2]$. Denoting by $z_{1;k},\ldots, z_{k;k}$ the zeros of $\He_k(\sqrt k \, z)$, which  are real and satisfy $|z_{j,k}| \leq \sqrt{4 + (2/k)}$, we have
\begin{equation}\label{eq:hermite_poly_asymptotics_proof}
\frac 1k \log \left|\frac{\He_k(\sqrt k\, x)}{k^{k/2}}\right| = \frac 1k \sum_{j=1}^k \log |x-z_{j;k}|,
\end{equation}
and the claim can be deduced from the well-known fact~\cite[p.~620]{van_assche_some_results} that $\frac 1k \sum_{j=1}^k \delta_{z_{j;k}}$ converges weakly to the Wigner distribution on $[-2,2]$ whose logarithmic potential is $\Psi(z)$.
Related results can be found in~\cite[Lemma~3]{gawronski_strong_asymptotics}, \cite[Eq.~(4.8)]{gawronski_van_assche}, \cite{gawronski}.
\end{proof}

\begin{proof}[Proof of Lemma~\ref{lem:exonential_rate_c_k_n}]
Fix $\eps\in (0,1)$ and let first $\alpha = k/n \in [\eps, 1]$ and $z\in K$.  Write the argument of the Hermite polynomial in~\eqref{eq:c_k_n_def} as  $\sqrt n \, z/ \sqrt t = \sqrt k \cdot \sqrt{n/(kt)}\, z$ and observe that $\sqrt{n/(kt)}\, z$ stays in some compact subset of $\C\backslash[-2,2]$.
Letting $n\to\infty$ and utilizing Lemma~\ref{lem:hermite_asymptotics}, we obtain
\begin{align*}
\frac 1n \log \left|\left(\frac t n \right)^{k/2} \He_k\left(\frac{\sqrt n z}{\sqrt t}\right)\right|
&=
\frac {k}{2n} \log\frac t n + \frac{k\log k}{2n} +  \frac {k}{n} \Psi\left(\sqrt{\frac{n} {kt}}\, z\right) + o(1)\\
&=
\alpha \log \sqrt{\alpha t}+ \alpha \Psi\left(\frac{z}{\sqrt{t\alpha}}\right)+o(1).
\end{align*}
Thus, \eqref{eq:summand_asymptotic} holds uniformly in $z\in K$ and $\alpha = k/ n \in [\eps, 1]$.

Let now $\alpha = k/n \in [0,\eps]$ and $z\in K$. Since~\eqref{eq:summand_asymptotic} trivially holds for $k=0$, we will assume $k\geq 1$.  It follows from~\eqref{eq:hermite_poly_asymptotics_proof} and the fact that all zeros of $\He_k(\sqrt k \, z)$ are located in $[-3,3]$ that, for sufficiently large $|x| > C_0$ and all $k\in \N$,
$$
\frac 1k \log \left|\frac{\He_k(\sqrt k\, x)}{k^{k/2}}\right| \leq  \log |x| + 1.
$$
If $\eps<\eps_0(t, K)$ is sufficiently small, we have $|\sqrt{n/(kt)}\, z| > C_0$ and it follows that
$$
\frac 1n \log \left|\left(\frac t n \right)^{k/2} \He_k\left(\frac{\sqrt n z}{\sqrt t}\right)\right|
\le
\frac {k}{2n} \log\frac t n + \frac{k\log k}{2n} +  \frac {k}{n} \left(\log \left|\sqrt{\frac{n} {kt}}\, z\right| + 1\right)
=
\frac kn \log |z| + \frac kn.
$$
For a given $\delta>0$ we can find $\eps>0$ such that the absolute value of the right-hand side is $\leq \delta$, for all $z\in K$ and $1\leq k \leq \eps n$.  On the other hand, $f_t(\alpha,z)-g(0)\to 0$ as $\alpha \downarrow 0$ (uniformly in $z\in K$), so that it also becomes $\leq \delta$ for all $\alpha \in [0,\eps]$ if $\eps$ is sufficiently small.  It follows that~\eqref{eq:summand_asymptotic} holds uniformly in $z\in K$ and $\alpha\in [0, 1]$.
\end{proof}

We now  claim that the largest exponent wins in~\eqref{eq:heat_flow_applied_to_P_n_with_general_g} and  that, in particular, the random factors $\xi_k$ \emph{prevent} the exponentially growing terms $\mathbf c_{k;n}(z)$ from canceling each other ``too much'' in  the sense that
\begin{align}\label{eq:stoch_convergence}
\frac 1 {n} \log \left|\mathbf G_n(z)\right|
%=
%\frac 1 {n} \log \left|\eee^{-\frac t{2n} \partial_z^2} P_n(z)\right|
=
\frac 1 {n} \log \left| \sum_{k=0}^n \xi_k \mathbf c_{k;n}(z)\right|
\toprobab  \sup_{\alpha\in [0,1]} f_t(\alpha, z)
\end{align}
for each fixed $z\in \C\backslash[-2\sqrt{t},2\sqrt{ t}]$.
%(This convergence still holds uniformly in compact subsets of $\C\backslash[-2\sqrt{t},2\sqrt{ t}]$ but we shall not need this fact here.)
Note that the right-hand side is finite since $f_t(\alpha, z)$ is continuous in $\alpha \in [0,1]$.
More precisely, \eqref{eq:stoch_convergence} follows from  Lemma \ref{lem:Laplace} below in which we choose $\mathbf c_{k;n}\coloneqq\mathbf c_{k;n}(z)$ and $f(\alpha)=f_t(\alpha,z)$. Note that its condition~\eqref{eq:lem_Laplace_ass} corresponds to~\eqref{eq:summand_asymptotic}.

It remains to show that the pointwise convergence \eqref{eq:stoch_convergence} can be leveraged to  $L^1_{\mathrm{loc}}$ convergence, implying weak convergence of their Laplacians.
By stochastic convergence \eqref{eq:stoch_convergence} and Remark~\ref{rem:lem}, for each subsequence $n_j^{-1} \log |\mathbf G_{n_j}(z)|$, there exists a subsubsequence, which we denote by $m_j^{-1} \log |\mathbf G_{m_j}(z)|$, for which \eqref{eq:stoch_convergence} holds $\P$-almost surely.
%For the sake of simplicity we shall index this subsubsequence with $n$ as well.
In other words, for all $z\in \C\backslash [-2\sqrt{t},2\sqrt{ t}]$ there exists an event $\Omega^{(z)}$ with $\P(\Omega^{(z)})=1$ such that for all $\omega\in\Omega^{(z)}$ the above subsequential $\P$-a.s.~ convergence holds (along a subsequence which does not  depend on $z$). By Fubini's theorem applied to the indicator function of $\{(z,\omega)\in\C\times\Omega: \lim_{j} \frac 1 {m_j} \log \left|\mathbf G_{m_j}(z)\right|\neq  U_t(z) + g(1)\}$, this implies  the existence of an event $\Omega_0$ with $\P(\Omega_0)=1$ such that for all $\omega\in\Omega_0$ there exists a Lebesgue negligible set $V_\omega\subset \C$ such that
\begin{align}\label{eq:logpot_as_convergence}
\lim_{j\to\infty} \frac 1 {m_j} \log \left|\mathbf G_{m_j}(z)\right|= \sup_{\alpha\in [0,1]} f_t(\alpha, z) = U_t(z) + g(1)\quad\text{ for }z\in V_\omega^c.
\end{align}
Since the coefficients $\xi_k$ are non-degenerate, there exists $\Omega_1$ with $\P(\Omega_1)=1$ such that for all $\omega\in\Omega_1$ it holds $\mathbf G_n(z)\not\equiv 0$ for $n$ sufficiently large, hence $\frac 1 n \log |\mathbf G_n(z)|$ is a subharmonic function.

Moreover, by \cite[Lemma 4.4]{KZ14}, the existence of the logarithmic moment is equivalent to sub-exponential growth of i.i.d.\ random variables. More precisely, for any $\varepsilon>0$ there is an a.s.\ finite random variable $M$ such that $|\xi_k|\le M \eee^{\varepsilon k}$ for all $k\in \N_0$.
Together with \eqref{eq:summand_asymptotic} it follows there exists an event $\Omega_2$ with $\P(\Omega_2)=1$ such that
\begin{align}
\vert \mathbf G_{n}(z) \vert\le M \sum_{k=0}^{n} \eee^{\varepsilon k}|\mathbf c _{k;n}(z) |
&\le M\eee^{\varepsilon n} \sum_{k=0}^{n }\exp(n f_t(k/n,z )+\varepsilon n)\nonumber\\
&\le
M\exp(nU_t(z)+ng(1) +3\varepsilon n)\label{eq:log_pot_upperbound}
\end{align}
for all $\omega\in \Omega_2$ and uniformly in $n>n_0$ (where $n_0$ is sufficiently large) and  in $z$ on compact subsets of $\C\backslash [-2\sqrt t ,+2\sqrt t ]$.
From now on, we work with a fixed $\omega\in\Omega_0\cap \Omega_1\cap \Omega_2$. %If we restrict $z$ to a compact subset of the  connected open domain $\C\backslash [-2\sqrt t ,+2\sqrt t ]$, then
The family $(m_j^{-1} \log |\mathbf G_{m_j}(z)|)_{j\in \N}$ is a family of subharmonic functions on  $\C\backslash [-2\sqrt t ,+2\sqrt t ]$ bounded from above by $U_{t}(\cdot)+g(1)+4\varepsilon$ \footnote{Note that due to the logarithmic singularities, the functions are not bounded from below, but this does not harm the argument that will follow. Alternatively, one can deal with these singularities and leverage the pointwise convergence to weak convergence by approximating the integral over test functions by a randomly shifted lattice (see \cite[Lemma 3.2]{jalowy}) or Monte Carlo approximation (see \cite[\S 4]{GJ21}).}
 and by \cite[Theorem 4.1.9]{hormander1}, any family of upper-bounded subharmonic functions is either converging to $-\infty$ or contains a further subsequence converging in  $L^1_{\mathrm{loc}}$ (hence in a distributional sense). The former is excluded by \eqref{eq:logpot_as_convergence} and the definition of $\Omega_1$. For simplicity, we denote the subsequence by $\frac 1 n \log |\mathbf G_{n}(z)|$ and observe that it can only converge to $U_t(z) + g(1)$ by \eqref{eq:logpot_as_convergence}. Let $\varphi\in\mathcal C_c^\infty(\C)$ be a compactly supported smooth  test function.
Using  \eqref{eq:distr_Poisson_finite_n}, the definition of the Laplacian $\Delta$ on the space of distributions and the $L^1_{\mathrm{loc}}$-convergence implies
\begin{multline*}
\int_{\C} \varphi d\llbracket \eee^{-\frac t{2n} \partial_z^2} P_n\rrbracket
=
\int_{\C} \Delta \varphi(z) \frac 1 {2\pi n} \log |\mathbf G_{n}(z)| \dint z \dint \bar z
\\\ton
\int_{\C}\Delta\varphi (z) \frac 1 {2\pi}(U_{t}(z)+g(1)) \dint z \dint \bar z
=
\int_{\C}\varphi(z) \nu_{t}(\dint z)
\end{multline*}
as $n\to\infty$; which means that $\llbracket \eee^{-\frac t{2n} \partial_z^2} P_n\rrbracket $ converges vaguely to $\nu_t \coloneqq \frac 1 {2\pi} \Delta U_t$ (which is a Radon measure by~\cite[Theorem~3.7.2]{ransford} since $U_{t}$ is subharmonic by Proposition~\ref{prop:log_potential_U_t_cont}). Note that $\nu_t(\C)\leq 1$ since  $\nu_t$ is a vague limit of probability measures.

Let us now show that $U_t(z)$ coincides with the logarithmic potential $W_t(z) \coloneqq \int_\C \log |z-y| \, \nu_t (\dd y)$ of $\nu_t$. Since $U_t$ is harmonic outside some large disk by Proposition~\ref{prop:log_potential_large_z}, $\nu_t$ is a measure with compact support. By Weyl's lemma~\cite[Lemma~3.7.10]{ransford}, the function  $U_t(z)-W_t(z)$ is harmonic. As $|z|\to\infty$, we have  $W_t(z) = \nu_t(\C) \log |z| + O(1/|z|)$ by~\cite[Theorem~3.1.2]{ransford} and $U_t(z) = \log |z| + O(1/|z|)$ (since $U_t$ coincides with the log-potential of the Wigner law $\mathsf{sc}_t$ for large $|z|$ by Proposition~\ref{prop:log_potential_large_z}). It follows from these formulas together with $\nu_t(\C)\leq 1$ that  the harmonic function $U_t(z)-W_t(z)$ is bounded below and hence is constant by Liouville's theorem~\cite[Corollary 1.3.2]{ransford}. We conclude that $\nu_t(\C) = 1$ and $U_t(z) = W_t(z)$, which proves~\eqref{eq:log_potential_limit}.

Now that we know that the limit measure $\nu_{t}$ is a probability measure, the vague convergence can be lifted to weak convergence.
So, for $\P$-almost all $\omega$,  $\llbracket \eee^{-\frac t{2n} \partial_z^2} P_n\rrbracket $ converges weakly to $\nu_t$.
%Note that , thus $\nu_{t}$ is a measure and $U_{t}\in L_{\mathrm{loc}}^p$ for all $1\leq p<\infty$ by \cite[Theorem 16.1.2]{hormander2}.
%\JJ{Do we need that?}
Since we have been working with subsequences, we conclude that $\llbracket \eee^{-\frac t{2n} \partial_z^2} P_n\rrbracket $ (viewed as a random element with values in the space of finite measures on $\C$ endowed with the weak topology) converges to $\nu_{t}$ in probability. This proves Theorem~\ref{theo:main_general_g} for $t>0$.

For general $t\in\C\backslash\{0\}$, we consider $t=|t| \eee^{i\phi}\in\C$ and note that \eqref{eq:heat_flow_applied_to_P_n_with_general_g} can be rephrased as
\begin{align*}
\eee^{-\frac t{2n} \partial_z^2} P_n (z)
=\sum_{k=0}^n \xi_k \eee^{n g(k/n) + o(n)} \left(\frac t n\right)^{k/2} \He_k\left(\frac {\sqrt n z} {\sqrt t}\right)=\sum_{k=0}^n \xi_k \tilde{\mathbf c}_{k;n}(\eee^{-i\phi/2}z)
\end{align*}
for $\tilde {\mathbf c}_{k;n}(z)=\eee^{i\phi k /2}\eee^{n g(k/n)+o(n)}\big(\frac {|t|} n \big)^{k/2} \He_k\big(\frac{\sqrt n z}{\sqrt {|t|}}\big)$. Obviously, $\tilde{\mathbf c}_{k;n}(z)$ satisfies the same assumption \eqref{eq:lem_Laplace_ass} as for real $t$, i.e. $\frac 1 n \log|\tilde{\mathbf c}_{k;n}(z)|\to f_t(k/n,z)$. Hence, the logarithmic potential is given by
\begin{align*}
U_{t}(z)
&\coloneqq
\lim_{n\to\infty} \frac 1 {n} \log \left|\eee^{-\frac t{2n} \partial_z^2} P_n(z)\right| -  g(1)\\
&=
\lim_{n\to\infty} \frac 1 {n} \log \left|\eee^{-\frac {|t|}{2n} \partial_z^2} P_n(\eee^{-i\phi/2}z)\right|- g(1)
=
U_{|t|}(\eee^{-i\phi/2}z).
\end{align*}
Applying the Laplacian, we obtain the push-forward
\begin{align*}
\nu_{t}(z)=\frac{1}{2\pi}\Delta U_{t}(z)=\frac{\eee^{-i\phi}}{2\pi}\Delta U_{|t|}(\eee^{-i\phi/2}z)=\nu_{|t|}(\eee^{-i\phi/2}z)
\end{align*}
and the representation of the function $f_t(\alpha,z)$ follows from the case of real $t$.

%\section{Technicalities}
\subsection{Laplace's method for random sums}
In this section we study random sums, such as polynomials \eqref{eq:heat_flow_applied_to_P_n_with_general_g} for fixed argument $z\in\C$, where the summation consists of independent coefficients with an exponential behavior as $n\to\infty$. In this case oscillatory cancellations cannot happen and the asymptotics are determined by the term with largest exponent.

\begin{lemma}\label{lem:Laplace}
Let $\xi_0,\xi_1,\dots$ be a sequence of non-deterministic complex i.i.d.~random variables with finite logarithmic moments, that is $\E\log(1+|\xi_0|)<\infty$. Let $(\mathbf c_{k;n})_{0\leq k\le n,n\in\N}$ be an array of complex numbers and $f:[0,1]\to \R$ be a continuous function.
If
\begin{align}\label{eq:lem_Laplace_ass}
\frac 1 n\log|\mathbf c_{k;n} |=f(k/n )+o(1)
\end{align}
as $n\to\infty$ uniformly in $0\leq k\leq  n$, then $\mathbf G_n \coloneqq\sum_{k=0}^{n}\xi_k \mathbf c_{k;n} $ satisfies
\begin{align}\label{eq:lem_Laplace}
\frac 1 n \log\lvert\mathbf G_n \rvert \toprobab  U \coloneqq \sup_{\alpha \in[0,1]}f(\alpha).
\end{align}
%for all $z\in\C$ such that $U$ is finite..
\end{lemma}

According to the previous proof, we may think of $f$ to be the logarithmic potential of the limiting empirical zero distribution of the summands and $U$ to be the limiting log-potential of the zeros of $\mathbf G_n(z)$, up to additive constants.

The proof of Lemma \ref{lem:Laplace} is an adjustment of the ideas in \cite[\S 4.3-4.4]{KZ14}. As we will see below, the upper bound of \eqref{eq:lem_Laplace} holds $\P$-a.s.~and simply follows from taking the supremum in $n\in\N$. The lower bound however, crucially depends on the coefficients to be random (i.e.\ non-deterministic). In particular, we will use the Kolmogorov--Rogozin inequality below, which states that sums of random variables with high probability do not concentrate anywhere.

For a complex random variable $X$, define its concentration function
\begin{align}\label{eq:conc_fct_def}
Q(X,r)=\sup_{z\in\C}\P(|X-z|<r),\quad r>0.
\end{align}
It is easy to see that for all $r,a>0$
\begin{align}\label{eq:conc_fct_scale}
Q(aX,r) = Q(X,r/a).
\end{align}
Also, if $X$ and $Y$ are independent, then
\begin{align}\label{eq:conc_fct_sum}
Q(X+Y,r)\le Q(X,r).
\end{align}

The following anti-concentration result  can be found in \cite[(6.7)]{Esseen68}, see also \cite{Kolm58}.

\begin{proposition}[Kolmogorov--Rogozin]\label{prop:KolmRog}
There exists an absolute constant $c>0$ such that for all $n\in\N$,  independent complex random variables $X_1,\dots,X_n$ and $r>0$ we have
\begin{align*}
Q\Big(\sum_{k=1}^n X_k,r\Big)\le c\left(\sum_{k=1}^n\big(1-Q(X_k,r)\big)\right)^{-1/2}.
\end{align*}
\end{proposition}

We are now ready to give the
\begin{proof}[Proof of Lemma \ref{lem:Laplace}]
Fix $\varepsilon>0$. Let us first show the upper bound of \eqref{eq:lem_Laplace}, which even holds $\P$-a.s., analogously as we have shown \eqref{eq:log_pot_upperbound}.
%We begin by arguing that the series converges in the case of $T=\infty$ and its tail is negligible.
By \cite[Lemma 4.4]{KZ14}, the finiteness  of the logarithmic moment is equivalent to sub-exponential growth of i.i.d.\ random variables. More precisely, there is an a.s.\ finite random variable $M$ such that $|\xi_k|\le M \eee^{\varepsilon k}$ for all $k\in \N_0$.
 Together with assumption \eqref{eq:lem_Laplace_ass} it follows that for $n$ sufficiently large we have

\begin{align}
\vert\mathbf G_n \vert\le M \sum_{k=0}^{n} \eee^{\varepsilon k}|\mathbf c _{k;n} |\le M\eee^{\varepsilon n} \sum_{k=0}^{n }\exp[n f(k/n )+\varepsilon n]\le M\exp[nU +3\varepsilon n]\quad \P-\text{a.s.}\label{eq:lem_upperbound2}
\end{align}
for $n$ sufficiently large. This proves the upper bound.
Now let us turn to the lower bound, which we claim to be
\begin{align}\label{eq:lem_lowerbound}
\P\Big(\vert\mathbf G_n \vert< \eee^{n(U -3\varepsilon)}\Big)=\mathcal O (1/\sqrt n).
\end{align}
Since $U=\sup_{\alpha\in [0,1]} f(\alpha)$ and $f$ is continuous, there exists an interval $J=[\alpha^* -\delta,\alpha^* +\delta] \subset (0,1)$ such that $U \le f(\alpha )+\varepsilon$  for all $\alpha$ in $J$. Combined with our assumption \eqref{eq:lem_Laplace_ass}, this implies that
\begin{align*}
|\mathbf c_{k;n} |>\eee^{nU -2\varepsilon n}
\qquad
\text{ for all }
k\in \mathcal J_n\coloneqq[n\alpha^* -n\delta,n\alpha^* + n\delta] \cap \N,
\end{align*}
provided that $n$ is sufficiently large.
Define $a_{k,n}=\eee^{-nU +2\varepsilon n}\mathbf c_{k;n} $ and note that $|a_{k,n}|>1$ for $k\in\mathcal J_n$. We split our sum $\mathbf G_n$  into $\mathbf G_{n,1}=\sum_{k\in\mathcal J_n} a_{k;n}\xi_k$ and $\mathbf G_{n,2}=\sum_{k\not\in\mathcal J_n} a_{k,n}\xi_k$. We use \eqref{eq:conc_fct_sum} to get rid of $\mathbf G_{n,2}$ and then apply Proposition \ref{prop:KolmRog} to obtain
\begin{align}
\P\Big(\vert\mathbf G_n \vert<\eee^{n(U -3\varepsilon)}\Big)&\le Q(\mathbf G_{n,1}+\mathbf G_{n,2},\eee^{-n\varepsilon})\nonumber\\
&\le Q(\mathbf G_{n,1},\eee^{-n\varepsilon})\nonumber\\
&\le c\left(\sum_{k\in\mathcal J_n}\big(1-Q(a_{k,n}\xi_k,\eee^{-n\varepsilon})\big)\right)^{-1/2}\nonumber\\
&\le  c\left(\sum_{k\in\mathcal J_n}\big(1-Q(\xi_k,\eee^{-n\varepsilon})\big)\right)^{-1/2},\label{eq:ineq_concentration_functions}
\end{align}
where the last step follows from \eqref{eq:conc_fct_scale} and monotonicity of $Q$ in $r$. Since the random variables $\xi_k$ are non-deterministic, we may choose $n$ so large that $Q(\xi_0,\eee^{-n\varepsilon})\le \tilde c$ for some $\tilde c <1$ and arrive at our claim
\begin{align*}
\P\Big(\vert\mathbf G_n \vert<\eee^{n(U -3\varepsilon)}\Big)&\le c/\sqrt{(1-\tilde c)|\mathcal J_n|} = \mathcal O (1/\sqrt n),
\end{align*}
where we used that $|\mathcal J_n| \geq \delta n$ for sufficiently large $n$.
Ultimately, the stochastic convergence in Lemma \ref{lem:Laplace} follows from combining \eqref{eq:lem_lowerbound} and \eqref{eq:lem_upperbound2}.
\end{proof}

\begin{remark}\label{rem:lem}
From \eqref{eq:lem_lowerbound} and the Borel--Cantelli Lemma, we deduce that the conclusion of Lemma~\ref{lem:Laplace} holds $\P$-a.s.\ along any subsequence $(m_j)_{j\in \N}$ such that $m_j>j^3$, that is
$$
\frac 1 {m_j} \log\lvert\mathbf G_{m_j} \rvert \toasj  U = \sup_{\alpha \in[0,1]}f(\alpha).
$$
\end{remark}

\begin{remark}
The only instance in which we made use of $\xi_k$ having \emph{identical} distribution is after \eqref{eq:ineq_concentration_functions}. Inspecting the proof, shows that Theorem \ref{theo:main_general_g} remains valid for $\xi_k$ having different distributions with uniform bound on their logarithmic moment and concentration function.
\end{remark}

\section{Proofs: Properties of \texorpdfstring{$\nu_t$}{nu t}}
\label{sec:proof_nu_t}

\subsection{Calculations with logarithmic potentials}\label{subsec:log_pots}
Let us collect some important identities for the upcoming considerations. Recall from~\eqref{eq:Psi_def} that
\begin{equation}\label{eq:Psi_def_rep}
\Psi(z) = \frac 14 \Re\left(z^2 - z \sqrt{z^2 - 4}\right)  + \log\left|\frac{z + \sqrt{z^2 - 4}}{2}\right| -\frac 12,
\qquad
z\in \C.
\end{equation}
It follows that $\Psi(z) = \Re \psi(z)$, where $\psi(z)$ is the multivalued analytic function
\begin{equation}\label{eq:psi_def_small}
\psi(z) = \frac 14 \left(z^2 - z \sqrt{z^2 - 4}\right)  + \log\left(\frac{z + \sqrt{z^2 - 4}}{2}\right) -\frac 12,
\qquad
z\in \C\backslash[-2,2].
\end{equation}
Because of the presence of the logarithm, the branches of $\psi(z)$ differ by $2\pi \ii n$ with  $n\in \Z$. It follows that $\psi'(z)$ is a \emph{univalued} analytic function on $\C\backslash[-2,2]$, and in fact it is easy to check that
$$
\psi'(z) =  \frac 12 \left( z - \sqrt{z^2 - 4}\right), \qquad \lim_{|z|\to\infty}\psi'(z) = 0.
$$
Also, it follows from~\eqref{eq:Psi_def_rep} that $\Psi(z) = \Re \psi(z):\C \to\R$ is a well-defined, continuous function which is harmonic outside $[-2,2]$. For $x\in [-2,2]$,  the term $\sqrt{x^2-4}$ is purely imaginary and $\Psi(x)$ simplifies to
\begin{equation}\label{eq:Psi_for_x_real_smaller_2}
\Psi(x)=\frac{1}{4}\Re(x^2-x\sqrt{x^2-4})+\log\left\lvert\frac{x+\sqrt{x^2-4}}{2}\right\rvert -\frac 12=\frac{x^2}{4}-\frac 1 2,
\qquad x\in [-2,2].
\end{equation}
As explained above, $\Psi(z)$ is the logarithmic potential of the semicircle distribution $\mathsf{sc}_1$ on $[-2,2]$, and we have $\Delta\Psi=2\pi\cdot\mathsf{sc}_1$ in the distributional sense. The  Stieltjes transform of  $\mathsf{sc}_1$ can be rephrased as
\begin{align}\label{eq:Stieltjes_semicircle}
\int_{-2}^{2} \frac{\mathsf{sc}_{1}(\dd y)}{z-y}
=2\partial_z \Psi(z) = \partial_z (\psi(z) + \psi(\bar z))= \psi'(z) = \frac 12 \left( z - \sqrt{z^2 - 4}\right),
\qquad
z\in \C\backslash[-2,2].
\end{align}
Let us finally mention one more consequence of~\eqref{eq:Psi_def_rep}:  for all $u\in \C$ with  $|u|>1$ we have
$$
\Psi\left( u + \frac 1u\right) = \frac 12 \Re \frac 1 {u^2} + \log |u|.
$$
Indeed, for $z=u+\frac 1u \in \C \backslash [-2,2]$ we have $z^2 - 4 = (u-\frac 1u)^2$ and hence $(z + \sqrt {z^2-4})/2 = (u + \frac 1u + u-\frac 1u)/2 = u$, where the sign of the square root was chosen to fulfill our convention that $(z + \sqrt {z^2-4})/2$ should have absolute value $>1$.

\begin{lemma}\label{lem:f_der_in_alpha}
Let $g:[0,1]\to\R$  be continuously differentiable on $(0,1)$.  For all $t>0$ and $z\in \C$, the function $f_t(\alpha, z)$ is continuously differentiable in $\alpha\in (0,1)$ and  we have
\begin{equation}\label{eq:f_alpha_z_der_in_alpha}
\frac{\partial}{\partial \alpha} f_t(\alpha, z) = g'(\alpha) + \log \left|\frac{z}{2} + \sqrt{\frac{z^2}{4} - t\alpha}\right|.
\end{equation}
\end{lemma}
\begin{proof}
It follows from~\eqref{eq:f_def} that
\begin{align*}
\frac{\partial}{\partial \alpha} f_t(\alpha,z)
&=
g'(\alpha)
+
\frac 12 \log (t\alpha)  + \frac 12  +  \Re \frac{\partial}{\partial \alpha} \left(\alpha \psi\left(\frac{z}{\sqrt{t\alpha}}\right)\right)
\\
&=
g'(\alpha)
+
\frac 12 \log (t\alpha)  + \frac 12  +  \Re \left(\psi\left(\frac{z}{\sqrt{t\alpha}}\right) - \frac 12 \cdot \frac z {\sqrt{t\alpha}} \cdot \psi'\left(\frac{z}{\sqrt{t\alpha}}\right)\right).
\end{align*}
Suppose first that $z\in \C\backslash[-2\sqrt{t\alpha},2\sqrt{t\alpha}]$.  From~\eqref{eq:psi_def_small} and~\eqref{eq:Stieltjes_semicircle} it follows that
$$
\Re\left(\psi(x) -  \frac 12 x \psi'(x)\right) = \log \left|\frac{x + \sqrt{x^2 - 4}}{2}\right| -\frac 12,
\qquad
x\in \C\backslash[-2,2],
$$
which yields~\eqref{eq:f_alpha_z_der_in_alpha}. If $z\in (-2\sqrt{t\alpha},2\sqrt{t\alpha})$, then~\eqref{eq:Psi_for_x_real_smaller_2} implies $\frac{\partial}{\partial \alpha} f_t(\alpha, z) = g'(\alpha) + \frac 12 \log (t\alpha)$, and~\eqref{eq:f_alpha_z_der_in_alpha} is true.  Finally, if $z=\pm 2\sqrt{t\alpha}$, the above shows that the left and right derivatives of $f_t(\alpha, z)$ in $\alpha$ exist and are equal, hence~\eqref{eq:f_alpha_z_der_in_alpha} continues to hold in this case.
\end{proof}

\subsection{Properties of the logarithmic potential}\label{logPotProp.sec}
Recall from~\eqref{eq:f_def} that $f_t(\alpha,z) = g(\alpha)+ \alpha\log \sqrt{|t|\alpha} + \alpha \Psi( z/\sqrt{t\alpha})$ and we claimed in Proposition \ref{prop:log_potential_U_t_cont} that $U_t(z)=\sup_{\alpha \in [0,1]} f_t(\alpha,z) -g(1)$ is continuous and subharmonic.

\begin{proof}[Proof of Proposition \ref{prop:log_potential_U_t_cont}]
Let $t>0$.
For each $\alpha \in (0,1]$, the function $z\mapsto f_t(\alpha, z)$ is continuous. To prove that $U_t(z)=\sup_{\alpha \in [0,1]} f_t(\alpha,z) -g(1) $ is continuous, it suffices to check that the family $(f_{t}(\alpha, \cdot))_{\alpha\in [0,1]}$ is equicontinuous if $z$ is restricted to an arbitrary compact set $K\subset \C$. The function $\Psi(z)$ is Lipschitz on $K$, that is $|\Psi(z_1) - \Psi(z_2)| \leq C_K |z_1-z_2|$ for all $z_1,z_2 \in K$. (Indeed, $\partial_z \Psi$ exists and is bounded on $K\backslash [-2,2]$, as follows from the explicit formula~\eqref{eq:Stieltjes_semicircle}. The same applies to $\partial_{\overline{z}} \Psi = \overline{\partial_z \Psi}$ since $\Psi$ is real-valued. It follows that $(\partial/ \partial x)\Psi$ and $(\partial/\partial y)\Psi$ are bounded on $K\backslash [-2,2]$. Since $\Psi$ is continuous, we can write $\Psi(z_1)-\Psi(z_2)$ as an integral of directional derivative of $\Psi$, which is bounded a.e.,\ and the claim follows.)  It follows from the Lipschitz property of $\Psi$ that for all $\alpha \in (0,1]$,
$$
|f_t(\alpha,z_1) - f_t(\alpha,z_2)|
=
\left|\alpha \Psi\left(\frac{z}{\sqrt{t\alpha}}\right) - \alpha \Psi\left(\frac{z}{\sqrt{t\alpha}}\right) \right|
\leq C_K \alpha^{1/2} t^{-1/2} |z_1-z_2|,
$$
which proves equicontinuity.
Being a logarithmic potential of the Wigner law, the function $\Psi(z)$ is subharmonic on $\C$; see~\cite[Theorem~3.1.2]{ransford}.  Hence, $z\mapsto f_t(\alpha, z) -g(1)$ is subharmonic for every $\alpha\in [0,1]$. Being a pointwise maximum of these functions and continuous, the function $U_{t}(z)$ is subharmonic as well by~\cite[Theorem~3.4.2]{ransford}.
\end{proof}

\begin{proof}[Proof of Proposition \ref{prop:log_potential_large_z}]
Suppose for simplicity that $g$ is differentiable.  By Lemma~\ref{lem:f_der_in_alpha},  for all $\alpha \in (0,1)$,
$$
\frac{\partial}{\partial \alpha} f_t(\alpha, z) = g'(\alpha) + \log \left|\frac{z}{2} + \sqrt{\frac{z^2}{4} - t\alpha}\right|.
$$
Our assumptions on $g$ (including its concavity and $g'_-(1) > -\infty$)  imply that $g'(\alpha)$ is bounded below on $(0,1)$. It follows that, for sufficiently large $|z|$, the derivative is positive and the function $\alpha \mapsto f_t(\alpha, z)$ is strictly increasing. Hence, its supremum is attained at $\alpha = 1$. Inserting this value into~\eqref{eq:log_potential_limit} gives the first claim. If $g$ is not differentiable, one argues similarly using one-sided derivatives.

The statement about the analytic moments then follows from the first claim, since the analytic moments appear as the coefficients in the Laurent expansion of the Cauchy transform---which is obtained by differentiating the log potential---near infinity. 
\end{proof}

\subsection{Collapse to the Wigner law: Proof of Theorem \ref{theo:collapse_to_wigner} and Proposition \ref{prop:t_sing<t_wig}}
\begin{proof}[Proof of Theorem \ref{theo:collapse_to_wigner}]
We are going to prove that for $t\ge t_{\mathrm{Wig}}$ and all $z\in\C$ the maximum of $f_t(\cdot,z)$ is attained at $\alpha=1$, that is $f_t(\alpha,z) \leq f_t(1,z)$ for all $\alpha \in (0,1)$. This would imply that $U_t(z) = f_t(1,z)- g(1)= \log \sqrt{t}+\Psi(z/\sqrt{t})$, which is the logarithmic potential of the semicircle distribution $\mathsf{sc}_t$ supported on $[-2\sqrt t , +2\sqrt t ]$; see  Section \ref{sec:Prop_log_pot}.    For  $t>0$ and by the definition of $f_t$, the inequality $f_t(\alpha,z) \leq f_t(1,z)$ is equivalent to the claim
\begin{align}\label{eq:comparison_alpha_to_1}
g(\alpha)-g(1)+(\alpha-1)\log\sqrt t+\alpha\log\sqrt{\alpha}\le \Psi\left(\frac{z}{\sqrt{t}}\right)-\alpha\Psi\left(\frac{z}{\sqrt{t\alpha}}\right)
\end{align}
for all $z\in\C$. Let as suppose for the moment that the minimum  of the right-hand side (considered as a function of $z$) is attained at $z=0$,\footnote{Note that if the maximizer of $f_t(\cdot,z)$ is a well-defined function $\alpha_t(z)$, then by \eqref{eq:alpha_0_measure} or \eqref{eq:alpha_t}, respectively, it as an cumulative distribution function of disks or ellipsoids, respectively. Hence, we may think of it as an ``increasing'' function along certain paths from $z=0$ to $z=\infty$ and this monotonicity makes it reasonable to guess that $\alpha_t(0)=1$ already implies $\alpha_t(z)\equiv 1$.} which we will explain below. Then, with $\Psi(0)=-1/2$ we can rephrase \eqref{eq:comparison_alpha_to_1} as
\begin{align*}
g(\alpha)-g(1)+\frac {\alpha}2\log(\alpha)+\frac {1-\alpha} 2 \le \frac {1-\alpha} 2 \log t.
\end{align*}
This is precisely our assumed bound
\begin{align*}
t\ge \exp \left( 2 \cdot \frac{g(\alpha)-g(1)}{1-\alpha}+\frac{\alpha}{1-\alpha}\log\alpha+1\right)
\end{align*}
for all $\alpha\in(0,1)$.

Conversely, if $0<t<t_{\mathrm{Wig}}$, then there exists an $\alpha\in (0,1)$ such that \eqref{eq:comparison_alpha_to_1} is violated at $z=0$ and, by continuity of $\Psi$, in a neighborhood of $z=0$. Hence, $U_t(z)$ is strictly larger than the logarithmic potential of the semicircle distribution $\mathsf{sc}_t$ in a neighborhood of $z=0$ and  we conclude that $\nu_t$ is not  $\mathsf{sc}_t$. On the other hand, the moments of $\nu_t$ agree with those of $\mathsf{sc}_t$ by Proposition~\ref{prop:log_potential_large_z} and since the latter distribution (being compactly supported) is uniquely determined by its moments (in the class of distributions on $\R$), we conclude that $\nu_t$ cannot be concentrated on $\R$.

In order to show that the right-hand side of \eqref{eq:comparison_alpha_to_1} is minimized at $z=0$, we shall analyze the auxiliary function
$$
h(z)= h(z;\alpha) \coloneqq\Psi(z)-\alpha \Psi(z/\sqrt {\alpha})
$$
for $0< \alpha<1$. Since $\Psi$ is a logarithmic potential of a probability distribution, but as can also be deduced directly from the definition of $\Psi$, we have  $\lim_{|z| \to\infty} {\Psi(z)}/{\log|z|}=1$. In view of $\alpha<1$, this implies that $h(z) \to +\infty$  as $|z|\to\infty$. Since $h$ is continuous on the whole plane,  the minimum of $h$ is attained.
By \eqref{eq:Stieltjes_semicircle}, we have for all $z\in \C\backslash [-2,2]$ that $h$ is differentiable and
$$\partial_z h(z)=\frac{1}{4}(z-\sqrt{z^2-4})-\frac{\sqrt{\alpha}}{4}(z/\sqrt{\alpha}-\sqrt{z^2/\alpha-4})=\sqrt{z^2-4}-\sqrt{z^2-4\alpha}\neq 0.$$
This implies that  $h$ has no critical points outside of $[-2,2]$ since at any such point we have $\frac{\partial}{\partial x} h(z)  = \frac{\partial}{\partial y} h(z) = 0$ implying that $\partial_z h(z) = 0$.
It follows that the minimum of $h$ is attained on $[-2,2]$.
Consequently, recalling~\eqref{eq:Psi_for_x_real_smaller_2}, we obtain
\begin{align*}
h(x)=\frac{x^2}{4}-\frac 1 2-\alpha\left(\frac{x^2}{4\alpha}-\frac 1 2\right)\equiv \frac{\alpha-1}{2},
\qquad
x\in[-2\sqrt{\alpha},+2\sqrt{\alpha}].
\end{align*}
By symmetry, it remains to study $x\in(2\sqrt{\alpha},2)$. On this interval the derivative is, by the identity $\frac{\partial}{\partial x} = 2 \Re \partial_z$ (valid for real-valued functions) and~\eqref{eq:Stieltjes_semicircle},
\begin{align*}
\frac{\dd}{\dd x} h(x)= \frac{\dd}{\dd x} \left(\frac{x^2}{4}-\frac 1 2 - \alpha \Psi(x/\sqrt {\alpha})\right) = \frac 1 2 x-\sqrt \alpha\Re\big(2(\partial_z \Psi)(x/\sqrt{\alpha})\big)=\frac 1 2 \sqrt{x^2-4\alpha}>0.
\end{align*}
Combining the above considerations, the claim $\inf_{z\in \C} h(z;\alpha) = h(0;\alpha) = (\alpha-1)/2$ follows.
\end{proof}

\begin{proof}[Proof of Proposition \ref{prop:t_sing<t_wig}]
Suppose that $\nu_0$ is such that $t_{\mathrm{sing}}= t_{\mathrm{Wig}}$. Our aim is to show that $\alpha_0(s) = \min (s^2/r^2,1)$, for some $r>0$ and all $s\in \C$.
Observe that
$$
t_{\mathrm{sing}} =  t_{\mathrm{Wig}} =\sup_{\alpha\in(0,1)}\exp \left( 2 \cdot \frac{g(\alpha)-g(1)}{1-\alpha}+\frac{\alpha}{1-\alpha}\log\alpha+1\right)
\geq \eee^{-2g'(1)}
$$
as can be seen by taking $\alpha \uparrow 1$. In view of the definition of $t_{\mathrm{sing}}$,
this implies
\begin{align}\label{eq:t_sing<t_wig}
\sup_{s \in  (\eee^{- g'(0)}, \eee^{- g'(1)})} \left\lvert \left(\frac {\alpha_0(s)}{s}\right)' \right\rvert
=
1 / t_{\mathrm{sing}}
\leq
\eee^{2g'(1)}.
\end{align}
Therefore,
\begin{align}\label{eq:contradiction}
\left\lvert \left(\frac {\alpha_0(s)}{s}\right)' \right\rvert\le \eee^{2g'(1)}
\qquad \text{ for all } s \in  (\eee^{- g'(0)}, \eee^{- g'(1)}).
\end{align}
By integration, it follows that for all $s\in (\eee^{- g'(0)}, \eee^{- g'(1)})$,
\begin{equation}\label{eq:contradiction_tech1}
\left|\frac{\alpha_0(s)}{s}\right|
=
\left|\int_{\eee^{-g'(0)}}^{s} \left(\frac {\alpha_0(y)}{y}\right)' \dd y \right|
\leq
\int_{\eee^{-g'(0)}}^{s} \left| \left(\frac {\alpha_0(y)}{y}\right)'\right| \dd y
\leq
(s- \eee^{-g'(0)}) \eee^{2g'(1)}.
\end{equation}
In particular, taking  $s\uparrow \eee^{-g'(1)}$ and recalling that $\alpha_0(\eee^{-g'(1)}) = 1$ gives
\begin{equation}\label{eq:contradiction_tech2}
1 \leq (\eee^{-g'(1)} - \eee^{-g'(0)}) \eee^{g'(1)}.
\end{equation}
It follows that  $\eee^{-g'(0)}=0$.
Suppose first that~\eqref{eq:contradiction} is strict for some $s_0\in (0, \eee^{- g'(1)})$.  Since $\alpha_0(s)$ is continuously differentiable on $(0, \eee^{- g'(1)})$ by    \hyperref[cond:A3]{(A3)},
%and we excluded the case when the function $ {\alpha_0(s)}/{s}$ is constant on this interval,
there exists an open interval $V\ni s_0$ such that~\eqref{eq:contradiction} is strict, for all $s\in V$.
It follows that~\eqref{eq:contradiction_tech1} is strict for all  $s>s_0$. In particular, taking $s\uparrow \eee^{-g'(1)}$ shows that~\eqref{eq:contradiction_tech2} is strict, that is $1 < (\eee^{-g'(1)} -0) \eee^{g'(1)} = 1$, which is  a contradiction. This proves that~\eqref{eq:contradiction} is, in fact, an equality. That is,
$$
\left\lvert \left(\frac {\alpha_0(s)}{s}\right)' \right\rvert = \eee^{2g'(1)}
\qquad \text{ for all } s \in  (0, \eee^{- g'(1)}).
$$
Since the derivative of $\alpha_0(s)/s$ is continuous, we infer that $(\alpha_0(s)/s)'$  is constant on $(0, \eee^{- g'(1)})$. Since $\lim_{s\downarrow 0}\alpha_0(s)/s = 0$ by    \hyperref[cond:A3]{(A3)},  we have $\alpha_0(s)/s = C_1 s$, where $C_1 =  \eee^{2g'(1)}$ or $C_1= - \eee^{2g'(1)}$. The latter case is impossible in view of $\alpha_0(s) \geq 0$. This proves that $\alpha_0(s) = \eee^{2g'(1)} s^2$ for all $s\in (0, \eee^{- g'(1)})$.
It follows that $\nu_0$ is the uniform distribution on the disk of radius $r= \eee^{- g'(1)}$ centered at the origin, which completes the proof of Proposition \ref{prop:t_sing<t_wig}.
\end{proof}

\subsection{Transport maps and proof of the global push-forward Theorem \ref{theo:pushforward}}\label{subsec:proof_pushforward}%Proof of Proposition \ref{prop:T_t_bijective}, Theorem~\ref{theo:pushforward} and Proposition~\ref{prop:littlewood_offord_push_forward_beta_greater_one_half}}
We need several lemmas in which we analyze the behavior of the transport maps $T_t$ and the Joukowsky-type map $J_{\beta}$. These maps are defined by
\begin{align}
J_{\beta}(y)&=y + \frac{\beta}{y},\quad  y\in \C\backslash\{0\}, \quad  \text{ for a parameter } \beta>0,\label{eq:J_r_t_def}\\
T_t (w)&= w+t\frac{\alpha_0(w)}{w}, \quad w\in \C\backslash\{0\}, \quad  \text{ for a parameter } t>0. \label{eq:t_t_def_2}
\end{align}
Note that, on each particular circle $\{|w|=r\}$, the map $T_t$ coincides with $J_{t\alpha_0(r)}$ since $\alpha_0(w)$ only depends on $|w|$, but, in general, these maps are different. The next lemma records some standard properties of the Joukowsky-type map $J_\beta$.

\begin{lemma}\label{lem:J_r_t_maps_conformally}
Fix any $\beta>0$.  %and recall that $J_{\alpha,t}(y)=y+t\alpha/y$.
\begin{enumerate}
\item[(a)] The function $J_{\beta}$ maps both $\{|y|< \sqrt{\beta}\}$ and $\{|y|>\sqrt{\beta}\}$ conformally to $\C\backslash [-2\sqrt{\beta},+2\sqrt{\beta}]$.
% where
%$
%I\coloneqq[-2\sqrt{t\alpha},+2\sqrt{t\alpha}].
%$
\item[(b)] The circle $\{|y| = \sqrt{\beta}\}$ is mapped to the interval $[-2\sqrt{\beta},+2\sqrt{\beta}]$.
\item[(c)] For every $r>0$ with $r\neq \sqrt{\beta}$, the circle $\{|y|= r\}$ is mapped bijectively to the ellipse
$$
\left\{u+ \ii v \in \C:\frac{u^2}{\left(r + \frac \beta r\right)^2} + \frac{v^2}{\left(r - \frac \beta r\right)^2} = 1\right\}.
$$
%with foci at $\pm 2\sqrt{\beta}$.
\item[(d)] Let $r>\sqrt \beta$. The domain $\{|w|>r\}$ is mapped conformally to the exterior of the above ellipse, while the annulus $\{\sqrt \beta < |w| < r\}$ is mapped conformally to the interior of the above ellipse from which the interval $[-2\sqrt{\beta},+2\sqrt{\beta}]$ has been excluded.
\end{enumerate}
\end{lemma}
\begin{proof}
The classical Joukowsky function $z\mapsto z + \frac 1 z$ maps each of the domains $\{|z|>1\}$ and $\{0<|z|<1\}$ conformally to $\C\backslash [-2,+2]$. The unit circle $\{|z| = 1\}$ is mapped to $[-2,2]$.  Writing
$$
J_{\beta}(y) = y+ \frac{\beta}{y} = \sqrt{\beta} \left(\frac{y}{\sqrt{\beta}} + \frac{\sqrt{\beta}}{y}\right),
\quad y\in \C\backslash\{0\},
$$
proves~(a) and~(b). To prove~(c), write $y=r \eee^{\ii \phi}$ and observe that $J_{\beta}(r\eee^{\ii \phi}) = (r + \frac \beta r) \cos \phi + \ii (r-\frac{\beta}{r})\sin \phi$.  Part (d) follows from the previous points upon noting that $J_\beta(\infty) = \infty$.
\end{proof}

In the following lemmas we work under the same assumptions as in Proposition \ref{prop:T_t_bijective}. In particular, $t>0$ is real and satisfies~\eqref{eq:t_sing}, which we write in the form
\begin{equation}\label{eq:t_sing_rep}
\frac 1 t > \sup_{s \in  (\eee^{- g'(0)}, \eee^{- g'(1)})} \left\lvert \left(\frac {\alpha_0(s)}{s}\right)' \right\rvert.
\end{equation}
Note that the second representation of $t_{\mathrm{sing}}$ in \eqref{eq:t_sing} follows writing $s=\eee^{-g'(\alpha)}$, observing that by monotonicity of $g'$ the maximizer can be found within $\alpha\in(0,1)$ and then carrying out the differentiation of $\alpha_0(s)/s$ by using \eqref{eq:alpha_0}.
%As in Theorem~\ref{theo:pushforward}, we suppose that
According to    \hyperref[cond:A3]{(A3)}, we have
\begin{equation}\label{eq:alpha_0_divided_s_0}
\alpha_0'(0)=\lim_{s\downarrow 0} \frac{\alpha_0(s)}{s} = 0.   %  \qquad \text{ if } \eee^{-g'(0)} = 0.
\end{equation}
\begin{lemma}\label{lem:ineq_e_g_prime_sqrt_t_alpha}
We have $\eee^{-g'(\alpha)} > \sqrt {t\alpha}$ for all $\alpha \in (0,1]$.
\end{lemma}
\begin{proof}
By~\eqref{eq:t_sing_rep}, we have
$$
\frac 1t \eee^{-g'(\alpha)} > \int_{\eee^{-g'(0)}}^{\eee^{-g'(\alpha)}} \left(\frac {\alpha_0(s)}{s}\right)' \dint s = \frac {\alpha_0(\eee^{-g'(\alpha)})}{\eee^{-g'(\alpha)}} - 0 = \alpha \eee^{g'(\alpha)},
$$
where we also used that $\alpha_0(\eee^{-g'(\alpha)}) = \alpha$ and $\alpha_0(\eee^{-g'(0)}) = 0$ (if $\eee^{-g'(0)} > 0$) or~\eqref{eq:alpha_0_divided_s_0} (if $\eee^{-g'(0)} = 0$).
\end{proof}
\begin{corollary}\label{cor:lem:ineq_e_g_prime_sqrt_t_alpha}
For all $w\in \C\backslash\{0\}$, we have $|w| > \sqrt{t\alpha_0(w)}$.
\end{corollary}
\begin{proof}
If $|w| \in (\eee^{-g'(0)}, \eee^{-g'(1)})$, then we apply Lemma~\ref{lem:ineq_e_g_prime_sqrt_t_alpha} with $\alpha = \alpha_0(w)$. For $0< |w|\leq \eee^{-g'(0)}$ the claim is trivial since $\alpha_0(w) =0$, while for $|w|\geq \eee^{-g'(1)}$ we have $\alpha_0(w)=1$ and can apply Lemma~\ref{lem:ineq_e_g_prime_sqrt_t_alpha} with $\alpha = 1$.
\end{proof}

\begin{lemma}\label{lem:map_T_t_bijective}
The map $T_t:\C\backslash\{ 0\}\to\C\backslash\{0\}$ is bijective.
\end{lemma}
\begin{proof}
Let us first argue that the functions  $s\mapsto s\pm t \alpha_0(s)/s$ are strictly increasing and continuous self-bijections of the interval $(0,\infty)$. For $s\in (\eee^{- g'(0)}, \eee^{- g'(1)})$, their derivatives  are given by $s\mapsto 1 \pm t (s\alpha_0'(s)-\alpha_0(s))/s^2$, and our condition on $t$, see~\eqref{eq:t_sing}, implies that both functions are strictly increasing. The same conclusion remains in force for $0 < s \leq \eee^{- g'(0)}$ since then $\alpha_0(s) = 0$, and for $s\geq \eee^{- g'(1)}$ since then $\alpha_0(s) = 1$ and the derivatives of the functions $s\mapsto s \pm t/s$, which are given by $s\mapsto 1 \mp t/s^2$, are positive for $s\geq \eee^{-g'(1)}$ (assuming that $\eee^{-g'(1)}< +\infty$ since otherwise the claim is empty). The latter claim follows from the inequality $t < \eee^{-2g'(1)}$ which holds by Lemma~\ref{lem:ineq_e_g_prime_sqrt_t_alpha}.
%can be shown as follows: By~\eqref{eq:t_sing},
%$$
%\frac 1t \eee^{-g'(1)} > \int_{\eee^{-g'(0)}}^{\eee^{-g'(1)}} \left(\frac {\alpha_0(s)}{s}\right)' \dint s = \frac %{\alpha_0(\eee^{-g'(1)})}{\eee^{-g'(1)}} - 0 = \eee^{g'(1)},
%$$
To summarize, the functions  $s\mapsto s\pm t \frac{\alpha_0(s)}{s}$ are strictly increasing and continuous on $s>0$. To prove that they are self-bijections of $(0,\infty)$, it remains to observe that
$$
\lim_{s\downarrow 0} \left(s\pm t \frac{\alpha_0(s)}{s}\right) = 0,
\qquad
\lim_{s\uparrow +\infty} \left(s\pm t \frac{\alpha_0(s)}{s}\right) = +\infty.
$$
In the former assertion we used~\eqref{eq:alpha_0_divided_s_0}, whereas the latter one follows from $0\leq \alpha_0(s)\leq 1$.

Let $s>0$ be arbitrary.  It follows from the definition of $T_t$ and from the formula $\alpha_0(w) = \alpha_0(|w|)$ that, on the circle $\{|w|= s\}$,  the map $T_t$ coincides with the Joukowski-type function $w\mapsto w + t\alpha_0(s)/w$.
It follows from Lemma~\ref{lem:J_r_t_maps_conformally} that $T_t$ is  a bijection between the circle $\{|w|= s\}$ and the ellipse $\partial \mathcal E_{s,t}$, where
\begin{equation}\label{eq:def_E_r_mathcal_ellipse}
\mathcal E_{s,t}\coloneqq\left\{z=u+iv\in\C:\frac{u^2}{\left(s+t\frac{\alpha_0(s)}{s}\right)^2}+\frac{v^2}{\left(s-t\frac{\alpha_0(s)}{s}\right)^2}\leq 1\right\}.
\end{equation}
To compute the semi-axes of $\mathcal E_{s,t}$, just compute $T_t(w)$ for  $w= \pm s$ and $w= \pm \ii s$.
It follows from the monotonicity established above that the ellipses $\mathcal E_{s,t}$ are nested in the sense that $\mathcal E_{s,t}\subset \mathcal E_{s',t}$ for $0<s<s'<\infty$. Moreover,  the boundaries of the ellipses $\mathcal E_{s,t}$ form a foliation of $\C\backslash\{0\}$ in the sense that they are disjoint and their union is $\C\backslash\{0\}$. Hence $T_t:\C\backslash \{0\}\to \C\backslash \{0\}$ is a bijection.
\end{proof}

Later, in \eqref{eq:DT_t>0} of the proof of Theorem \ref{theo:pushforward} (b), we shall see that the Jacobian of $T_t$ is strictly positive, providing an analytic proof of the bijectivity of $T_t$, which however lacks the geometric insight of the proof above.

\begin{lemma}\label{lem:T_t_maps_circles_to_ellipses}
Let $s>0$ be arbitrary. The map $T_t$ is a bijection between the following pairs of sets:
\begin{itemize}
\item the punctured disk $\{w\in \C: 0< |w| \leq  s\}$ and the punctured ellipse $\mathcal E_{s,t}\backslash\{0\}$.
\item the circle $\{w\in \C: |w| = s\}$ and the ellipse $\partial \mathcal E_{s,t}$.
\item the set $\{w\in \C: |w|>s\}$ and $\C\backslash \mathcal E_{s,t}$.
\end{itemize}
\end{lemma}
\begin{proof}
All these facts have been established in the proof of Lemma~\ref{lem:map_T_t_bijective}.
\end{proof}

Note that Lemmas~\ref{lem:map_T_t_bijective}, \ref{lem:T_t_maps_circles_to_ellipses} and their proofs imply Proposition \ref{prop:T_t_bijective}. %  Part~(d) follows easily from these three parts.
We now turn to the proof of Theorem~\ref{theo:pushforward}, which we prove in the order \textit{(d),(c),(a),(b),(e)}: Part \textit{(d)} provides the unique maximizer, with which we calculate the Stieltjes transform of $\nu_t$ in Part \textit{(c)}. This will then be compared to the Stieltjes transform of the push-forward $(T_t)_\#\nu_0$ for Part \textit{(a)}, which we use to calculate the density in Part \textit{(b)}, implying the differentiability claimed in Part \textit{(e)}.

\begin{proof}[Proof of Theorem~\ref{theo:pushforward}: Part~(d)]
Recall from Lemma~\ref{lem:f_der_in_alpha} that
\begin{equation}\label{eq:d_d_alpha_f_t_rep}
\frac{\partial}{\partial \alpha} f_t(\alpha, z) = g'(\alpha) + \log \left|\frac{z}{2} + \sqrt{\frac{z^2}{4} - t\alpha}\right|.
\end{equation}
Take some $\alpha \in (0,1)$.
%Indeed, by~\eqref{eq:t_sing},
%$$
%\frac 1t \eee^{-g'(\alpha)} > \int_{\eee^{-g'(0)}}^{\eee^{-g'(\alpha)}} \left(\frac {\alpha_0(s)}{s}\right)' \dint s = \frac %{\alpha_0(\eee^{-g'(\alpha)})}{\eee^{-g'(\alpha)}} - 0 = \eee^{g'(\alpha)},
%$$
%and the claim follows.
We  claim that the right-hand side of~\eqref{eq:d_d_alpha_f_t_rep} is $=0$ (respectively, $\leq 0$ or $>0$) if and only if $z\in \partial \mathcal E_{\eee^{-g'(\alpha)},t}$ (respectively, $z\in \mathcal E_{\eee^{-g'(\alpha)},t}$ or $z\in \C \backslash \mathcal E_{\eee^{-g'(\alpha)},t}$) . To prove the claim, consider the quadratic equation $J_{t\alpha}(y) = y + t\alpha/y =z$. By the properties of the Joukowsky-type maps, see Lemma~\ref{lem:J_r_t_maps_conformally},  it has two solutions given by
$$
y_{\pm}= \frac 12 (z\pm \sqrt{z^2 - 4 t\alpha}),
\qquad
|y_-|\leq \sqrt{t\alpha},
\qquad
|y_+|\geq  \sqrt {t\alpha}.
$$
Note that $\sqrt {t\alpha} < \eee^{-g'(\alpha)}$ by Lemma~\ref{lem:ineq_e_g_prime_sqrt_t_alpha}. We can  rewrite the right-hand side of~\eqref{eq:d_d_alpha_f_t_rep} as follows:
$$
g'(\alpha) + \log \left|\frac{z}{2} + \sqrt{\frac{z^2}{4} - t\alpha}\right| = g'(\alpha) + \log |y_+|.
$$
Now, $g'(\alpha) + \log |y_+| = 0$ if and only if $|y_+| = \eee^{-g'(\alpha)}$, which in turn is equivalent to $z\in \partial \mathcal E_{\eee^{-g'(\alpha)},t}$ since the function $y\mapsto y + t\alpha/y$ maps the circle $\{|y| = \eee^{-g'(\alpha)}\}$ to the ellipse
$$
\partial \mathcal E_{\eee^{-g'(\alpha)},t}
=
\left\{u+iv\in\C: \frac{u^2}{\left(\eee^{-g'(\alpha)} +t\alpha \eee^{g'(\alpha)}\right)^2}+\frac{v^2}{\left(\eee^{-g'(\alpha)} -t\alpha \eee^{g'(\alpha)}\right)^2} = 1\right\}.
$$
Similarly, $g'(\alpha) + \log |y_+| \leq  0$ if and only if $\sqrt {t\alpha} \leq |y_+| \leq  \eee^{-g'(\alpha)}$, which is equivalent to $z\in  \mathcal E_{\eee^{-g'(\alpha)},t}$ by Lemma~\ref{lem:J_r_t_maps_conformally}. Finally, $g'(\alpha) + \log |y_+| > 0$ if and only if $|y_+| >  \eee^{-g'(\alpha)}$ (the right-hand side being $>\sqrt {t\alpha}$ by Lemma~\ref{lem:ineq_e_g_prime_sqrt_t_alpha}),  which is  equivalent to $z\in \C \backslash \mathcal E_{\eee^{-g'(\alpha)},t}$. This proves our claim.

Let now $z= T_t(w)$ for some $w\in \mathcal R_0$. Then, $z\in \partial \mathcal E_{|w|,t}$ and $|w| = \eee^{-g'(\alpha_0(w))} \in (\eee^{-g'(0)}, \eee^{-g'(1)})$. By the above claim, the right-hand side of~\eqref{eq:d_d_alpha_f_t_rep} vanishes if $\alpha = \alpha_0(w)$. Moreover, since the ellipses $\mathcal E_{s,t}$ are nested and the function $\alpha \mapsto \eee^{-g'(\alpha)}$ is increasing, we have $z\in \C \backslash \mathcal E_{\eee^{-g'(\alpha)},t}$ provided $0 < \alpha < \alpha_0(w)$ and $z\in \mathcal E_{\eee^{-g'(\alpha)},t}$ provided $\alpha_0(w) \leq \alpha < 1$. By the above claim it follows that the right-hand side of~\eqref{eq:d_d_alpha_f_t_rep} is positive for $0 < \alpha < \alpha_0(w)$ and negative for $\alpha_0(w) \leq \alpha < 1$. Hence, the unique maximizer of the function $\alpha \mapsto f_t(\alpha, z)$ is attained at $\alpha = \alpha_0(w)$. This proves that $\alpha_t(T_t(w)) = \alpha_0(w)$ for $|w|\in (\eee^{-g'(0)}, \eee^{-g'(1)})$.

Let now $z= T_t(w)$ with  $0<|w|\leq \eee^{-g'(0)}$. Then, $z\in \partial \mathcal E_{|w|,t}$ and  since the  $\mathcal E_{s,t}$ are nested, we have $z\in \mathcal E_{\eee^{-g'(\alpha)},t}$ for all $\alpha\in (0,1)$. It follows from the above claim that the function $\alpha \mapsto f_t(\alpha, z)$ is strictly decreasing, whence $\alpha_t(z) = 0 = \alpha_0(w)$.
Similarly,  if $z= T_t(w)$ with $|w|\geq \eee^{-g'(1)}$, then $z$ is outside all ellipses $\mathcal E_{\eee^{-g'(\alpha)},t}$, $\alpha\in (0,1)$,  and it follows that the function $\alpha \mapsto f_t(\alpha, z)$ is strictly increasing, whence $\alpha_t(z) = 1 = \alpha_0(w)$.
\end{proof}

\begin{proof}[Proof of Theorem~\ref{theo:pushforward}: Part~(c)]
Here, we prove the identity \eqref{eq:Stieltjes_pushforward}. Continuity of the Stieltjes transform will be discussed after the proof of Part (b), below.

To compute the Stieltjes transform of $\nu_t$, we shall use the formula $m_t(z) = 2(\partial_z U_{t})(z)$; see Remark \ref{rem:Wirtinger_logpot_Stieltjes}.
If $z = T_t(w)\in \mathcal R_t$ with $w\in \mathcal R_0$, then by Part~(d) $\alpha_t(z)= \alpha_0(w)\in (0,1)$ is the unique  maximum of $\alpha\mapsto f_t(\alpha, z)$ implying that $(\partial_\alpha f_t)(\alpha_t(z),z)=0$. On the other hand, if $z=T_t(w)$ with $|w| < \eee^{-g'(0)}$ or $|w|>\eee^{-g'(1)}$, then $\alpha_t(z) = 0$ or $1$, implying that $\partial_z\alpha_t(z)=0$. In either case, the total differential formula together with~\eqref{eq:log_potential_limit} and~\eqref{eq:f_def} gives
\begin{align*}
\partial_z U_{t}(z)
&=
(\partial_\alpha f)(\alpha_t(z),z)\partial_z \alpha_t(z)+(\partial_z f)(\alpha_t(z),z)
=
(\partial_z f)(\alpha_t(z),z)\\
&=
\alpha_t(z)(\partial_z \Psi) \left(\frac z {\sqrt{t\alpha_t(z)}}\right) \frac{1}{\sqrt{t\alpha_t(z)}}
=
\alpha_0(w)(\partial_z \Psi) \left(\frac z {\sqrt{t\alpha_0(w)}}\right) \frac{1}{\sqrt{t\alpha_0(w)}}
.
\end{align*}
By~\eqref{eq:Stieltjes_semicircle}, this leads to
$$
\partial_z U_{t}(z)
=
\frac 1 {4t} \left(z - \sqrt{z^2-4\alpha_t(z) t}\right)
$$
provided $z\notin [-2\sqrt{\alpha_0(w) t},+2\sqrt{\alpha_0(w) t}]$. In fact, this condition is fulfilled automatically since otherwise we would have
\begin{align}\label{eq:z_not_in_I}
z^2 \leq 4 t \alpha_0(w) < (w + t\alpha_0(w)/w)^2 = (T_t(w))^2
\end{align}
(with a strict inequality by Corollary~\ref{cor:lem:ineq_e_g_prime_sqrt_t_alpha}), which is a contradiction to $z= T_t(w)$.
Therefore, taking $z=T_t(w)$ and using that $\alpha_t(T_t(w))= \alpha_0(w)$ by Part~(d), we get
$$
m_t(T_t(w)) = 2(\partial_z U_{t})(T_t(w)) = \frac 1 {2t} \Big(T_t(w)-\sqrt{(T_t(w))^2-4t\alpha_0(w)}\Big)=\frac{\alpha_0(w)}{w},
$$
which proves the claim in following three cases $w\in \mathcal R_0$, $|w| < \eee^{-g'(0)}$ or $|w|>\eee^{-g'(1)}$. Since $w\mapsto \alpha_t(T_t(w))$ is continuous at $\partial \mathcal R_0$ (by part (d) and the definition \eqref{eq:alpha_0} of $\alpha_0$), so is $w\mapsto m_t(T_t(w))$ (also, as $m_t$ and $T_t$ are continuous) and the claim $m_t(T_t(w))=\alpha_0(w)/w$ for all $w\neq 0$ follows by continuity.
\end{proof}

The next lemma will be needed to prove Part~(a) of Theorem~\ref{theo:pushforward}.

\begin{lemma}\label{lem:roots_y_-_y_+}
Fix some $r>0$. Let $z= T_t(w)$ for some $w\in \C\backslash\{0\}$. The quadratic equation $J_{t\alpha_0(r)} (y) = z$ has two complex solutions $y=y_+$ and $y=y_-$ in $\C \backslash\{0\}$ (which may coincide).
\begin{itemize}
\item[(a)] If $|w|<r$, then the solutions satisfy $|y_-|<r$ and  $|y_+|<r$.
\item[(b)] If $|w|>r$, then one of the solutions satisfies  $|y_-|<r$,   while the other satisfies $|y_+|>r$.
\item[(c)] If $|w| = r$, then one of the solutions satisfies  $|y_-|<r$,   while the other satisfies $|y_+|=r$.
\end{itemize}
\end{lemma}
\begin{proof}
Recall from the proof of Lemma~\ref{lem:map_T_t_bijective} that $s-t\alpha_0(s)/s > 0$ for all $s>0$. Taking $s=r$ this implies that $\sqrt{t\alpha_0(r)}<r$. In particular, $\{|y|>r\} \subset \{|y|>\sqrt{t\alpha_0(r)}\}$. It follows that the function $J_{t\alpha_0(r)}$ is a bijection between the following pairs of sets:
\begin{itemize}
\item the circle $\{|y| = r\}$ and the ellipse $\partial \mathcal E_{r,t}$.
\item $\{|y|>r\}$ and $\C\backslash \mathcal E_{r,t}$.
\end{itemize}

\vspace*{2mm}
\noindent
\textit{Case 0.} Let first $z\in I = [-2\sqrt{t\alpha_0(r)},+2\sqrt{t\alpha_0(r)}]$ and $z\neq 0$ (in particular, $z$ is real). Then, $z^2 \leq 4 t\alpha_0(r) < (r + t\alpha_0(r)/r)^2$ implying that $z$ belongs to the interior of $\mathcal E_{r,t}$; see~\eqref{eq:def_E_r_mathcal_ellipse}. From Lemma~\ref{lem:T_t_maps_circles_to_ellipses} it follows that $|w|<r$, so that we are in the setting of Part~(a). Now, for $z\in I$, the solutions of $J_{t\alpha_0(r)}(y) = z$ satisfy $|y_+| = |y_-|  = \sqrt{t\alpha_0(r)} < r$ by Lemma~\ref{lem:J_r_t_maps_conformally}, which proves the claim of Part~(a).

\vspace*{2mm}
In the following we assume that $z\notin I$. It follows from Lemma~\ref{lem:J_r_t_maps_conformally} that one of the solutions of the equation $J_{t\alpha_0(r)}(y) = z$, denoted by $y_-$, satisfies $|y_-| < \sqrt{t \alpha_0(r)} < r$, while the other solution, denoted by $y_+$, satisfies $|y_+| > \sqrt{t \alpha_0(r)}$. In order to decide whether $|y_+|$ is smaller or larger than $r$, we have to consider the following cases.

\vspace*{2mm}
\noindent
\textit{Case 1.}
Suppose that $|w|<r$ and $z\notin I$.  By Lemma~\ref{lem:T_t_maps_circles_to_ellipses}, $z= T_t(w)$ belongs to the interior of $\mathcal E_{r,t}$. Recall that $J_{t\alpha_0(r)}$ maps $\{|y|\geq r\}$ conformally to the complement of the interior of $\mathcal E_{r,t}$. It follows that $|y_+|<r$.

\vspace*{2mm}
\noindent
\textit{Case 2.}
Suppose that $|w|>r$ and $z\notin I$.  By Lemma~\ref{lem:T_t_maps_circles_to_ellipses}, $z = T_t(w) \in \C\backslash \mathcal E_{r,t}$. Since $J_{t\alpha_0(r)}$ is a conformal map between $\{|y|>r\}$ and $\C\backslash \mathcal E_{r,t}$, it follows that  $|y_+| >r$.
%On the other hand, $J_{t\alpha_0(r)}$ is a conformal map between $\{0 < |y| < \sqrt{t \alpha_0(r)}\}$  and $\C\backslash I$. It follows that the other solution satisfies $|y_-| < \sqrt{t \alpha_0(r)} < r$.
%In agreement with the definition above, $y_{\pm}$ are the two solutions of $J_{t\alpha_0(r)}(y)=T_t(w)\equiv z$ such that $y_-\in B_{\sqrt{t\alpha_0(r)}}$ and $y_+\in \C\backslash B_{\sqrt{t\alpha_0(r)}}$.

\vspace*{2mm}
\noindent
\textit{Case 3.} Suppose that $|w|=r$ and $z\notin I$. By Lemma~\ref{lem:T_t_maps_circles_to_ellipses}, $z = T_t(w)\in \partial \mathcal E_{r,t}$ and since $J_{t\alpha_0(r)}$ is a bijection between $\{|y|=r\}$ and $\partial \mathcal E_{r,t}$, it follows that  $|y_+| = r$. Note that by the definition of $T_t(w)$ it actually holds $y_+=w$.
\end{proof}

\begin{proof}[Proof of Theorem~\ref{theo:pushforward}: Part~(a)]
%Let us prove Part~(a) of Theorem~\ref{theo:pushforward}  .
We will compute the Stieltjes transform $\tilde m_t(z)$ of the push-forward measure $(T_t)_\#\nu_{0}$ with the help of residue calculus and show that it coincides with the Stieltjes transform $m_t(z)$ of $\nu_t$.  Fix some $z = T_t(w) \in \C \backslash\{0\}$ with  $w\in \C\backslash\{0\}$.   Recall that  $\alpha_0(y)=\alpha_0(|y|)$.  By \eqref{eq:alpha_0}, $\nu_{0}$ has a two-dimensional  Lebesgue density $\alpha_0'(|y|)/(2\pi|y|)$, $y\in \C\backslash\{0\}$.  Using polar coordinates and representing the angular integral as a contour integral leads to
\begin{align*}
\tilde m_{t}(z)
&\coloneqq\int_{\C}\frac{1}{z-u} ((T_t)_\#\nu_{0})(\dint u)\\
&=\int_{\C}\frac{1}{z-T_t(y)} \nu_{0}(\dint y)\\
&=\frac 1 {2\pi} \int_0^\infty \int_0^{2\pi}\frac{1}{z-T_t(r \eee^{i\varphi})}\alpha_0'(r) \, \dint r \, \dint \varphi\\
&=\int_0^\infty  \frac{1}{2\pi i} \oint_{\partial B_r(0)}\frac{1}{yz-y^2-t\alpha_0(r)} \, \dint y  \, \alpha_0'(r)\dint r,
\end{align*}
where we used the substitution $y= r \eee^{i\varphi}$ and~\eqref{eq:t_t_def_2}.
For every $r>0$ we consider the quadratic polynomial
$$
f_r(y) = yz - y^2 - t\alpha_0(r)=  y (z-J_{t\alpha_0(r)}(y)).
$$

\vspace*{2mm}
\noindent
\textit{Case 1.}
If $|w|<r$, then by Lemma~\ref{lem:roots_y_-_y_+} the roots of $f_r$ satisfy $|y_-| <r$ and $|y_+|<r$. We claim that
\begin{equation}\label{eq:residue_counting0}
\frac 1 {2\pi \ii} \oint_{\partial B_r(0)}\frac{\dint y}{yz-y^2-t\alpha_0(r)}  = 0,
\qquad
|w|<r.
\end{equation}
Indeed, we can therefore deform the integration contour to be $\partial B_s(0)$ with any $s>r$ and then let $s\to +\infty$. Observing that $1/f_r(y) \sim -1/y^2$ as $|y|\to\infty$ while the length of $\partial B_s(0)$ is $2\pi s$ proves the claim.

\vspace*{2mm}
\noindent
\textit{Case 2.} Let $0< r < |w|$. Then, by Lemma~\ref{lem:roots_y_-_y_+}, the roots of $f_r$ satisfy $|y_-|<r< |y_+|$. In the proof of the same lemma we showed that $z\notin I = [-2\sqrt{t\alpha_0(r)},+2\sqrt{t\alpha_0(r)}]$.  The roots of $f_r$ are given by
\begin{align*}
y_{\pm}(r)=\frac 1 2 \Big(z\pm\sqrt{z^2-4t\alpha_0(r)}\Big).
\end{align*}
%Note that $y_- y_+ = t\alpha_0(r)$. If $r>0$ is such that $z\in [-2\sqrt{t\alpha_0(r)}, +2\sqrt{t\alpha_0(r)}]$, then $|y_+| = |y_-|= \sqrt{t\alpha_0(r)}$. If, on the other hand, $z\in \C \backslash [-2\sqrt{t\alpha_0(r)}, +2\sqrt{t\alpha_0(r)}]$,  then we agree that $|y_+| >  \sqrt{t\alpha_0(r)} >  |y_-|$.
Here, $\sqrt{\,\cdot\,}$ has a branch cut along the negative axis (in which the argument never falls since $z\notin I$) and the sign is chosen such that $y_-(r) \to 0$ and $y_+(r) \to z$ as $r\downarrow 0$.   From the residue theorem it follows that
\begin{align}
\frac{1}{2\pi i}\oint_{\partial B_r(0)}\frac{1}{yz-y^2-t\alpha_0(r)} dy
=
\mathop{\mathrm{Res}}_{y= y_-}\frac{1}{f(y)}
=
\frac{1}{z-2y_-}
=
\frac{1}{\sqrt{z^2-4t\alpha_0(r)}},
\qquad
0< r < |w|.
\label{eq:residue_counting}
\end{align}

With this at hand, we are  ready to complete the proof of Part~(a). It follows from~\eqref{eq:residue_counting0} and~\eqref{eq:residue_counting} that
\begin{align*}
\tilde m_{t}(z)&=\int_0^{|w|}  \frac{\alpha_0'(r)\dint r}{\sqrt{z^2-4t\alpha_0(r)}}
&=\frac{z- \sqrt{z^2 - 4 t\alpha_0(|w|)}}{2t}\\
&=\frac{T_t(w)-\sqrt{(T_t(w))^2-4t\alpha_0(w)}}{2t}
=\frac{\alpha_0(w)}{w}.
\end{align*}
It follows from Part~(c) that $\tilde m_{t}(z) = m_t(z)$, which yields $(T_t)_\#\nu_{0}= \nu_t$. The proof is complete.
\end{proof}

\begin{proof}[Proof of Theorem~\ref{theo:pushforward}: Part~(b)]
Let us first consider the density $p_0$ of $\nu_0$. By \eqref{eq:alpha_0_measure}, it is given by $p_0(z)=\frac{\alpha_0'(|z|)}{2\pi |z|}$.
Our assumption $t_{\mathrm{sing}}>0$ is equivalent to $\sup_{s}|\frac{\alpha_0'(s)}{s}-\frac{\alpha_0(s)}{s^2}|<\infty$, where the supremum is taken over $s\in (\eee^{-g'(0)},\eee^{-g'(1)})$. Moreover, by Corollary \ref{cor:lem:ineq_e_g_prime_sqrt_t_alpha} we have $s>\sqrt{t\alpha_0(s)}$ for all $0<t<t_{\mathrm{sing}}<\infty$, or equivalently $\frac{\alpha_0(s)}{s^2}<\frac 1 t$. Thus
\begin{align*}
2\pi\sup_{z\in\C}p_0(z)=\sup_{s\in (\eee^{-g'(0)},\eee^{-g'(1)})}\frac{\alpha_0'(s)}{ s}&<\sup_{s\in (\eee^{-g'(0)},\eee^{-g'(1)})} \left(\frac{\alpha_0'(s)}{s} -\frac{\alpha_0(s)}{s^2}+\frac{1}{ t}\right) \\
&\leq \frac 1 {t_{\mathrm{sing}}}+\frac 1 t<\infty.
\end{align*}
By the push-forward statement of Part (a), it holds $p_t(z) = p_0(T_t^{-1} (z)) |DT_t^{-1}(z)|$, and since $p_0$ is bounded, it remains to show that the Jacobian determinant $|D T_t^{-1}|$ stays bounded. In other words, we have to show that $|DT_t(w)|$ stays away from $0$ for $|w|\in (\eee^{-g'(0)},\eee^{-g'(1)})$, where
\begin{align}\label{eq:Jacobian}
|DT_t(w)|&=\left\vert\left(
\begin{array}
[c]{cc}%
\partial_w T_t(w) & \partial_{\bar{w}} T_t(w)\\
\partial_w\overline{T_t(w)} & \partial_{\bar w}\overline{T_t(w)}%
\end{array}
\right)\right\vert\\
&=|\partial_w T_t(w)|^2-|\partial_{\bar w}T_t(w)|^2=\left\lvert 1+t\frac{\partial_w\alpha_0(w)}{w}-t\frac{\alpha_0(w)}{w^2}\right\rvert^2-\left\lvert t\frac{\partial_{\bar w}\alpha_0(w)}{w}\right\rvert^2.\nonumber
\end{align}
Observe that the complex derivatives of the real-valued and rotation invariant  function $\alpha_0(w) = \alpha_0(|w|)$ can be written as $\partial_w\alpha_0(w)=\frac{\bar w}{2|w|}\alpha_0'(|w|)$ and $\partial_{\bar w}\alpha_0(w)=\frac{w}{2|w|}\alpha_0'(|w|)$. Hence,
\begin{align*}
|\partial_w T_t(w)|=\left\vert 1  + t\frac{|w|^2}{w^2}\left(\frac{\alpha_0'(|w|)}{2|w|}-\frac{\alpha_0(|w|)}{|w|^2}\right) \right\vert \geq  1 -t\frac{\alpha_0'(|w|)}{2|w|}+t\frac{\alpha_0(|w|)}{|w|^2},
\end{align*}
since $t>0$ and $|w|^2/w^2$ has absolute value $1$. Using again the assumption $t<t_{\mathrm{sing}}$ (see \eqref{eq:t_sing_rep}), we obtain
\begin{align*}
|\partial_w T_t(w)| \geq 1 - t\frac{\alpha_0'(|w|)}{2|w|}+ t\frac{\alpha_0(|w|)}{|w|^2}\geq 1  -\frac t {t_{\mathrm{sing}}}+ t\frac{\alpha_0'(|w|)}{2|w|}>0.
\end{align*}
Noting that $|\partial_{\bar w} T_t(w)| = t\frac{\alpha_0'(|w|)}{2|w|}$, we conclude that
\begin{align}\label{eq:DT_t>0}
|DT_t(w)|&=|\partial_w T_t(w)|^2-|\partial_{\bar w}T_t(w)|^2\\
&\geq \left(1-\frac t {t_{\mathrm{sing}}}+t\frac{\alpha_0'(|w|)}{2|w|}\right)^2-\left(t\frac{\alpha_0'(|w|)}{2|w|}\right)^2>\left(1-\frac t {t_{\mathrm{sing}}}\right)^2 > 0,\nonumber
\end{align}
which completes the proof.
\end{proof}

Now, the continuity of the Stieltjes transform $m_t$ on $\C$ claimed in Part (c) follows from the fact that $m_t(z) = \int_{\C} \frac{1}{z - y} p_t(y)\dint y$ converges to $m_t(z_0)$ as $z\to z_0$ since the family $\{1/(z-y)\}_{z\in B_1(z_0)}$ is uniformly locally integrable in $\C$ and $p_t$ is bounded with bounded support $\mathcal R_t$.

\begin{proof}[Proof of Theorem~\ref{theo:pushforward}: Part~(e)]
Following up on the previous discussion,  $\Re m_t$ and $-\Im m_t$ are partial weak derivatives of $U_t$ according to $m_t=2\partial_zU_t$, hence $U_t$ is weakly (real) differentiable with continuous weak derivatives, implying that $U_t$ is continuously differentiable in the classical sense on the whole of $\C$.

Let now $w\in \C\backslash\{0\}$. We already know from Part~(d) that $\alpha_t(T_t(w))=\alpha_0(w)$.
It follows from~\eqref{eq:log_potential_limit} and~\eqref{eq:f_def} that
\begin{align*}
U_{t}(T_t(w))
&=
f_t(\alpha_0(w),T_t(w))-g(1)\\
&=
g(\alpha_0(w))-g(1)+\alpha_0(w)\log\sqrt{t\alpha_0(w)}
+
\alpha_0(w)\Psi\left(\frac{w}{\sqrt{t\alpha_0(w)}} + \frac{\sqrt{t\alpha_0(w)}} {w}\right).
\end{align*}
By Corollary~\ref{cor:lem:ineq_e_g_prime_sqrt_t_alpha}, $|w| > \sqrt{t\alpha_0(w)}$.  We now use the Joukowsky substitution:
$$
u\coloneqq\frac{w}{\sqrt{t\alpha_0(w)}}, \qquad |u|>1, \qquad z\coloneqq u + \frac 1 u, \qquad z\in \C \backslash[-2,2].
$$
With this notation, we have $\sqrt{z^2-4} = u - \frac 1 u$ and the definition of $\Psi$, see~\eqref{eq:Psi_def_rep}, simplifies to
\begin{equation*}
\Psi(z) = \Psi\left(u + \frac 1u \right) =  \frac 12 \Re \left(\frac 1 {u^2}\right) + \log |u|.
\end{equation*}
Taking everything together we arrive at
$$
U_{t}(T_t(w))
=
g(\alpha_0(w))-g(1) + \frac{\alpha_0(w)^2}{2} \Re \left(\frac{t}{w^2}\right) + \alpha_0(w) \log |w|,
$$
which completes the proof since $U_0(w) = g(\alpha_0(w)) + \alpha_0(w) \log |w| - g(1)$ by~\eqref{eq:nu_g}.
\end{proof}

\begin{proof}[Proof of Proposition~\ref{prop:littlewood_offord_push_forward_beta_greater_one_half}]
Fix some $t  \in (0,1)$.
First of all, we show that the push-forward of the restriction of $\nu_0$ to the disk $\{|w|< t^{\beta/(2\beta-1)}\}$ is the Wigner law on $[-2t^{\beta/(2\beta-1)}, 2t^{\beta/(2\beta-1)}]$ multiplied by $t^{1/(2\beta-1)}$. To this end, consider some infinitesimal annulus $\{r < |w|< r + \dint r\}$, where $0 < r < t^{\beta/(2\beta-1)}$. The total $\nu_0$-measure of this annulus is $\dint (r^{1/\beta}) = (1/\beta) r^{(1/\beta) - 1} \dint r$. The map $\tilde T_t$ maps any point $w$ from this annulus to $(\Re w) \cdot 2\sqrt t r^{1/(2\beta)-1}$. The uniform probability distribution on the infinitesimal annulus is thus mapped to the arcsine distribution\footnote{We use the following standard fact: If $(X,Y)$ is uniformly distributed on the boundary of the unit circle, then $X$ has the arcsine density $x\mapsto 1/(\pi\sqrt{1-x^2})$ on the interval $(-1,1)$.} on the interval $[-2\sqrt t r^{1/(2\beta)}, +2\sqrt t r^{1/(2\beta)}]$. The density is
$$
f_{r;t}(a) =  \frac 1 {\pi \sqrt{4 t r^{1/\beta} - a^2}} \ind_{|a|<2\sqrt t r^{1/(2\beta)}}.
$$
Let us now fix some $a\in (0, 2t^{\beta/(2\beta-1)})$. It follows from the formula for $f_{r;t}(a)$ that the roots from the annulus $\{r < |w|< r + \dint r\}$ do not reach the point $a$ at time $t$ if $2\sqrt t r^{1/(2\beta)} < a$ or,  equivalently, $r < (a/(2\sqrt t))^{2\beta}$. Since the measure $\nu_0$ is rotationally invariant, it can be disintegrated into uniform distributions on the infinitesimal annuli  of the form $\{r < |w|< r + \dint r\}$. For the density of the push-forward of $\nu_0$ restricted to the disk $\{|w|< t^{\beta/(2\beta-1)}\}$ under $\tilde T_t$ we obtain the formula
\begin{equation}\label{eq:LO_singular_part_wigner}
p_{t}(a) = \int_{(a/(2\sqrt t))^{2\beta}}^{t^{\beta/(2\beta-1)}} \frac {\dint (r^{1/\beta})} {\pi \sqrt{4 t r^{1/\beta} - a^2}}
=
\int_0^{\sqrt{4 t \cdot t^{1/(2\beta-1)} - a^2}} \frac {\dint x} {2\pi t} = \frac {\sqrt{4 t^{2\beta/(2\beta-1)} - a^2}}{2\pi t},
\end{equation}
where we used the substitution $x= \sqrt{4 t r^{1/\beta} - a^2}$ with  $\dint x = (2 t / x) \dint (r^{1/\beta})$. The right-hand side is
the density of the Wigner law on $[-2t^{\beta/(2\beta-1)}, 2t^{\beta/(2\beta-1)}]$ multiplied by $t^{1/(2\beta-1)}$, which proves our claim.

The remainder of the proof of Proposition~\ref{prop:littlewood_offord_push_forward_beta_greater_one_half} follows the same method as the proof of Theorem~\ref{theo:pushforward} with some modifications which we are now going to describe.   It follows from Lemma~\ref{lem:J_r_t_maps_conformally} that the points $w\in \C$ with fixed $|w| = r \in (t^{\beta/(2\beta-1)}, 1)$ are mapped by $\tilde T_t$ to the ellipse $\partial \mathcal E_{r,t}$ with half-axes $r + t r^{(1/\beta)-1}$ and $r - t r^{(1/\beta)-1}$.
The crucial observation is that the functions $r\mapsto r +  t r^{(1/\beta) - 1}$ and $r\mapsto  r -  t r^{(1/\beta) - 1}$ are monotone increasing as long as $r\in [t^{\beta/(2\beta-1)} , 1]$. Indeed, their derivatives are $r\mapsto 1 \pm ((1/\beta)-1) t r^{(1/\beta)-2}$. The minimum of these derivatives is attained either at $r= t^{\beta/(2\beta-1)}$ or at $r=1$. Taking $r=1$ gives the value $1 \pm ((1/\beta)-1) t \geq 1 - |(1/\beta)-1| > 0$ since $\beta>1/2$ and $0< t < 1$. Taking $r= t^{\beta/(2\beta-1)}$ gives the value $1 \pm ((1/\beta)-1) \geq 1 - |(1/\beta)-1| > 0$ since $\beta>1/2$. So, the half-axes of the ellipses $\partial \mathcal E_{r,t}$ with $r\in (t^{\beta/(2\beta-1)} , 1)$ are monotone increasing and it follows that these ellipses foliate the ellipse $\mathcal E_t$ with half-axes $1+t$ and $1-t$ with exception of the interval $[-2t^{\beta/(2\beta-1)}, +2t^{\beta/(2\beta-1)}]$. This makes it possible to repeat the proof of Theorem~\ref{theo:pushforward}, provided we exclude the disk $|w|\leq t^{\beta/(2\beta-1)}$ from consideration. In particular, one shows that $\alpha_t(T_t(w)) = \alpha_0(w)$ for all $w\in \C$ with $|w|> t^{\beta/(2\beta-1)}$ and that the Stieltjes transform of $\nu_t$ is given by
$$
m_{t}(T_t(w))
\coloneqq
\int_{\C}\frac{\nu_{t}(\dint y)}{T_t(w) - y}  = \frac{\alpha_0(w)}{w},
\qquad
|w|> t^{\beta/(2\beta-1)}.
$$

It remains to compute the Stieltjes transform of the push-forward $(\tilde T_t)_\#\nu_{0}$. We decompose it into two parts as follows:
\begin{align*}
\tilde m_{t}(z)
\coloneqq\int_{\C}\frac{\nu_{0}(\dint y)}{z- \tilde T_t(y)}
=\int_{\{|u| < t^{\beta/(2\beta-1)}\}} \frac{\nu_{0}(\dint y)}{z- \tilde T_t(y)}
+
\int_{\{t^{\beta/(2\beta-1)} < u < 1\}} \frac{\nu_{0}(\dint y)}{z- \tilde T_t(y)}
=: I_1 + I_2.
\end{align*}
The first integral can be evaluated using formula~\eqref{eq:wigner_stieltjes} for the Stieltjes transform of the Wigner measure~\eqref{eq:LO_singular_part_wigner}:
$$
I_1 \coloneqq \int_{\{|u| < t^{\beta/(2\beta-1)}\}} \frac{\nu_{0}(\dint y)}{z- \tilde T_t(y)}
=
t^{1/(2\beta-1)} \cdot \frac {z - \sqrt{z^2 -4t^{2\beta/(2\beta-1)}}} {2t^{2\beta/(2\beta-1)}}
=
\frac {z - \sqrt{z^2 -4t^{2\beta/(2\beta-1)}}} {2t}
.
$$
The second integral can be evaluated by repeating the arguments used in the proof of Theorem~\ref{theo:pushforward},  Part~(a). Recall that the measure $\nu_0$ is rotationally invariant and $\nu_0(\{|z|<r\}) = r^{1/\beta}$ for $0\leq r\leq 1$. Hence,
\begin{align*}
I_2\coloneqq \int_{\{t^{\beta/(2\beta-1)} < u < 1\}} \frac{\nu_{0}(\dint y)}{z- \tilde T_t(y)}
&= \frac 1 {2\pi} \int_{t^{\beta/(2\beta-1)}}^1 \int_0^{2\pi}\frac{\dint \varphi}{z- \tilde T_t(r \eee^{i\varphi})}  \, \dint (r^{1/\beta})\\
&= \int_{t^{\beta/(2\beta-1)}}^1  \frac{1}{2\pi i} \oint_{\partial B_r(0)} \frac{\dint y}{yz-y^2-tr^{1/\beta}}   \, \dint (r^{1/\beta}),
\end{align*}
where we used the substitution $y= r \eee^{i\varphi}$ and the formula $\tilde T_t(y) = y + t r^{1/\beta}/y$. The inner integral can be analyzed using residue calculus as in the proof of Theorem~\ref{theo:pushforward},  Part~(a). With the notation $z= \tilde T_t(w) \in \C\backslash [-2t^{\beta/(2\beta-1)}, 2t^{\beta/(2\beta-1)}]$ for some $w\in \C$ with $|w|> t^{\beta/(2\beta-1)}$, the result is
$$
\frac{1}{2\pi i} \oint_{\partial B_r(0)} \frac{\dint y}{yz-y^2-tr^{1/\beta}}
=
\frac{1}{\sqrt{z^2-4t r^{1/\beta}}},
\quad
\text{ if }
\quad
t^{\beta/(2\beta-1)} < r < |w|,
$$
while for $|w| < r$ the integral vanishes.
Recall that $\alpha_0(w) = \min \{1, |w|^{1/\beta}\}$. It follows that
\begin{align*}
I_2
=
\int_{t^{\beta/(2\beta-1)}}^{\min\{|w|,1\}}  \frac{\dint (r^{1/\beta})}{\sqrt{z^2-4t r^{1/\beta}}}
&=
-\int_{\sqrt{z^2 - 4 t \cdot t^{1/(2\beta-1)}}}^{\sqrt{z^2 - 4t \alpha_0(w)}} \frac{\dint x}{2t}\\
&=
\frac {\sqrt{z^2 - 4  t^{2\beta/(2\beta-1)}} - \sqrt{z^2 - 4t \alpha_0(w)}} {2t}, % \left(\right),
\end{align*}
where we used the substitution $x = \sqrt{z^2-4t r^{1/\beta}}$ with $\dint x = - (2t/x)  \dint (r^{1/\beta})$. Taking our formulas for $I_1$ and $I_2$ together and recalling that $z= \tilde T_t(w) = w + t \alpha_0(w)/w$ gives
$$
\tilde m_{t}(z)
=
\int_{\C}\frac{\nu_{0}(\dint y)}{z- \tilde T_t(y)}
=
\frac {z - \sqrt{z^2 - 4t \alpha_0(w)}} {2t}
=
\frac {z - \sqrt{z^2 - 4t \alpha_0(w)}} {2t}
=
\frac{\alpha_0(w)}{w}
.
$$
By comparing the Stieltjes transforms, it follows that $\nu_t$ coincides with the push-forward $(\tilde T_t)_\#\nu_{0}$.
%On the other hand, the Stieltjes transform of $\nu_t$ can be computed as in the proof of Theorem~\ref{theo:pushforward}, Part~(h).
%With the same notation $z= T_t w$ as above, the result is
%$$
%\tilde m_{t}(z)
%=
%\int_{\C}\frac{\nu_{0}(\dint y)}{z- \tilde T_ty} m_t(z) =\frac{\alpha_0(w)}{w}.
%$$
\end{proof}

\subsection{Proofs of the local results around Theorem~\ref{theo:local_push_forward}}\label{subsec:proof_local}
\begin{proof}[Proof of Proposition~\ref{prop:Stieltjes}]
For all $(z,t) \in \mathcal D$, the maximizer $\alpha_t(z)$ of $f_t(\alpha,z)$  satisfies $\partial_z\alpha_t(z)=0$ (in cases D0 and D1) or $(\partial_\alpha f_t)(\alpha_t(z),z)=0$ (in case D2). In all three cases,
\begin{equation}\label{eq:prod_ders_vanish}
(\partial_\alpha f_t)(\alpha_t(z),z)\partial_z \alpha_t(z) = 0.
\end{equation}
We denote the derivative of $f_t(\alpha, z)$ in its first argument by
$\partial_\alpha f_t$. By~\eqref{eq:f_def} and the chain rule, the Wirtinger derivative of $f_t(\alpha, z)$ in its second argument is given by
\begin{align*}
\partial_z f_t (\alpha, z)
&=
\alpha \partial_z \left[\Psi\left(\frac{z}{\sqrt{t\alpha}}\right)\right]
=
\frac{\sqrt \alpha}{\sqrt t} (\partial_z \Psi)\left(\frac{z}{\sqrt{t\alpha}}\right)
=
\frac{\sqrt \alpha}{4 \sqrt t}  \left(\frac{z}{\sqrt{t\alpha}} - \sqrt{\frac{z^2}{t\alpha} - 4}\right)\\
&=
\frac 1 {4t} \left(z - \sqrt{z^2 - 4\alpha  t}\right)
,
\quad
z\in \C \backslash [-2\sqrt{\alpha t},2\sqrt{\alpha t}],
\end{align*}
where we used the formula $\partial_z \Psi(z) = \frac 14 ( z - \sqrt{z^2 - 4})$ for $z\in \C \backslash [-2,2]$; see~\eqref{eq:Stieltjes_semicircle}.
Recall from Theorem~\ref{theo:main_general_g} that $U_{t}(z) =  f_t(\alpha_t(z),z) - g(1)$.  In cases D1 and D2, the chain rule and~\eqref{eq:prod_ders_vanish} give
\begin{align}
\partial_z U_{t}(z)
&=
(\partial_\alpha f_t)(\alpha_t(z),z)\partial_z \alpha_t(z)+ (\partial_z f_t)(\alpha_t(z),z) \partial_z z + (\partial_{\overline{z}} f_t)(\alpha_t(z),z) \partial_z \overline{z}\notag \\
&=
(\partial_z f_t)(\alpha_t(z),z) \notag\\
&=
\frac 1 {4t} \left(z - \sqrt{z^2 - 4\alpha_t(z)  t}\right)
,
\quad
z\in \C \backslash [-2\sqrt{\alpha_t(z) t},2\sqrt{\alpha_t(z) t}].\label{eq:partial_z_U_t_z}
\end{align}
Note that the latter condition on $z$ is fulfilled by \eqref{eq:z_not_in_I}.   Observe that~\eqref{eq:partial_z_U_t_z} continues to hold in case D0 since then $\alpha_t(z) \equiv 0$ implying  $U_t(z) = f_t(0,z) - g(1) = g(0)- g(1)$ and $\partial_z U_{t}(z) = 0$. In all three cases, we have
$$
m_t(z) = 2 \partial_z U_t(z) = \frac 1 {2t} \left(z - \sqrt{z^2 - 4\alpha_t(z)  t}\right).
$$
Solving this in $\alpha_t(z)$ gives
\begin{equation}\label{eq:alpha_t_as_function_of_m_t_rep}
\alpha_t(z) = m_t(z) (z - t m_t(z)),
\end{equation}
which completes the first part of the claim. The last claim \eqref{eq:m_t_alpha_0} follows from combining previous considerations: By Lemma \ref{lem:f_der_in_alpha} and the definition \eqref{eq:T_inverse} of the local inverse $T_t^{-1}$, $\alpha_0(T_t^{-1})$ is a critical point of $f_t(\cdot,z)$ and by uniqueness it follows $\alpha_t(z)=\alpha_0(T_t^{-1}(z))$.
\end{proof}

\begin{proof}[Proof of Theorem \ref{theo:local_push_forward}]
If $(z,t)$ is regular, we have a local bijection and we shall write $z=T_t(w)$ and $w=T_t^{-1}(z)$, respectively.
Similar to \eqref{eq:Jacobian}, the Jacobian $DT_t^{-1}$ of $T_t^{-1}$ can be rephrased as the inverse Jacobian of $T_t$ by the inverse function theorem
\begin{align*}
\left(
\begin{array}
[c]{cc}%
\partial_z T_t^{-1}(z) & \partial_{\bar{z}} T_t^{-1}(z)\\
\partial_z\overline{T_t^{-1}(z)} & \partial_{\bar z}\overline{T_t^{-1}(z)}%
\end{array}
\right)
=DT_t^{-1}(z)
=DT_t(w)^{-1}
&=\left(
\begin{array}
[c]{cc}%
\partial_w T_t(w) & \partial_{\bar{w}} T_t(w)\\
\partial_w\overline{T_t(w)} & \partial_{\bar w}\overline{T_t(w)}%
\end{array}
\right)^{-1}\\
&=\frac{1}{|DT_t(w)|}\left(
\begin{array}
[c]{cc}%
\partial_{\bar w} \overline{T_t(w)} &- \partial_{\bar{w}} T_t(w)\\ -
\partial_w\overline{T_t(w)} & \partial_{ w}T_t(w)%
\end{array}
\right)
\end{align*}
Recall that by Proposition \ref{prop:Stieltjes} or Remark \ref{rem:Wirtinger_logpot_Stieltjes} the density is given by $\pi p_t(z)=\partial_{\bar z}m_t(z)$. By the chain rule we have
\begin{align*}
\partial_{\bar{z}}m_{t}(z)
&=\partial_{\bar{z}}\big(m_{0}(T_t^{-1}(z))\big)=(\partial_z m_0)(T_t^{-1}(z))\cdot\partial_{\bar z}T_t^{-1}(z) + (\partial_{\bar z}m_0)(T_t^{-1}(z))\cdot\partial_{\bar z}\overline{T_t^{-1}}(z)
\\
&=\frac{1}{|DT_t(w)|}\Big(-(\partial_w m_0)(w)\cdot\partial_{\bar w}T_t(w) + (\partial_{\bar w}m_0)(w)\cdot\partial_{w}T_t(w)\Big).
\end{align*}
But since $T_{t}(w)=w+tm_{0}(w)$, we have $\partial_{\bar w}T_t(w)=t\partial_{\bar w} m_0(w)$ and $\partial_w T_t(w)=1+t\partial_w m_0(w)$
giving
\begin{align*}
\partial_{\bar{z}}m_{t}(z)
&=\frac{1}{|DT_t(w)|}\Big(-\partial_w m_0(w)\cdot t\partial_{\bar w} m_0(w) + \partial_{\bar w}m_0(w)+\partial_{\bar w}m_0(w)\cdot t\partial_w m_0(w)\Big)\\
&=\frac{\partial_{\bar w } m_0(w)}{|DT_t(w)|}.
\end{align*}
After dividing this result by $\pi,$ the theorem follows by the change of
variables theorem: the density of $(T_{t})_{\#}(\nu_{0})$ at $z=T_t(w)$ is just
$1/{|DT_t(w)|}$ times the density of $\nu_{0}$ at $w$.
\end{proof}

\subsection{Optimal transport for Weyl polynomials: Proof of Theorem~\ref{theo:OT}}\label{subsec:proof_OT}
%Let us now turn to optimal transport.
%\begin{proof}
By a rescaling argument of the polynomial $P_n$, and hence its coefficients, we may assume without loss of generality that $g(0)=0$.
Since the Theorem \ref{theo:pushforward} implies that $\nu_0$ does not charge hypersurfaces, we may apply Brenier's theorem \ref{thm:Brenier}.
As described above Theorem \ref{theo:OT}, this implies that $T$ is an optimal transport map iff it is the gradient of a convex function. Indeed, this is the case for Weyl polynomials. For $t<1$, Theorem \ref{theo:pushforward} shows that $\nu_t$ is a push-forward under $T_t$ and the function\footnote{This function is $U_t(T_t(z))$ up to an additive constant; see Section~\ref{Weyl.sec}.}
$$
\Phi(z)\coloneqq\frac{1+t}{2}\Re(z)^2+\frac{1-t}{2}\Im (z)^2
$$
is convex with gradient given by $z+t\bar z = T_t(z)$. Note that even though $T_1$ is not injective and Theorem \ref{theo:pushforward} does technically not apply, it is true that $\nu_{1}={T_1}_\#\nu_0$ (which is the semicircle distribution).
For any $0\le s,t\le 1$, it holds
\begin{align*}
W_2(\nu_s,\nu_t)\le \Big(\int |T_s(w)-T_t(w)|^2d\nu_0(w)\Big)^{1/2}=|t-s|W_2(\nu_0,\nu_1).
\end{align*}
Therefore, it follows from the triangle inequality of the Wasserstein metric that $(\nu_{t})_{t\in[0,1]}$ is a constant speed geodesic, see \cite[Theorem 5.27]{Santa}.

Conversely, assume there would be a convex function $\Phi$ such that its gradient\footnote{here, we identify $\nabla=2\partial_{\bar z}=2\bar \partial_z$.} is given by
\[\nabla \Phi(z)=T_t(z)=z+t\frac{\bar z\alpha_0(z)}{|z|^2}=x(1+\beta(x^2+y^2))+iy(1-\beta(x^2+y^2)),\]
where we denoted $z=x+iy$ and abbreviated the radial function $\beta(x^2+y^2)=t\alpha_0(x+iy)/(x^2+y^2)$.
A convex function is $\mathcal C^2$ almost everywhere, hence by differentiating in both real directions again, we have
\[2xy\beta'(x^2+y^2)=\partial_y\partial_x\Phi(z)=\partial_x\partial_y\Phi(z)=-2xy\beta'(x^2+y^2).\]
But this implies $\beta'=0$, or with other words $\alpha_0(z)=c|z|^2$ for some $c>0$. Due to $g(0)=0$, this is the case of $\nu_0$ being the uniform measure on $\bD$.

\subsection{Evenly spaced roots: Proof of Proposition~\ref{prop:evenly_spaced_roots}}
Let first  $t>0$. Observe that, by~\eqref{eq:heat_flow_z_n_hermite_poly},
$$
\eee^{-\frac t{2n} \partial_z^2} (z^n - r^n) = \Big(\frac t n\Big)^{n/2} \He_n\left(\frac {z\sqrt n}{\sqrt t}\right) - r^n.
$$
By Lemma~\ref{lem:hermite_asymptotics}, the logarithmic rate of growth of this polynomial is given by
\begin{equation}\label{eq:evenly_spaced_log_potential}
V(z) \coloneqq \lim_{n\to\infty} \frac 1n \log \left|\Big(\frac t n\Big)^{n/2} \He_n\left(\frac {z\sqrt n}{\sqrt t}\right) - r^n\right|
=
\max\left\{\frac 12 \log |t| + \Psi\left(\frac z {\sqrt t}\right),\log r\right\},
\end{equation}
for all $z\in \C \backslash( [-2\sqrt t,2\sqrt t]\cup L_{r,t})$, where $L_{r,t}$ is the set of all $z\in \C$ for which the terms under the maximum are equal; see~\eqref{eq:def_level_set_L_r_t_evenly_spaced}.  To prove that the empirical distribution of zeros of $\eee^{-\frac t{2n} \partial_z^2} (z^n - r^n)$ converges to $\frac 1{2\pi} \Delta V$ it suffices to check that~\eqref{eq:evenly_spaced_log_potential} holds in $L^1_{\mathrm{loc}}$ and all roots of $\eee^{-\frac t{2n} \partial_z^2} (z^n - r^n)$ are located in some disk $\{|z|\leq B\}$ not depending on $n$. All zeros of the monic polynomial  $(\frac t n)^{n/2} \He_n(z\sqrt n/\sqrt t)$ are located in $[-3\sqrt t, + 3\sqrt t]$. Denoting them by $y_{1:n}, \ldots, y_{n:n}$ we have
$$
\frac 1n \log \left|\Big(\frac t n\Big)^{n/2} \He_n\left(\frac {z\sqrt n}{\sqrt t}\right)\right|
=
\frac 1n \sum_{j=1}^n \log |z - y_{j:n}|
\geq
\frac 1n \sum_{j=1}^n \log |z/2| = \log |z/2| \geq  2 \log r
$$
provided $|z|\geq  B \coloneqq \max\{6\sqrt t, 2r^2\}$. It follows that all zeros of $\eee^{-\frac t{2n} \partial_z^2} (z^n - r^n)$, which we denote by $u_{1:n}, \ldots, u_{n:n}$, are located in $\{|z|\leq B\}$, for all $n\in \N$. To prove the $L^1_{\mathrm{loc}}$-convergence in~\eqref{eq:evenly_spaced_log_potential} it  suffices to check the local boundedness in $L^2$, meaning that
$$
\sup_{n\in \N} \frac 1 {n^2} \int_K   \log^2 \left|\prod_{j=1}^n (z - u_{j:n})\right|\, \dint z \dint \bar z <\infty,
$$
for every disk $K = \{|z|< C_0\}$. By the inequality between arithmetic and quadratic means we have
\begin{align*}
\frac 1 {n^2} \int_K   \log^2 \left|\prod_{j=1}^n (z - u_{j:n})\right|\, \dint z \dint \bar z
&\leq
\frac 1 {n} \int_K \sum_{j=1}^n \log^2 |z - u_{j:n}|\, \dint z \dint \bar z\\
&\leq
\frac 2 {n} \sum_{j=1}^n \int_K  (\log^2_+ |z - u_{j:n}| + \log^2_- |z - u_{j:n}|)\, \dint z \dint \bar z.
\end{align*}
The right-hand side is bounded uniformly in $n$ since $\log^2_+ |z - u_{j:n}|$ is uniformly bounded (recall that $|z|\leq C_0$ and $|u_{j,n}|\leq B$) and since $\int_K\log^2_- |z - u_{j:n}|\, \dint z \dint \bar z \leq \int_{\{|w|\leq 1\}}\log^2_- |w| \, \dint w \dint \bar w <\infty$.

Let us finally prove parts (a), (b), (c) of Proposition~\ref{prop:evenly_spaced_roots}.
Recall that $\Psi(z)$ is harmonic outside $[-2,2]$. It follows that the function  $V$ is harmonic outside $L_{r,t}\cup [-2\sqrt t, + 2\sqrt t]$, hence $\Delta V$ is concentrated on this set. For $z\in [-2\sqrt t, + 2\sqrt t]$, we have $\Psi(z/\sqrt t) = z^2/(4t) - \frac 12$; see~\eqref{eq:Psi_for_x_real_smaller_2}. \emph{Proof of (a)}: If $0< t\leq r^2/\eee$, then $\Psi(z/\sqrt t) < \log (r/\sqrt t)$ for all $z\in (-2\sqrt t, + 2\sqrt t)$, hence the function $V(z)$ stays constant in a small open set containing this interval and the support of $\Delta V$ is contained in $L_{r,t}$. \emph{Proof of (b)}:  If $t>r^2/\eee$, then $\Psi(z/\sqrt t) < \log (r/\sqrt t)$ for all $z\in (-\sqrt{2 t \log (r^2 \eee /t)}, +\sqrt{2 t \log (r^2 \eee /t)})$ (which is a subset of the interval $(-2\sqrt t, + 2\sqrt t)$), hence the function $V(z)$ stays constant in a small open set containing the former interval and (b) follows. \emph{Proof of (c)}: Let us check that $\inf_{z\in \C}\Psi(z) = \Psi(0) = -1/2$. Since $\Psi$ is continuous and $\lim_{|z| \to\infty} {\Psi(z)}/{\log|z|}=1$, the infimum is attained somewhere in $\C$. By \eqref{eq:Stieltjes_semicircle}, for all $z\in \C\backslash [-2,2]$ we have,
$$
\partial_z \Psi(z)=\frac{1}{4}(z-\sqrt{z^2-4})\neq 0.
$$
This implies that  $\Psi$ has no critical points outside of $[-2,2]$ since at any such point we have $\frac{\partial}{\partial x} \Psi(z)  = \frac{\partial}{\partial y} \Psi(z) = 0$ implying that $\partial_z \Psi(z) = 0$. So, the infimum is attained on the interval $[-2,2]$ and must be $\Psi(0) = -1/2$ by~\eqref{eq:Psi_for_x_real_smaller_2}.   It follows that $V(z) = \log \sqrt t + \Psi(z/\sqrt t)$ for all $t>r^2\eee$ and hence $\mu_{r,t}$ is the Wigner law on $[-2\sqrt t, 2\sqrt t]$.

Note that the above analysis applies if $r\neq 0$ is complex, provided we replace $r$ by $|r|$ throughout.  To prove part (d), let $t= |t|\eee^{\ii \theta/2}\in \C \backslash\{0\}$.  Then,
$$
\eee^{-\frac t{2n} \partial_z^2} (z^n - r^n) = \left(\frac {\sqrt{|t|} \, \eee^{\ii \theta/2}} {\sqrt n}\right)^{n} \He_n\left(\frac {z\sqrt n}{\sqrt {|t|}\, \eee^{\ii \theta/2}}\right) - r^n.
$$
We can absorb the term $\eee^{\ii \theta/2}$  into $z$ and $r$ and reduce the problem to the one already solved.
\hfill $\Box$
%\end{proof}

\subsection{Proofs of ODE's and PDE's}\label{subsec:proof_PDE}
For completeness, let us first provide the well known arguments that lead to the finite $n$ equations.

\begin{proof}[Proof of \eqref{eq:ODE_for_poly}, \eqref{eq:PDE_finite_n} and \eqref{eq:PDE_finite_n_stieltjes}]
The derivatives of $P_n(z;t)=\xi_n a_{n,n}\prod_{j=1}^n (z-z_j(t))$ are given by
\begin{align*}
\partial_z P_n(z;t)&=\xi_n a_{n,n}\sum_{k=1,\ldots,n}\prod_{\substack{j=1,\dots,n\\j\neq k}}(z-z_j(t))\\
\partial_z^2 P_n(z;t)&=\xi_n a_{n,n}\sum_{\substack{k,l=1,\dots,n\\k\neq l}}\prod_{\substack{j=1,\dots,n\\j\neq k,l}}(z-z_j(t))
\end{align*}
and from the backward heat equation we know that $\partial_t P_n(z;t)=-\frac{1}{2n}\partial_z^2P_n(z;t)$. In particular, for a zero $z_j(t)$ it follows from the chain rule
\begin{align*}
0=\partial_t \big(P_n(z_j(t);t) \big)=\big(\partial_t P_n\big)(z_j(t);t)+\partial_z P_n(z_j(t);t)\partial _t z_j(t),
\end{align*}
which is equivalent to \eqref{eq:ODE_for_poly}, since
\begin{align}\label{eq:ODE_for_poly_proof}
\partial_t z_j(t)=\frac 1 {2n}\frac{\partial_z^2P_n(z_j(t);t)}{\partial_z P_n(z_j(t);t)}=\frac{1}{n}\sum_{\substack{k=1,\dots,n\\k\neq j}}\frac{1}{z_j(t)-z_k(t)}.
\end{align}
For the logarithmic potential $U_t^{(n)}$ of the empirical distribution $\llbracket P_n(\cdot;t) \rrbracket $ it follows for all $z\neq z_1(t),\dots,z_n(t)$ that
\begin{align*}
\partial_t U_t^{(n)}(z)=\partial_t \frac{1}{n} \log | P_n(z;t)|&=\frac{1}{n} \frac{\partial_t P_n(z;t)}{2P_n(z;t)}+0 \\
&=-\frac 1 {2n^2}\frac{\partial_z^2 P_n(z;t)}{2P_n(z;t)}\\
&=-\frac 1 {4n^2}\sum_{\substack{j,k=1,\dots,n\\j\neq k}} \frac{1}{(z-z_j(t))(z-z_k(t))} ,
\end{align*}
where we used the complex chain rule with zero antiholomorphic part. On the other hand
\begin{align*}
(\partial_z U_t^{(n)}(z))^2=\left( \frac{1}{2n}\frac{\partial_z P_n(z;t)}{P_n(z;t)}+0\right)^2=\frac 1 {4n^2}\sum_{j,k=1,\dots,n} \frac{1}{(z-z_j(t))(z-z_k(t))} ,
\end{align*}
hence what remains in the difference is the diagonal term
\begin{align*}
\partial_t U_t^{(n)}(z)+(\partial_z U_t^{(n)}(z))^2=\frac 1 {4n^2}\sum_{j=1,\dots,n} \frac{1}{(z-z_j(t))^2}.
\end{align*}
Finally, \eqref{eq:PDE_finite_n_stieltjes} follows immediately from \eqref{eq:PDE_finite_n} by applying $2\partial_z$.
\end{proof}

We now turn to the derivation of the PDE's for $U_t$ and $m_t$.
\begin{proof}[Proof of Theorem \ref{theo:PDE}]
On the one hand, squaring~\eqref{eq:partial_z_U_t_z} gives
\begin{align*}
-\left(\partial_z U_{t}(z)\right)^2
&=
-\frac{1}{16t^2}\left(z^2-2z\sqrt{z^2-4t\alpha_t(z)}+z^2-4t\alpha_t(z) \right)\\
&=
\frac{\alpha_t(z)}{4t}-\frac{z}{8t^2}\left(z-\sqrt{z^2-4t\alpha_t(z)}\right).
\end{align*}
On the other hand, to compute the $\partial_t (U_{t}(z))$, we use the chain rule:
\begin{align*}
\partial_t (U_{t}(z))
=
\partial_t \left(f_t(\alpha_t(z),z)\right)
=
(\partial_\alpha f_t)(\alpha_t(z),z)\partial_t \alpha_t(z)+(\partial_t f_t)(\alpha_t(z),z)
=
(\partial_t f_t)(\alpha_t(z),z),
\end{align*}
where the last equality holds since $\partial_t\alpha_t(z)=0$ (in cases (D0) and (D1)) or $(\partial_\alpha f_t)(\alpha_t(z),z)=0$ (in case (D2)).
Using the expression for $f_t(\alpha, z)$ given in Theorem~\ref{theo:main_general_g} and the formula $\partial_t\log\sqrt{|t|}=1/(4t)$ (recall that $\partial_t$  is the Wirtinger derivative in the \emph{complex} variable $t$, explaining the strange constant $1/4$) we obtain
\begin{align*}
\partial_t f_t(\alpha,z)
&=
\alpha \left(\frac 1 {4t} - \left(\partial _z\Psi\right)\left(\frac{z}{\sqrt{t\alpha}}\right) \frac{z}{2t^{3/2}\sqrt{\alpha}}\right)\\
&=
\frac \alpha {4t}-\frac{\alpha z}{8t\sqrt{t\alpha}}\left(\frac{z}{\sqrt{t\alpha}}-\sqrt{\frac{z^2}{t\alpha}-4}\right)\\
&=
\frac{\alpha}{4t}-\frac{z}{8t^2}\left(z-\sqrt{z^2-4t\alpha}\right).
\end{align*}
It follows that
$$
\partial_t U_{t}(z)
=
(\partial_t f_t)(\alpha_t(z),z)
=
\frac{\alpha_t(z)}{4t}-\frac{z}{8t^2}\left(z-\sqrt{z^2-4t\alpha_t(z)}\right)
=
-(\partial_z U_{t}(z))^2,
$$
thus proving part~(a) of the theorem.

Part~(b) again follows from the formula  $m_t(z) = 2 \partial_z U_t(z)$ and the PDE for $U_t$, more precisely
$$
\partial_t m_t(z)
=
2 \partial_t \partial_z U_t(z)
=
2 \partial_z \partial_t U_t(z)
=
- 2 \partial_z [(\partial_z U_{t}(z))^2]
=
- \frac 12  \partial_z [m_{t}^2(z)]
=
- m_t(z) \partial_z m_t(z).
$$
Since the function $U_t(z)$ takes real values,  $\partial_{\overline{t}} U_t(z)$ is the complex conjugate of $\partial_t U_t(z)$. It follows that
$$
\partial_{\overline{t}}  m_t(z)
=
2 \partial_{\overline{t}} \partial_z U_t(z)
=
2 \partial_z \partial_{\overline{t}} U_t(z)
=
2 \partial_z  \overline{\partial_{t} U_t(z)}
=
- \frac 12  \partial_z [\overline{m_{t}^2(z)}]
=
- \overline{m_t(z)} \partial_z \overline{m_t(z)}.
$$

Part (c) follows by applying the operator $(2/\pi)\partial_z \partial_{\bar z}$ to both sides of the PDE for $U_t$, leaving the $z$-derivative unevaluated on the right-hand side of the PDE.

In order to prove of part (d), recall that in the proof of Theorem \ref{theo:pushforward}, we verified that the condition $z/\sqrt{t \alpha_t(z)}\not\in[-2,2]$ is fulfilled, see \eqref{eq:z_not_in_I}. Moreover, it follows from Theorem \ref{theo:pushforward} that $\mathcal D_0=\{(T_t(w),t)\in\C: |w|<\eee^{-g'(0)}, 0<|t|<t_{\mathrm{sing}}\}$, $\mathcal D_1=\{(T_t(w),t)\in\C: |w|>\eee^{-g'(1)}, 0<|t|<t_{\mathrm{sing}}\}$ and $\mathcal D_2=\{(z,t)\in\C: z\in\mathcal R_t, 0<|t|<t_{\mathrm{sing}}\}$ satisfy (D0), (D1) and (D2), respectively, in Definition \ref{def:regular_points} of regular points. Hence, we excluded only the 1-dimensional curves $\partial \mathcal R_t=\partial\mathcal E_{\eee^{-g'(0)},t}\cup\partial\mathcal E_{\eee^{-g'(1)},t}$. On the other hand, Theorem \ref{theo:pushforward}~(e) states that $U_t$ is differentiable for all $z\in\C$ (actually, we show in \eqref{eq:partial_z_U_t_z} that the derivatives match on both sides of the curves). Therefore, the domains $\mathcal D_0,\mathcal D_1$ and $\mathcal D_2$ can be glued together and the PDE can be continuously extended to $\mathcal D\coloneqq\C\times B_{t_{\mathrm{sing}}}(0)$. More precisely, the right hand side $-\left(\partial_z U_t(z)\right)^2$ of \eqref{eq:PDE_global} exists as a continuous function on $\mathcal D$ and the partial derivative $\partial_t U_t(z)$ on $ \mathcal D_0\cup\mathcal D_1\cup \mathcal D_2$ can be continuously extended to a weak derivative $\partial_t U_t(z)$ on whole $\mathcal D$, but a continuous weak derivative $\partial_t U_t (z)$ of a continuous function is continuously differentiable, see for instance \cite[\S 9.1, Remark 2]{brezis_book}.
\end{proof}

%%%%%%%%%%%%%%%%%%%%%%%%%%%%%%%%%%%%%%%%%%%%%%
%% Single Appendix:                         %%
%%%%%%%%%%%%%%%%%%%%%%%%%%%%%%%%%%%%%%%%%%%%%%
%\begin{appendix}
%\section*{???}%% if no title is needed, leave empty \section*{}.
%\end{appendix}
%%%%%%%%%%%%%%%%%%%%%%%%%%%%%%%%%%%%%%%%%%%%%%
%% Multiple Appendixes:                     %%
%%%%%%%%%%%%%%%%%%%%%%%%%%%%%%%%%%%%%%%%%%%%%%
%\begin{appendix}
%\section{???}
%
%\section{???}
%
%\end{appendix}

%%%%%%%%%%%%%%%%%%%%%%%%%%%%%%%%%%%%%%%%%%%%%%
%% Support information, if any,             %%
%% should be provided in the                %%
%% Acknowledgements section.                %%
%%%%%%%%%%%%%%%%%%%%%%%%%%%%%%%%%%%%%%%%%%%%%%
\subsection*{Acknowledgments}
We thank Martin Huesmann for suggesting the example in Section~\ref{martin.sec}.

%%%%%%%%%%%%%%%%%%%%%%%%%%%%%%%%%%%%%%%%%%%%%%
%% Funding information, if any,             %%
%% should be provided in the                %%
%% funding section.                         %%
%%%%%%%%%%%%%%%%%%%%%%%%%%%%%%%%%%%%%%%%%%%%%%
\subsection*{Funding}
BH is supported in part by a grant from the Simons Foundation. CH is supported in part by the MoST grant 111-2115-M-001-011-MY3.
JJ and ZK are funded by the Deutsche Forschungsgemeinschaft (DFG, German Research Foundation) under Germany's Excellence Strategy EXC 2044 - 390685587, Mathematics M\"unster: \emph{Dynamics-Geometry-Structure} and have been supported by the DFG priority program SPP 2265 \emph{Random Geometric Systems}.

\bibliography{BIB_weyl_heat_flow_bib}

\begin{thebibliography}{113}
\providecommand{\natexlab}[1]{#1}
\providecommand{\url}[1]{\texttt{#1}}
\expandafter\ifx\csname urlstyle\endcsname\relax
  \providecommand{\doi}[1]{doi: #1}\else
  \providecommand{\doi}{doi: \begingroup \urlstyle{rm}\Url}\fi

\bibitem[Akemann and Burda(2012)]{ake}
Gernot Akemann and Zdzislaw Burda.
\newblock Universal microscopic correlation functions for products of
  independent {G}inibre matrices.
\newblock \emph{Journal of Physics A: Mathematical and Theoretical},
  45\penalty0 (46):\penalty0 465201, 2012.

\bibitem[Akemann et~al.(2018)Akemann, Cikovic, and
  Venker]{akemann2018universality}
Gernot Akemann, Milan Cikovic, and Martin Venker.
\newblock Universality at weak and strong non-{H}ermiticity beyond the elliptic
  {G}inibre ensemble.
\newblock \emph{Communications in Mathematical Physics}, 362:\penalty0
  1111--1141, 2018.

\bibitem[Ambrosio and Gigli(2013)]{AmbrosioGigli}
Luigi Ambrosio and Nicola Gigli.
\newblock A user’s guide to optimal transport.
\newblock In \emph{Modelling and optimisation of flows on networks}, pages
  1--155. Springer, 2013.

\bibitem[Ameur and Byun(2023)]{ameur2023almost}
Yacin Ameur and Sung-Soo Byun.
\newblock Almost-{H}ermitian random matrices and bandlimited point processes.
\newblock \emph{Analysis and Mathematical Physics}, 13\penalty0 (3):\penalty0
  52, 2023.

\bibitem[Anderson et~al.(2010)Anderson, Guionnet, and Zeitouni]{AGZ}
Greg~W. Anderson, Alice Guionnet, and Ofer Zeitouni.
\newblock \emph{An introduction to random matrices}, volume 118.
\newblock Cambridge university press, 2010.

\bibitem[Angst et~al.(2023)Angst, Malicet, and Poly]{Angst}
J{\"u}rgen Angst, Dominique Malicet, and Guillaume Poly.
\newblock Almost sure behavior of the critical points of random polynomials.
\newblock \emph{arXiv preprint arXiv:2301.06973}, 2023.

\bibitem[Arizmendi et~al.(2023)Arizmendi, Garza-Vargas, and
  Perales]{arizmendi_garza_vargas_perales}
Octavio Arizmendi, Jorge Garza-Vargas, and Daniel Perales.
\newblock Finite free cumulants: multiplicative convolutions, genus expansion
  and infinitesimal distributions.
\newblock \emph{Trans. Amer. Math. Soc.}, 376\penalty0 (6):\penalty0
  4383--4420, 2023.
\newblock URL \url{https://doi.org/10.1090/tran/8884}.

\bibitem[Bai(1997)]{bai}
Zhi~Dong Bai.
\newblock Circular law.
\newblock \emph{The Annals of Probability}, pages 494--529, 1997.

\bibitem[Belinschi et~al.(2024)Belinschi, Yin, and Zhong]{BYZ}
Serban Belinschi, Zhi Yin, and Ping Zhong.
\newblock The {B}rown measure of a sum of two free random variables, one of
  which is triangular elliptic.
\newblock \emph{Adv. Math.}, 441:\penalty0 Paper No. 109562, 62, 2024.
\newblock ISSN 0001-8708,1090-2082.
\newblock \doi{10.1016/j.aim.2024.109562}.
\newblock URL \url{https://doi.org/10.1016/j.aim.2024.109562}.

\bibitem[Berti et~al.(2006)Berti, Pratelli, and
  Rigo]{berti_pratelli_rigo_almost_sure_weak_conv}
Patrizia Berti, Luca Pratelli, and Pietro Rigo.
\newblock Almost sure weak convergence of random probability measures.
\newblock \emph{Stochastics}, 78\penalty0 (2):\penalty0 91--97, 2006.
\newblock URL \url{https://doi.org/10.1080/17442500600745359}.

\bibitem[Biane(1997)]{Biane}
Philippe Biane.
\newblock On the free convolution with a semi-circular distribution.
\newblock \emph{Indiana University Mathematics Journal}, pages 705--718, 1997.

\bibitem[Biane and Lehner(2001)]{BianeLehner}
Philippe Biane and Franz Lehner.
\newblock Computation of some examples of {B}rown's spectral measure in free
  probability.
\newblock \emph{Colloq. Math.}, 90\penalty0 (2):\penalty0 181--211, 2001.
\newblock ISSN 0010-1354,1730-6302.
\newblock \doi{10.4064/cm90-2-3}.
\newblock URL \url{https://doi.org/10.4064/cm90-2-3}.

\bibitem[B{\o}gvad et~al.(2023)B{\o}gvad, H{\"{a}}gg, and Shapiro]{bogvad_etal}
R.~B{\o}gvad, C.~H{\"{a}}gg, and B.~Shapiro.
\newblock Rodrigues' descendants of a polynomial and {B}outroux curves.
\newblock To appear in Constructive Approximation. Preprint at
  http://arxiv.org/abs/2107.05710, 2023.
\newblock URL \url{https://doi.org/10.1007/s00365-023-09657-x}.

\bibitem[Bordenave and Chafa{\"\i}(2012)]{BC12}
Charles Bordenave and Djalil Chafa{\"\i}.
\newblock Around the circular law.
\newblock \emph{Probability surveys}, 9, 2012.

\bibitem[Brezis(2011)]{brezis_book}
Haim Brezis.
\newblock \emph{Functional analysis, Sobolev spaces and partial differential
  equations}.
\newblock Springer, 2011.

\bibitem[Burda et~al.(2010)Burda, Janik, and Waclaw]{burda}
Zdzis{\l}aw Burda, Romuald~A Janik, and Bartek Waclaw.
\newblock Spectrum of the product of independent random {G}aussian matrices.
\newblock \emph{Physical Review E}, 81\penalty0 (4):\penalty0 041132, 2010.

\bibitem[Burgers(1974)]{burgers_book}
J.~M. Burgers.
\newblock \emph{The nonlinear diffusion equation. {A}symptotic solutions and
  statistical problems}.
\newblock Dordrecht - {Boston}: {D}. {Reidel} {Publishing} {Company}., 1974.

\bibitem[Byun and Forrester(2022)]{byun2022progress}
Sung-Soo Byun and Peter~J Forrester.
\newblock {Progress on the study of the Ginibre ensembles I: GinUE}.
\newblock \emph{arXiv preprint arXiv:2211.16223}, 2022.

\bibitem[Byun et~al.(2022)Byun, Lee, and Reddy]{Byun}
Sung-Soo Byun, Jaehun Lee, and Tulasi Reddy.
\newblock Zeros of random polynomials and their higher derivatives.
\newblock \emph{Transactions of the American Mathematical Society},
  375\penalty0 (09):\penalty0 6311--6335, 2022.

\bibitem[Byun et~al.(2023)Byun, Kang, Lee, and Lee]{byun2023real}
Sung-Soo Byun, Nam-Gyu Kang, Ji~Oon Lee, and Jinyeop Lee.
\newblock Real eigenvalues of elliptic random matrices.
\newblock \emph{International Mathematics Research Notices}, 2023\penalty0
  (3):\penalty0 2243--2280, 2023.

\bibitem[Campbell et~al.(2023)Campbell, O'Rourke, and Renfrew]{COR23}
Andrew Campbell, Sean O'Rourke, and David Renfrew.
\newblock The fractional free convolution of {R}-diagonal operators and random
  polynomials under repeated differentiation.
\newblock \emph{arXiv preprint arXiv:2307.11935}, 2023.

\bibitem[Craig(2018)]{craig_book}
Walter Craig.
\newblock \emph{A course on partial differential equations}, volume 197 of
  \emph{Graduate Studies in Mathematics}.
\newblock American Mathematical Society, Providence, RI, 2018.
\newblock URL \url{https://doi.org/10.1090/gsm/197}.

\bibitem[Csordas et~al.(1994)Csordas, Smith, and Varga]{csordas_smith_varga}
G.~Csordas, W.~Smith, and R.~S. Varga.
\newblock Lehmer pairs of zeros, the de {B}ruijn-{N}ewman constant {$\Lambda$},
  and the {R}iemann hypothesis.
\newblock \emph{Constr. Approx.}, 10\penalty0 (1):\penalty0 107--129, 1994.
\newblock URL \url{https://doi.org/10.1007/BF01205170}.

\bibitem[del Monaco and Schlei{\ss}inger(2016)]{del_monaco}
Andrea del Monaco and Sebastian Schlei{\ss}inger.
\newblock Multiple {SLE} and the complex {B}urgers equation.
\newblock \emph{Math. Nachr.}, 289\penalty0 (16):\penalty0 2007--2018, 2016.
\newblock URL \url{https://doi.org/10.1002/mana.201500230}.

\bibitem[Diaconis and Shahshahani(1994)]{diaconis1994eigenvalues}
Persi Diaconis and Mehrdad Shahshahani.
\newblock On the eigenvalues of random matrices.
\newblock \emph{Journal of Applied Probability}, 31\penalty0 (A):\penalty0
  49--62, 1994.

\bibitem[Driver et~al.(2022)Driver, Hall, and Kemp]{DHK22}
Bruce~K Driver, Brian Hall, and Todd Kemp.
\newblock The {B}rown measure of the free multiplicative {B}rownian motion.
\newblock \emph{Probability Theory and Related Fields}, 184\penalty0
  (1-2):\penalty0 209--273, 2022.

\bibitem[Dyson(1962)]{dyson1962statistical}
Freeman~J Dyson.
\newblock Statistical theory of the energy levels of complex systems. i.
\newblock \emph{Journal of Mathematical Physics}, 3\penalty0 (1):\penalty0
  140--156, 1962.

\bibitem[Erbar(2010)]{erbar}
Matthias Erbar.
\newblock The heat equation on manifolds as a gradient flow in the
  {W}asserstein space.
\newblock \emph{Ann. Inst. Henri Poincar\'{e} Probab. Stat.}, 46\penalty0
  (1):\penalty0 1--23, 2010.
\newblock ISSN 0246-0203,1778-7017.
\newblock \doi{10.1214/08-AIHP306}.
\newblock URL \url{https://doi.org/10.1214/08-AIHP306}.

\bibitem[Esseen(1968)]{Esseen68}
Carl-Gustav Esseen.
\newblock On the concentration function of a sum of independent random
  variables.
\newblock \emph{Zeitschrift f{\"u}r Wahrscheinlichkeitstheorie und Verwandte
  Gebiete}, 9\penalty0 (4):\penalty0 290--308, 1968.

\bibitem[Evans(2010)]{evans_book_PDE}
Lawrence~C. Evans.
\newblock \emph{Partial differential equations}, volume~19 of \emph{Graduate
  Studies in Mathematics}.
\newblock American Mathematical Society, Providence, RI, second edition, 2010.
\newblock URL \url{https://doi.org/10.1090/gsm/019}.

\bibitem[Evans(2014)]{Evans14}
Lawrence~C Evans.
\newblock Envelopes and nonconvex {H}amilton--{J}acobi equations.
\newblock \emph{Calculus of Variations and Partial Differential Equations},
  50\penalty0 (1-2):\penalty0 257--282, 2014.

\bibitem[Feng and Yao(2019)]{FengYao19}
Renjie Feng and Dong Yao.
\newblock Zeros of repeated derivatives of random polynomials.
\newblock \emph{Analysis \& PDE}, 12\penalty0 (6):\penalty0 1489--1512, 2019.

\bibitem[Forrester(2010)]{forrester2010log}
Peter~J Forrester.
\newblock \emph{Log-gases and random matrices (LMS-34)}.
\newblock Princeton University Press, 2010.

\bibitem[Forrester and Ipsen(2019)]{forrester2019generalisation}
Peter~J Forrester and Jesper~R Ipsen.
\newblock A generalisation of the relation between zeros of the complex {K}ac
  polynomial and eigenvalues of truncated unitary matrices.
\newblock \emph{Probability Theory and Related Fields}, 175:\penalty0 833--847,
  2019.

\bibitem[Fyodorov et~al.(1997)Fyodorov, Khoruzhenko, and
  Sommers]{fyodorov1997almost}
Yan~V Fyodorov, Boris~A Khoruzhenko, and Hans-J{\"u}rgen Sommers.
\newblock Almost {H}ermitian random matrices: crossover from {W}igner-{D}yson
  to {G}inibre eigenvalue statistics.
\newblock \emph{Physical review letters}, 79\penalty0 (4):\penalty0 557, 1997.

\bibitem[Gawronski(1987)]{gawronski}
W.~Gawronski.
\newblock On the asymptotic distribution of the zeros of {H}ermite, {L}aguerre,
  and {J}onqui\`ere polynomials.
\newblock \emph{J. Approx. Theory}, 50\penalty0 (3):\penalty0 214--231, 1987.
\newblock URL \url{https://doi.org/10.1016/0021-9045(87)90020-7}.

\bibitem[Gawronski(1993)]{gawronski_strong_asymptotics}
W.~Gawronski.
\newblock Strong asymptotics and the asymptotic zero distributions of
  {L}aguerre polynomials {$L_n^{(an+\alpha)}$} and {H}ermite polynomials
  {$H_n^{(an+\alpha)}$}.
\newblock \emph{Analysis}, 13\penalty0 (1-2):\penalty0 29--67, 1993.
\newblock URL \url{https://doi.org/10.1524/anly.1993.13.12.29}.

\bibitem[Gawronski and Van~Assche(2003)]{gawronski_van_assche}
Wolfgang Gawronski and Walter Van~Assche.
\newblock Strong asymptotics for relativistic {H}ermite polynomials.
\newblock \emph{Rocky Mountain J. Math.}, 33\penalty0 (2):\penalty0 489--524,
  2003.
\newblock ISSN 0035-7596.
\newblock URL \url{https://doi.org/10.1216/rmjm/1181069964}.

\bibitem[Ginibre(1965)]{ginibre}
J.~Ginibre.
\newblock Statistical ensembles of complex, quaternion, and real matrices.
\newblock \emph{J. Mathematical Phys.}, 6:\penalty0 440--449, 1965.
\newblock URL \url{https://doi.org/10.1063/1.1704292}.

\bibitem[Girko(1984)]{girko_circular_law}
V.~L. Girko.
\newblock The circular law.
\newblock \emph{Teor. Veroyatnost. i Primenen.}, 29\penalty0 (4):\penalty0
  669--679, 1984.

\bibitem[Girko(1985)]{girko_elliptic}
V.~L. Girko.
\newblock The elliptic law.
\newblock \emph{Teor. Veroyatnost. i Primenen.}, 30\penalty0 (4):\penalty0
  640--651, 1985.

\bibitem[G{\"o}tze and Jalowy(2021)]{GJ21}
Friedrich G{\"o}tze and Jonas Jalowy.
\newblock Rate of convergence to the circular law via smoothing inequalities
  for log-potentials.
\newblock \emph{Random Matrices: Theory and Applications}, 10\penalty0
  (03):\penalty0 2150026, 2021.

\bibitem[G{\"o}tze and Tikhomirov(2010{\natexlab{a}})]{GT10}
Friedrich G{\"o}tze and Alexander Tikhomirov.
\newblock On the asymptotic spectrum of products of independent random
  matrices.
\newblock \emph{arXiv preprint arXiv:1012.2710}, 2010{\natexlab{a}}.

\bibitem[G{\"o}tze and Tikhomirov(2010{\natexlab{b}})]{GT10Circular}
Friedrich G{\"o}tze and Alexander Tikhomirov.
\newblock The circular law for random matrices.
\newblock \emph{The Annals of Probability}, 38\penalty0 (4):\penalty0
  1444--1491, 2010{\natexlab{b}}.

\bibitem[G{\"o}tze et~al.(2015)G{\"o}tze, K{\"o}sters, and Tikhomirov]{GKT}
Friedrich G{\"o}tze, Holger K{\"o}sters, and Alexander Tikhomirov.
\newblock Asymptotic spectra of matrix-valued functions of independent random
  matrices and free probability.
\newblock \emph{Random Matrices: Theory and Applications}, 4\penalty0
  (02):\penalty0 1550005, 2015.

\bibitem[Haagerup and Larsen(2000)]{haageruplarsen}
Uffe Haagerup and Flemming Larsen.
\newblock Brown's spectral distribution measure for {$R$}-diagonal elements in
  finite von {N}eumann algebras.
\newblock \emph{J. Funct. Anal.}, 176\penalty0 (2):\penalty0 331--367, 2000.
\newblock ISSN 0022-1236,1096-0783.
\newblock \doi{10.1006/jfan.2000.3610}.
\newblock URL \url{https://doi.org/10.1006/jfan.2000.3610}.

\bibitem[Hall and Ho(2022)]{hallho}
Brian~C. Hall and Ching-Wei Ho.
\newblock The heat flow conjecture for random matrices.
\newblock \emph{arXiv preprint arXiv:2202.09660}, 2022.

\bibitem[Hall and Ho(2023)]{HHfamily}
Brian~C. Hall and Ching-Wei Ho.
\newblock The {B}rown measure of a family of free multiplicative {B}rownian
  motions.
\newblock \emph{Probab. Theory Related Fields}, 186\penalty0 (3-4):\penalty0
  1081--1166, 2023.
\newblock ISSN 0178-8051,1432-2064.
\newblock \doi{10.1007/s00440-022-01166-5}.
\newblock URL \url{https://doi.org/10.1007/s00440-022-01166-5}.

\bibitem[Hall et~al.(2023)Hall, Ho, Jalowy, and Kabluchko]{diff-paper}
Brian~C Hall, Ching-Wei Ho, Jonas Jalowy, and Zakhar Kabluchko.
\newblock Roots of polynomials under repeated differentiation and repeated
  applications of fractional differential operators.
\newblock \emph{arXiv preprint arXiv:2312.14883}, 2023.

\bibitem[Hall et~al.(to appear)Hall, Ho, Jalowy, and Kabluchko]{GAF-paper}
Brian~C. Hall, Ching-Wei Ho, Jonas Jalowy, and Zakhar Kabluchko.
\newblock The heat flow, {GAF}, and ${SL}(2;\mathbb{R})$.
\newblock \emph{Indiana University Mathematics Journal}, to appear.

\bibitem[Hiai and Petz(2000)]{hiai_petz_book}
Fumio Hiai and D\'{e}nes Petz.
\newblock \emph{The semicircle law, free random variables and entropy},
  volume~77 of \emph{Mathematical Surveys and Monographs}.
\newblock American Mathematical Society, Providence, RI, 2000.
\newblock URL \url{https://doi.org/10.1090/surv/077}.

\bibitem[H{\"o}rmander(2015)]{hormander1}
Lars H{\"o}rmander.
\newblock \emph{The analysis of linear partial differential operators I:
  Distribution theory and Fourier analysis}.
\newblock Springer, 2015.

\bibitem[Hoskins and Kabluchko(2021)]{HK21}
Jeremy Hoskins and Zakhar Kabluchko.
\newblock Dynamics of zeroes under repeated differentiation.
\newblock \emph{Experimental Mathematics}, pages 1--27, 2021.

\bibitem[Hotta and Katori(2018)]{hotta_katori}
Ikkei Hotta and Makoto Katori.
\newblock Hydrodynamic limit of multiple {SLE}.
\newblock \emph{J. Stat. Phys.}, 171\penalty0 (1):\penalty0 166--188, 2018.
\newblock URL \url{https://doi.org/10.1007/s10955-018-1996-y}.

\bibitem[Hotta and Schlei{\ss}inger(2021)]{hotta_schleissinger}
Ikkei Hotta and Sebastian Schlei{\ss}inger.
\newblock Limits of radial multiple {SLE} and a {B}urgers-{L}oewner
  differential equation.
\newblock \emph{J. Theoret. Probab.}, 34\penalty0 (2):\penalty0 755--783, 2021.
\newblock URL \url{https://doi.org/10.1007/s10959-020-00996-0}.

\bibitem[Jalowy(2021)]{jalowy}
Jonas Jalowy.
\newblock Rate of convergence for products of independent non-{H}ermitian
  random matrices.
\newblock \emph{Electronic Journal of Probability}, 26:\penalty0 1--24, 2021.

\bibitem[Jordan et~al.(1998)Jordan, Kinderlehrer, and Otto]{JKO}
Richard Jordan, David Kinderlehrer, and Felix Otto.
\newblock The variational formulation of the fokker--planck equation.
\newblock \emph{SIAM journal on mathematical analysis}, 29\penalty0
  (1):\penalty0 1--17, 1998.

\bibitem[Kabluchko(2015)]{K15}
Zakhar Kabluchko.
\newblock Critical points of random polynomials with independent identically
  distributed roots.
\newblock \emph{Proceedings of the American Mathematical Society}, 143\penalty0
  (2):\penalty0 695--702, 2015.

\bibitem[Kabluchko(2022)]{ZakharLeeYang}
Zakhar Kabluchko.
\newblock Lee-{Y}ang zeroes of the {C}urie-{W}eiss ferromagnet, unitary
  {H}ermite polynomials, and the backward heat flow.
\newblock \emph{arXiv preprint arXiv:2203.05533}, 2022.

\bibitem[Kabluchko and Zaporozhets(2014)]{KZ14}
Zakhar Kabluchko and Dmitry Zaporozhets.
\newblock Asymptotic distribution of complex zeros of random analytic
  functions.
\newblock \emph{The Annals of Probability}, 42\penalty0 (4):\penalty0
  1374--1395, 2014.

\bibitem[Kiselev and Tan(2022)]{KT22}
Alexander Kiselev and Changhui Tan.
\newblock The flow of polynomial roots under differentiation.
\newblock \emph{Annals of PDE}, 8\penalty0 (2):\penalty0 16, 2022.

\bibitem[Kolmogorov(1958)]{Kolm58}
Andr{\'e} Kolmogorov.
\newblock Sur les propri{\'e}t{\'e}s des fonctions de concentrations de mp
  {L}{\'e}vy.
\newblock In \emph{Annales de l'institut Henri Poincar{\'e}}, volume~16, pages
  27--34, 1958.

\bibitem[Kopel et~al.(2020)Kopel, O’Rourke, and Vu]{kopel}
Phil Kopel, Sean O’Rourke, and Van Vu.
\newblock Random matrix products: Universality and least singular values.
\newblock \emph{The Annals of Probability}, 48\penalty0 (3):\penalty0
  1372--1410, 2020.

\bibitem[Krishnapur(2009)]{krishnapur2009random}
Manjunath Krishnapur.
\newblock From random matrices to random analytic functions.
\newblock \emph{Ann. Probab.}, 37\penalty0 (1):\penalty0 314--346, 2009.

\bibitem[Littlewood and Offord(1938)]{littlewood_offord1}
J.E. Littlewood and A.C. Offord.
\newblock {On the number of real roots of a random algebraic equation.}
\newblock \emph{J. Lond. Math. Soc.}, 13:\penalty0 288--295, 1938.

\bibitem[Littlewood and Offord(1939)]{littlewood_offord2}
J.E. Littlewood and A.C. Offord.
\newblock {On the number of real roots of a random algebraic equation. II.}
\newblock \emph{Proc. Camb. Philos. Soc.}, 35:\penalty0 133--148, 1939.

\bibitem[Lott(2008)]{lott}
John Lott.
\newblock Some geometric calculations on {W}asserstein space.
\newblock \emph{Comm. Math. Phys.}, 277\penalty0 (2):\penalty0 423--437, 2008.
\newblock URL \url{https://doi.org/10.1007/s00220-007-0367-3}.

\bibitem[Marcus(2021)]{Marcus21}
Adam~W Marcus.
\newblock Polynomial convolutions and (finite) free probability.
\newblock \emph{arXiv preprint arXiv:2108.07054}, 2021.

\bibitem[Marcus(2022)]{MarcusFPP}
Adam~W Marcus.
\newblock Finite free point processes.
\newblock \emph{arXiv preprint arXiv:2205.00495}, 2022.

\bibitem[Marden(1966)]{Marden}
Morris Marden.
\newblock \emph{Geometry of polynomials}, volume No. 3 of \emph{Mathematical
  Surveys}.
\newblock American Mathematical Society, Providence, RI, second edition, 1966.

\bibitem[Mehta(2004)]{Mehta}
Madan~Lal Mehta.
\newblock \emph{Random matrices}.
\newblock Elsevier, 2004.

\bibitem[Menon(2012)]{menon_miracles_burgers}
G.~Menon.
\newblock Lesser known miracles of {B}urgers equation.
\newblock \emph{Acta Math. Sci. Ser. B (Engl. Ed.)}, 32\penalty0 (1):\penalty0
  281--294, 2012.
\newblock ISSN 0252-9602.
\newblock \doi{10.1016/S0252-9602(12)60017-4}.

\bibitem[Menon(2017)]{menon_complex_burgers}
G.~Menon.
\newblock The complex {B}urgers equation, the {HCIZ} integral and the
  {C}alogero-{M}oser system, 2017.
\newblock URL \url{http://www.dam.brown.edu/people/menon/talks/cmsa.pdf}.

\bibitem[Michelen and Vu(2022)]{MV22}
Marcus Michelen and Xuan-Truong Vu.
\newblock Zeros of a growing number of derivatives of random polynomials with
  independent roots.
\newblock \emph{arXiv preprint arXiv:2212.11867}, 2022.

\bibitem[Michelen and Vu(2023)]{MV23}
Marcus Michelen and Xuan-Truong Vu.
\newblock Almost sure behavior of the zeros of iterated derivatives of random
  polynomials.
\newblock \emph{arXiv preprint arXiv:2307.06788}, 2023.

\bibitem[Mirabelli(2021)]{Mirabelli}
Benjamin Benno~Pine Mirabelli.
\newblock \emph{Hermitian, Non-Hermitian and Multivariate Finite Free
  Probability}.
\newblock PhD thesis, Princeton University, 2021.

\bibitem[Naumov(2012)]{naumov}
Alexey Naumov.
\newblock Elliptic law for real random matrices.
\newblock \emph{arXiv preprint arXiv:1201.1639}, 2012.

\bibitem[Nemish(2017)]{nemish}
Yuriy Nemish.
\newblock Local law for the product of independent non-{H}ermitian random
  matrices with independent entries.
\newblock \emph{Electronic Journal of Probability}, 22:\penalty0 1--35, 2017.

\bibitem[Nguyen and O’Rourke(2015)]{nguyenorourke}
Hoi~H Nguyen and Sean O’Rourke.
\newblock The elliptic law.
\newblock \emph{International Mathematics Research Notices}, 2015\penalty0
  (17):\penalty0 7620--7689, 2015.

\bibitem[Nica and Speicher(1997)]{nicaspeicher}
Alexandru Nica and Roland Speicher.
\newblock {$R$}-diagonal pairs---a common approach to {H}aar unitaries and
  circular elements.
\newblock In \emph{Free probability theory ({W}aterloo, {ON}, 1995)}, volume~12
  of \emph{Fields Inst. Commun.}, pages 149--188. Amer. Math. Soc., Providence,
  RI, 1997.
\newblock ISBN 0-8218-0675-0.

\bibitem[Nica and Speicher(2006)]{nicaspeicherbook}
Alexandru Nica and Roland Speicher.
\newblock \emph{Lectures on the combinatorics of free probability}, volume 335
  of \emph{London Mathematical Society Lecture Note Series}.
\newblock Cambridge University Press, Cambridge, 2006.
\newblock ISBN 978-0-521-85852-6; 0-521-85852-6.
\newblock \doi{10.1017/CBO9780511735127}.
\newblock URL \url{https://doi.org/10.1017/CBO9780511735127}.

\bibitem[Olver(2014)]{olver_book_PDE}
Peter~J. Olver.
\newblock \emph{Introduction to partial differential equations}.
\newblock Undergraduate Texts in Mathematics. Springer, Cham, 2014.
\newblock URL \url{https://doi.org/10.1007/978-3-319-02099-0}.

\bibitem[O'Rourke(2016)]{O16}
Sean O'Rourke.
\newblock Critical points of random polynomials and characteristic polynomials
  of random matrices.
\newblock \emph{International Mathematics Research Notices}, 2016\penalty0
  (18):\penalty0 5616--5651, 2016.

\bibitem[Otto(2001)]{Otto}
Felix Otto.
\newblock The geometry of dissipative evolution equations: the porous medium
  equation.
\newblock \emph{Comm. Partial Differential Equations}, 26\penalty0
  (1-2):\penalty0 101--174, 2001.
\newblock ISSN 0360-5302,1532-4133.
\newblock \doi{10.1081/PDE-100002243}.
\newblock URL \url{https://doi.org/10.1081/PDE-100002243}.

\bibitem[Otto and Villani(2000)]{OttoVillani}
Felix Otto and C{\'e}dric Villani.
\newblock Generalization of an inequality by {T}alagrand and links with the
  logarithmic {S}obolev inequality.
\newblock \emph{Journal of Functional Analysis}, 173\penalty0 (2):\penalty0
  361--400, 2000.

\bibitem[O’Rourke and Steinerberger(2021)]{OSteiner}
Sean O’Rourke and Stefan Steinerberger.
\newblock A nonlocal transport equation modeling complex roots of polynomials
  under differentiation.
\newblock \emph{Proceedings of the American Mathematical Society}, 149\penalty0
  (4):\penalty0 1581--1592, 2021.

\bibitem[O’Rourke and Williams(2019)]{OW19}
Sean O’Rourke and Noah Williams.
\newblock Pairing between zeros and critical points of random polynomials with
  independent roots.
\newblock \emph{Transactions of the American Mathematical Society},
  371\penalty0 (4):\penalty0 2343--2381, 2019.

\bibitem[Pemantle and Rivin(2013)]{PR13}
Robin Pemantle and Igor Rivin.
\newblock The distribution of zeros of the derivative of a random polynomial.
\newblock In \emph{Advances in Combinatorics: Waterloo Workshop in Computer
  Algebra, W80, May 26-29, 2011}, pages 259--273. Springer, 2013.

\bibitem[Poplavskyi and Schehr(2018)]{poplavskyi2018exact}
Mihail Poplavskyi and Gr{\'e}gory Schehr.
\newblock Exact persistence exponent for the 2{D}-diffusion equation and
  related {K}ac polynomials.
\newblock \emph{Physical review letters}, 121\penalty0 (15):\penalty0 150601,
  2018.

\bibitem[Ransford(1995)]{ransford}
Thomas Ransford.
\newblock \emph{Potential theory in the complex plane}.
\newblock Cambridge university press, 1995.

\bibitem[Rodgers and Tao(2020)]{rodgers_tao}
B.~Rodgers and T.~Tao.
\newblock The de {B}ruijn--{N}ewman constant is non-negative.
\newblock \emph{Forum Math. Pi}, 8:\penalty0 e6, 62, 2020.
\newblock URL \url{https://doi.org/10.1017/fmp.2020.6}.

\bibitem[Rogers and Shi(1993)]{RogersShi}
Leonard~CG Rogers and Zhan Shi.
\newblock Interacting {B}rownian particles and the {W}igner law.
\newblock \emph{Probability theory and related fields}, 95\penalty0
  (4):\penalty0 555--570, 1993.

\bibitem[Saff and Totik(1997)]{saff_totik_book}
E.~B. Saff and V.~Totik.
\newblock \emph{Logarithmic potentials with external fields}, volume 316 of
  \emph{Grundlehren der mathematischen Wissenschaften [Fundamental Principles
  of Mathematical Sciences]}.
\newblock Springer-Verlag, Berlin, 1997.
\newblock URL \url{https://doi.org/10.1007/978-3-662-03329-6}.
\newblock Appendix B by Thomas Bloom.

\bibitem[Santambrogio(2015)]{Santa}
Filippo Santambrogio.
\newblock Optimal transport for applied mathematicians.
\newblock \emph{Birk{\"a}user, NY}, 55\penalty0 (58-63):\penalty0 94, 2015.

\bibitem[She et~al.(1992)She, Aurell, and Frisch]{she_aurell_frisch}
Zhen-Su She, Erik Aurell, and Uriel Frisch.
\newblock The inviscid {B}urgers equation with initial data of {B}rownian type.
\newblock \emph{Comm. Math. Phys.}, 148\penalty0 (3):\penalty0 623--641, 1992.
\newblock URL \url{http://projecteuclid.org/euclid.cmp/1104251047}.

\bibitem[\'{S}niady(2002)]{Sniady}
Piotr \'{S}niady.
\newblock Random regularization of {B}rown spectral measure.
\newblock \emph{J. Funct. Anal.}, 193\penalty0 (2):\penalty0 291--313, 2002.
\newblock ISSN 0022-1236,1096-0783.
\newblock \doi{10.1006/jfan.2001.3935}.
\newblock URL \url{https://doi.org/10.1006/jfan.2001.3935}.

\bibitem[Steinerberger(2021)]{Steiner21}
Stefan Steinerberger.
\newblock Free convolution powers via roots of polynomials.
\newblock \emph{Experimental Mathematics}, pages 1--6, 2021.

\bibitem[Subramanian(2012)]{Sub12}
Sneha Subramanian.
\newblock {On the distribution of critical points of a polynomial}.
\newblock \emph{Electronic Communications in Probability}, 17:\penalty0 1 -- 9,
  2012.

\bibitem[Szeg\H{o}(1975)]{szegoe_book}
G.~Szeg\H{o}.
\newblock \emph{Orthogonal polynomials}.
\newblock American Mathematical Society Colloquium Publications, Vol. XXIII.
  American Mathematical Society, Providence, R.I., fourth edition, 1975.

\bibitem[Tao(2017)]{tao_blog1}
T.~Tao.
\newblock Heat flow and zeroes of polynomials.
\newblock
  \url{https://terrytao.wordpress.com/2017/10/17/heat-flow-and-zeroes-of-polynomials/},
  2017.

\bibitem[Tao(2018)]{tao_blog2}
T.~Tao.
\newblock Heat flow and zeroes of polynomials {II}.
\newblock
  \url{https://terrytao.wordpress.com/2018/06/07/heat-flow-and-zeroes-of-polynomials-ii-zeroes-on-a-circle/},
  2018.

\bibitem[Tao and Vu(2010)]{tao_vu_circular}
T.~Tao and V.~Vu.
\newblock Random matrices: universality of {ESD}s and the circular law.
\newblock \emph{Ann. Probab.}, 38\penalty0 (5):\penalty0 2023--2065, 2010.
\newblock URL \url{https://doi.org/10.1214/10-AOP534}.
\newblock With an appendix by Manjunath Krishnapur.

\bibitem[Tran(2021)]{Tran}
Hung~V Tran.
\newblock \emph{Hamilton--Jacobi equations: theory and applications}, volume
  213.
\newblock American Mathematical Soc., 2021.

\bibitem[Uchiyama(2000)]{Uchiyama}
K{\^o}hei Uchiyama.
\newblock A non-linear evolution equation driven by logarithmic potential.
\newblock \emph{Bulletin of the London Mathematical Society}, 32\penalty0
  (3):\penalty0 353--363, 2000.

\bibitem[Van~Assche(1985)]{van_assche_some_results}
W.~Van~Assche.
\newblock Some results on the asymptotic distribution of the zeros of
  orthogonal polynomials.
\newblock In \emph{Proceedings of the international conference on computational
  and applied mathematics ({L}euven, 1984)}, volume 12/13, pages 615--623,
  1985.
\newblock URL \url{https://doi.org/10.1016/0377-0427(85)90053-6}.

\bibitem[Vergassola et~al.(1994)Vergassola, Dubrulle, Frisch, and
  Noullez]{vergassola}
M.~Vergassola, B.~Dubrulle, U.~Frisch, and A~Noullez.
\newblock {Burgers' equation, devil's staircases and the mass distribution for
  large-scale structures}.
\newblock \emph{Astronomy and Astrophysics}, 289\penalty0 (2):\penalty0
  325--356, 1994.

\bibitem[Villani(2009)]{Villani}
C{\'e}dric Villani.
\newblock \emph{Optimal transport: old and new}, volume 338.
\newblock Springer, 2009.

\bibitem[Voiculescu et~al.(1992)Voiculescu, Dykema, and
  Nica]{voiculescu_nica_dykema_book}
D.~V. Voiculescu, K.~J. Dykema, and A.~Nica.
\newblock \emph{Free random variables}, volume~1 of \emph{CRM Monograph
  Series}.
\newblock American Mathematical Society, Providence, RI, 1992.
\newblock URL \url{https://doi.org/10.1090/crmm/001}.

\bibitem[Voiculescu(1993)]{Voiculescu93}
Dan Voiculescu.
\newblock The analogues of entropy and of {F}isher's information measure in
  free probability theory, i.
\newblock \emph{Communications in mathematical physics}, 155\penalty0
  (1):\penalty0 71--92, 1993.

\bibitem[Voit and Woerner(2022)]{VW22}
Michael Voit and Jeannette~HC Woerner.
\newblock Limit theorems for {B}essel and {D}unkl processes of large dimensions
  and free convolutions.
\newblock \emph{Stochastic Processes and their Applications}, 143:\penalty0
  207--253, 2022.

\bibitem[Wigner(1955)]{wigner55}
Eugene~P. Wigner.
\newblock Characteristic vectors of bordered matrices with infinite dimensions.
\newblock \emph{Annals of Mathematics}, 62\penalty0 (3):\penalty0 548--564,
  1955.

\bibitem[Wigner(1958)]{wigner58}
Eugene~P Wigner.
\newblock On the distribution of the roots of certain symmetric matrices.
\newblock \emph{Annals of Mathematics}, pages 325--327, 1958.

\bibitem[Zhong(2022)]{Zhong}
Ping Zhong.
\newblock Brown measure of the sum of an elliptic operator and a free random
  variable in a finite von {N}eumann algebra.
\newblock \emph{arXiv preprint 2108.09844v4}, 2022.

\end{thebibliography}
\bibliographystyle{plainnat}

\end{document}